\documentclass[12pt, reqno]{amsart}
\usepackage{ifthen,amsfonts,
amsmath,amssymb, enumerate,
epsfig,color,graphicx,cite,bbm,graphics}
\usepackage{stmaryrd}
\usepackage{caption}
\usepackage{a4wide}
\usepackage{xfrac}
\usepackage{mathrsfs,bm,amsthm,mathtools,yfonts}
\usepackage{xcolor}

\makeatletter

\def\fnum{equation}
\def\fnum{theorem}

\begingroup

\newtheorem{theorem}{Theorem}[section]
\newtheorem{lemma}[theorem]{Lemma}
\newtheorem{propos}[theorem]{Proposition}
\newtheorem{corol}[theorem]{Corollary}
\endgroup

\newtheorem{assumption}[theorem]{Assumption}
\newtheorem{definition}[theorem]{Definition}

\newtheorem{remark}[theorem]{Remark}

\numberwithin{equation}{section}

\newcommand{\M}{{\mathcal{M}}}

\newcommand{\diam}{{\text {diam}}}
\newcommand{\dist}{{\text {dist}}}

\newcommand{\expo}{\text{exp}}
\newcommand{\supp}{{\rm spt}}

\newcommand\ress{\mathop{\hbox{\vrule height 14pt width 1pt depth 0pt
\vrule height 1pt width 12pt depth 0pt}}\nolimits}
\newcommand\res{\mathop{\hbox{\vrule height 7pt width .5pt depth 0pt
\vrule height .5pt width 6pt depth 0pt}}\nolimits}
\newcommand\spt{{\rm spt}\,}  
\newcommand\weaks{{\rightharpoonup^*}\,}
\newcommand\haus{{\mathcal{H}}}
\newcommand\va{{\mathcal{V}}}
\newcommand\sing{{\rm sing}\,}
\newcommand\Tan{{\rm Tan}\,}
\newcommand{\mass}{\mathbf{M}}
\newcommand{\fl}{\mathbbm{F}}

\def\RR{{\mathbb{R}}}
\def\ZZ{{\bf  Z}}

\def\NN{{\mathbb{N}}}

\renewcommand\P{{\mathcal{P}}}

\begin{document}

\title[Minimal surfaces with boundary]
{Min-max theory for minimal hypersurfaces with boundary}

\author{Camillo De Lellis}%
\author{Jusuf Ramic}%
\address{Zurich university}

\email{camillo.delellis@math.uzh.ch and 
jusuf.ramic@math.uzh.ch}

\begin{abstract}
In this note we propose a min-max theory for embedded hypersurfaces with a fixed boundary and apply it to prove several theorems about the existence of embedded minimal hypersurfaces with a given boundary.  A simpler variant of these theorems holds also for the case of the free boundary minimal surfaces. 
\end{abstract}

\maketitle
 
\section{Introduction}

The primary interest of this note is the following question, which has been posed to us by White:
\begin{itemize}
\item[(Q)] consider two minimal strictly stable embedded hypersurfaces $\Sigma_0$ and $\Sigma_1$ in $\mathbb R^{n+1}$ which have the same boundary $\gamma$; is there a third smooth minimal hypersurface $\Gamma_2$ with the same boundary?
\end{itemize}
In this note we anser positively, under some suitable assumptions, cf. Corollary \ref{c:main2}: the main one is that $\gamma$ lies in the boundary of a smooth strictly convex subset of $\mathbb R^{n+1}$, but there are also two other assumptions of technical nature, which we discuss at the end of the next section. 

If we regard $\Sigma_0$ and $\Sigma_1$ as two local minima of the Plateau's problem, we expect that a ``mountain pass lemma'' type argument yields a third minimal surface $\Gamma_2$ which is a saddle point. This intuition has a long history, which for closed geodesics goes back to the pioneering works of Birkhoff and Ljusternik and Fet (cf. \cite{Birkhoff, LF}). In the case $n=2$ the first results regarding the Plateau's problem are due to Shiffman \cite{Sh} and Morse and Tompkins \cite{MT}, using the parametric approach of Douglas-Rado, which therefore answers positively to question (Q) if we drop the requirements of embeddedness of the surfaces. 

Using a degree theory approach to the Plateau problem, Tromba \cite{Tromba1,Tromba2}
was able to derive a limited Morse theory for disk-type surfaces, generalizing
the Morse-Shiffman-Tompkins result. M. Struwe \cite{Struwe1,Struwe2,Struwe3} then developed a
general Morse theory for minimal surfaces of disk and annulus type, based on
the $H^{1,2}$ topology (as opposed to the $C^0$
topology used in Morse-Shiffman-Tompkins
approach). These (and other related works) were expanded by Jost and Struwe in
\cite{JS}, where they consider minimal surfaces of arbitrary topological type. Among
other things, they succeed in applying saddle-point methods to prove the existence
of unstable minimal surfaces of prescribed genus.

\medskip

It is well-known that the parametric approach breaks down when $n\geq 3$. 
For this reason in their celebrated pioneering works \cite{Alm,Alm1,pitts} Almgren and subsequently Pitts developed a variational calculus in the large using geometric measure theory. The latter enabled Pitts to prove in \cite{pitts} the existence of a nontrivial closed embedded minimal hypersurface in any closed Riemannian manifold of dimension at most 6. The higher dimensional case of Pitts' theorem was then settled by Schoen and Simon in \cite{SS}. A variant of the Almgren-Pitts theory was later introduced by Simon and Smith (cf. Smith's PhD thesis \cite{smith}) in the $2$-dimensional case, with the aim of producing a min-max construction which allows to control the topology of the final minimal surface. Further investigations in that direction were then set forth by Pitts and Rubinstein in \cite{PR1,PR2}, who also proposed several striking potential applications to the topology of $3$-manifolds. Part of this program was carried 
out later in the papers \cite{c-dl}, \cite{dl-p}, \cite{ketover} and \cite{cgk}, whereas several other questions raised in \cite{PR1,PR2} constitute an active area of research right now (see for instance \cite{MNindex}).  Currently, min-max constructions have gained a renewed interest thanks to the celebrated recent work of Marques and Neves \cite{MN} which uses the Almgren-Pitts machinery to prove a long-standing conjecture
of Willmore in differential geometry (cf. also the paper \cite{AMN}, where the authors use similar ideas to prove a conjecture of Freedman, He and Wang in knot theory).

All the literature mentioned above regards closed hypersurfaces, namely without boundary. The aim of this paper is to provide a similar framework in order to attack analogous existence problems in the case of prescribed boundaries. The real issue is only in the boundary regularity of the final min-max surface, as the interior regularity and much of the existence setting can be extended without problems. When the boundary $\gamma$ is lying in the boundary of a strictly convex set, our theorems give full regularity of the min-max hypersurface in any dimension. The most relevant difficulty in the proof can be overcome thanks to a beautiful idea of Brian White, whom we thank very warmly for sharing it with us, cf. \cite{white-personal}: the elegant argument of White to get curvature estimates at the boundary is reported in Section \ref{ss:white_est}. In passing, since many of the techniques are essentially the same, we also handle the case of free boundaries, which for $n=2$ has already been considered by Gr\"uter and Jost in \cite{gru-jost} and by Li in \cite{Li} (modifying the Simon-Smith approach and reaching stronger conclusions). Slightly after completing this work we learned that a more general approach to
the existence of minimal hypersurfaces with free boundary has been taken at the same time by Li and Zhou in \cite{LiZhou1,LiZhou2}. Note in particular that in \cite{LiZhou2} the authors are able to prove their existence even without the assumption of convexity of the boundary. 

Our min-max constructions use the simpler and less technical framework proposed by the first author and Tasnady in \cite{dl-t}, which is essentially a variant of that developed by Pitts in his groundbreaking monograph \cite{pitts}. It allows us to avoid a lot of technical details and yet be sufficiently self-contained, but on the other hand there are certain limitations which Pitts' theory does not have. 
However, several of the tools developed in this paper can be applied to a suitable ``boundary version'' of Pitts' theory as well and we believe that the same statements can be proved in that framework. Similarly, we do not expect any problems in using the same ideas to extend the approach of Simon and Smith when the boundary is prescribed. 

\section{Main statements}
 
Consider a smooth, compact, oriented Riemannian manifold $(\M,g)$ of dimension $n+1$
with boundary $\partial \M$. We will assume that $\partial \M$ is strictly uniformly convex, namely:
\begin{assumption}\label{a:(C)}
The principal curvatures of $\partial \M$ with respect to the unit normal $\nu$ pointing inside $\M$ have a uniform, positive lower bound.
\end{assumption}
Sometimes we write the condition above as $A_{\partial \M} \succeq \xi g$, where $\xi >0$, $A_{\partial \M}$ denotes the second fundamental form of $\partial \M$ (with the choice of inward pointing normal) and $g$ the induced metric as submanifold of $\M$. 
We also note that we do not really need $\M$ to be $C^\infty$ since a limited amount of regularity (for instance $C^{2,\alpha}$ for some positive $\alpha$) suffices for all our considerations, although we will not pay any attention to this detail. 

We start by recalling the continuous families of hypersurfaces used in \cite{dl-t}. 

\begin{definition}
We fix a smooth compact $k$-dimensional manifold $\P$ with boundary $\partial \P$ (possibly empty) and we will call it the 
{\em space of parameters}.

A smooth family of hypersurfaces in $\M$ parametrized by $\P$ is given by a map $t\mapsto \Gamma_t$ which assigns to each $t \in \P$ a closed subset $\Gamma_t$ of $\M$ and satisfies  the following properties:
\begin{itemize}
\item[(s1)] For each $t$ there is a finite $S_t\subset \M$ such that 
$\Gamma_t$ is a smooth 
{\em oriented} hypersurface in $\M\setminus S_t$ with boundary $\partial \Gamma_t \subset \partial \M \setminus S_t$;
\item[(s2)] $\haus^n(\Gamma_t)$ is continuous in $t$ and
$t\mapsto \Gamma_t$ is continuous in the following sense: for all $t$ and all open $U\supset \Gamma_t$ there is $\varepsilon>0$ such that 
\[
\Gamma_s\subset U\quad \mbox{for all $s$ with $|t-s|< \varepsilon$;}
\]
\item[(s3)] on any $U\subset\subset \M\setminus S_{t_0}$,  
$\Gamma_t \stackrel{t\rightarrow t_0}{\longrightarrow} \Gamma_{t_0}$ 
smoothly in $U$. 
\end{itemize}
\end{definition}

\begin{remark}\label{r:differente_1}
There is indeed an important yet subtle difference between the above definition and the corresponding one used in \cite{dl-t}. In the latter reference the families of hypersurfaces $\{\Gamma_t\}_t$ are also assumed to have underlying families of open sets $\{\Omega_t\}_t$ which vary continuously (more precisely $t\mapsto {\mathbf 1}_{\Omega_t}$ is continuous in $L^1$) and such that $\partial \Omega_t = \Sigma_t$. For this reason, in a lot of considerations where \cite{dl-t} invokes the theory of Caccioppoli sets, we will need to consider more general integral currents, allowing for multiplicities higher than $1$; cf. Remark \ref{r:differente_2} and Section \ref{s:differente_3}. 
\end{remark}

From now on we will simply refer to such objects as {\em families parametrized by $\P$} and we will omit to mention the space of parameters when this is obvious from the context. 
Additionally we will distinguish between two classes of smooth families according to their behaviour at the boundary $\partial \M$.

\begin{definition}\label{homotopies}
Consider a smooth, closed submanifold $\gamma\subset \partial \M$ of dimension $n-1$. A smooth family of
hypersurfaces parametrized by $\P$ is constrained by $\gamma$ if
$\partial \Gamma_t= \gamma \setminus S_t$ for every $t\in \P$. Otherwise we talk about ``uncostrained families''.

Two unconstrained families $\{\Gamma_t\}$ and $\{\Sigma_t\}$ parametrized by $\P$ are homotopic if there is a family $\{\Lambda_{t,s}\}$ parametrized by $\P\times [0,1]$ such that 
\begin{itemize}
\item $\Lambda_{t,0}= \Sigma_t$ $\forall t\in \P$,
\item $\Lambda_{t,1} = \Gamma_t$ $\forall t\in \P$,
\item and $\Lambda_{t,s} = \Lambda_{t,0}$ $\forall t\in \partial \P$ and for all $s\in [0,1]$. 
\end{itemize}
When the two families are constrained by $\gamma$ we then additionally require that the family $\{\Lambda_{t,s}\}$ is also constrained by $\gamma$. 

Finally a set $X$ of constrained (resp. unconstrained) families parametrized by the same $\P$ is called homotopically closed if $X$ includes the homotopy class of each of its elements.
\end{definition}

\begin{definition}
Let $X$ be a homotopically closed set of constrained (resp. unconstrained) families parametrized by the same $\P$. The min-max value of $X$, denoted by $m_0 (X)$ is the number
\begin{equation}\label{min energy}
m_0 (X) = \inf \left\{ \max_{t\in\P} \haus^n (\Sigma_t): \{\Sigma_t\}\in X \right\}\, .
\end{equation}
The boundary-max value of $X$ is instead
\begin{equation}
bM_0 (X) = \max_{t\in\partial \P} \left\{ \haus^n (\Sigma_t): \{\Sigma_t\}\in X \right\}\, .
\end{equation}
A minimizing sequence is given by a sequence of elements $\{\{\Sigma_t\}^\ell\} \subset X$ such that
\[
 \lim_{\ell\uparrow\infty} \max_{t\in\P} \haus^n (\Sigma^\ell_t) = m_0 (X).
\]
A min-max sequence is then obtained from a minimizing sequence by taking the slices $\{\Sigma^\ell_{t_\ell}\}$, for a choice of parameters $t_\ell\in \P$ such that $\haus^n (\Sigma^\ell_{t_\ell}) \to m_0 (X)$. 
\end{definition}

As it is well known, even the solutions of the codimension one Plateau problem can exhibit singularities if the dimension $n+1$ of the ambient manifold is strictly larger than $7$. If we say that an embedded minimal hypersurface $\Gamma$ is {\em smooth} then we understand that it has no singularities. Otherwise we denote by ${\rm Sing}\, (\Gamma)$ its closed singular set, i.e. the set of points 
where $\Gamma$ cannot be described locally as the graph of a smooth function. Such singular set will always have Hausdorff dimension at most $n-7$ and thus with a slight abuse of terminology we will anyway say that $\Gamma$ is embedded, although in a neighborhood of the singularities the surface might not be a continuous embedded submanifold. When we write ${\rm dim}\, ({\rm Sing}\, (\Gamma)) \leq n-7$ we then understand that the singular set is empty for $n\leq 6$. 

Our main theorem is the following.

\begin{theorem}\label{t:main}
Let $\M$ be a smooth Riemannian manifold that satisfies Assumption \ref{a:(C)} and
$X$ be a homotopically closed set of constrained (resp. unconstrained) families parametrized by $\P$ such that
\begin{equation}\label{mountain pass family}
m_0 (X) > bM_0 (X)\, .
\end{equation}
Then there is a min-max sequence $\{\Sigma^\ell_{t_\ell}\}$, finitely many disjoint  embedded and connected compact minimal hypersurfaces $\{\Gamma_1, \ldots, \Gamma_N\}$ with boundaries $\partial \Gamma_i\subset \partial \M$ (possibly empty) and finitely many positive integers $c_i$ such that
\[
\Sigma^\ell_{t_\ell}\qquad \weaks\qquad \sum_i c_i \Gamma_i
\]
in the sense of varifolds and ${\rm dim}\, ({\rm Sing}\, (\Gamma_i))\leq n-7$ for each $i$. In addition:
\begin{itemize}
\item[(a)] If $X$ consists of unconstrained families, then each $\Gamma_i$ meets $\partial \M$ orthogonally;
\item[(b)] If $X$ consists of families constrained by $\gamma$, then: $\sum \partial \Gamma_i = \gamma$,
${\rm Sing}\, (\Gamma_i) \cap \partial \M= \emptyset$ for each $i$ and $c_i=1$ whenever $\partial \Gamma_i \neq \emptyset$.  
\end{itemize}
\end{theorem}\

Our main concern is in fact the case (b), because the regularity at the boundary requires much more effort. The regularity at the boundary for the case (a) is instead much more similar to the usual interior regularity for minimal surfaces and for this reason we will not spend much time on it but rather sketch the needed changes in the arguments. As an application of the main theorem we give the following two interesting corollaries. 

\begin{corol}\label{c:main1}
Under the assumptions above there is always a nontrivial embedded minimal hypersurface $\Gamma$ in $\M$, meeting the boundary
$\partial \M$ orthogonally, with ${\rm dim}\, ({\rm Sing}\, (\Gamma)) \leq n-7$.
\end{corol}

Note that the corollary above does not necessarily imply that $\Gamma$ has nonempty boundary: we do not exclude that $\Gamma$ might be a closed minimal surface. On the other hand, if $\Gamma$ has nonempty boundary, then it is contained in $\partial \M$ and any connected component of $\Gamma$ is thus a nontrivial solution of the free boundary problem. Therefore the existence of such nontrivial solution is guaranteed by the following

\begin{assumption}\label{a:stronger}
$\M$ does not contain any nontrivial minimal closed hypersurface $\Sigma$ embedded and smooth except for a singular set ${\rm Sing}\, (\Sigma)$ with ${\rm dim}\, ({\rm Sing}\, (\Sigma))\leq n-7$. 
\end{assumption} 

Note that the property above holds if $\M$ satisfies some stronger convexity condition than Assumption \ref{a:(C)}: for instance if there is a point $p$ such that $\M\setminus \{p\}$ can be foliated with convex hypersurfaces, then it follows from the maximum principle. In particular both the Assumptions \ref{a:(C)} and \ref{a:stronger} are satisfied by any bounded convex subset of the Euclidean space, or by any ball of a closed Riemannian manifold with radius smaller than the convexity radius. 

Likewise, under the very same assumptions we can conclude the following Morse-theoretical result for the Plateau's problem. 

\begin{corol}\label{c:main2}
Let $\M$ be a smooth Riemannian manifold satisfying Assumptions \ref{a:(C)} and \ref{a:stronger} and let $\gamma\subset \partial \M$ be a smooth, oriented, closed $(n-1)$-dimensional submanifold. Assume further that:
\begin{itemize}
\item[(i)] there are two distinct  smooth, oriented, minimal embedded hypersurfaces $\Sigma_0$ and $\Sigma_1$ with $\partial \Sigma_0 = \partial \Sigma_1 = \gamma$ which are strictly stable, meet only at the boundary and bound some open domain $A$ (in particular $\Sigma_0$ and $\Sigma_1$ are homologous).
\end{itemize}
Then there exists a third distinct embedded minimal hypersurface $\Gamma_2$ with $\partial \Gamma_2 = \gamma$ such that ${\rm dim}\, ({\rm Sing}\, (\Gamma_2))\leq n-7$ and ${\rm Sing}\, (\Gamma_2) \cap \partial \M =\emptyset$. 
\end{corol}

The corollary above asks for two technical assumptions which are not really natural:
\begin{itemize}
\item $\Sigma_0$ and $\Sigma_1$ intersect only at the boundary;
\item they are regular {\em everywhere}.
\end{itemize}
We use both to give an elementary construction of a $1$-dimensional sweepout which ``connects'' $\Sigma_0$ and $\Sigma_1$ (i.e. a one-parameter family 
$\{\Sigma_t\}_{t\in [0,1]}$), but by taking advantage of more avdanced techniques in geometric measure theory and algebraic topology - as for instance Pitts' approach via discretized faimly of currents - it should suffice to assume that $\Sigma_0$ and $\Sigma_1$ are homologous and that the dimensions of their singular sets do not exceed $n-7$.

The smoothness  enters however more crucially in showing that any sweepout connecting 
$\Sigma_0$ and $\Sigma_1$ must have a ``slice'' with $n$-dimensional volume larger than $\max \{\haus^n (\Sigma_0), \haus^n (\Sigma_1)\}$.
It is needed to take advantage of an argument of White \cite{white-strong}, where regularity is a key ingredient. The following local minimality property could replace strict stability and smoothness:
\begin{itemize}
\item For each $i\in \{0,1\}$ there is a $\varepsilon >0$ such that any current $\Gamma$ with boundary $\gamma$ which is distinct from $\Sigma_i$ and at flat distance smaller than $\varepsilon$ from $\Sigma_i$ has mass strictly larger than that of $\Sigma_i$. 
\end{itemize}

\subsection{Acknowledgements} Both authors acknowledge the support of the Swiss National Foundation. Moreover, they wish to express their gratitude to the anonymous referee who has pointed out several mistakes in an earlier version of the paper. 

\section{Notation and outline of the paper} 

\subsection{Outline of the paper}  The Corollaries \ref{c:main1} and \ref{c:main2} will be shown in the very last section of the paper (cf. Section \ref{s:final}). The remaining part will instead be entirely devoted to prove Theorem \ref{t:main}.

First of all in the Section \ref{s:pull-tight} we will introduce two adapted classes of stationary varifolds
for the constrained and unconstrained case, which are a simple variants of the usual notion of stationary varifold introduced by
Almgren. Then in Proposition \ref{p:pull-tight} we prove the existence of a suitable sequence of families $\{\{\Sigma_t\}^\ell\}$ in $X$ with the property that each min-max sequence generated by it converges to a stationary varifold: the argument is a straightforward adaptation of Almgren's pull-tight procedure used in Pitts' book and in several other later references (indeed we follow the presentation in \cite{c-dl})).

In the sections \ref{s:AM} and \ref{s:comb} we adapt the notion of almost minimizing surfaces used in \cite{dl-t} to the case at the boundary and we ultimately prove the existence of a min-max sequence which is almost minimizing in any sufficiently small annulus centered at any given point, cf. Proposition \ref{a.m. in annuli}. The arguments follow closely those used by Pitts in \cite{pitts} and a trick introduced in \cite{dl-t} to avoid Pitts' discretized families. 
The min-max sequence generated in Proposition \ref{a.m. in annuli} is the one for which we will conclude the properties claimed in Theorem \ref{t:main}. Indeed the interior regularity follows from the arguments of Pitts (with a suitable adaptation by Schoen and Simon to the case $n\geq 6$) and we refer to \cite{dl-t} for the details. The remaining sections are thus devoted to the boundary regularity. 

First of all in Section \ref{s:boundary} we collect several tools about the boundary behavior of stationary varifolds (such as the monotonicity formulae in both the constrained and uncostrained case and a useful maximum principle in the constrained one), but more importantly, we will use a very recent argument of White to conclude suitable curvature estimates at the boundary in the constrained case, under the assumption that the minimal surface meets $\partial \M$ transversally in a suitable (quantified) sense, cf. Theorem \ref{t:boundary_est}. 

In Section \ref{s:stability} we recall the celebrated Schoen-Simon compactness theorem for stable minimal hypersurfaces in the interior and 
its variant by Gr\"uter and Jost in the free boundary case. Moreover, we combine the Schoen-Simon theorem with Theorem \ref{t:boundary_est} to conclude a version of the Schoen-Simon compactness theorem for stable hypersurfaces up to the boundary, when the latter is a fixed given smooth $\gamma$ and the surfaces meet $\partial \M$ transversally. 
In Section \ref{s:replacements} we modify the proof in \cite{dl-t} to construct replacements for almost minimizing varifolds. The main difficulty and contribution here is to preserve the boundary conditions for the surfaces in the constrained case, throughout the various steps of the construction. Following the arguments in \cite{dl-t}, we analogously define the $(2^{m+2}j)^{-1}$ - homotopic Platau problem for $j \in \NN$, and we conclude that in sufficiently small balls, the corresponding minimizers are actually minimizing for the (usual) Plateau problem. Hence their regularity (with no singular points!) at the boundary will follow from Allard's boundary regularity in \cite{allard-bdry}. Finally the tools of Section \ref{s:stability} and Section \ref{s:replacements} are used in Section \ref{s:bdry_reg} to conclude the boundary regularity of the minmax surface and hence complete the proof of Theorem \ref{t:main}.

\subsection{Notation} Since we are always dealing with manifolds $\M$ which have a nonempty boundary, as it is customary an open subset $U$ of $\M$ can contain a portion of $\partial \M$. For instance, if $\M$ is the closed unit ball in $\mathbb R^{n+2}$, $P$ the north pole $(0,0,\ldots, 1)$ and $\tilde{U}$ a neighborhood of $P$ in $\mathbb R^{n+2}$, then $\tilde U\cap \M$ is, in the relative topology, 
an open subset of $\M$.  Hence, although we will denote by ${\rm Int}\, (\M)$ the set $\M\setminus \partial \M$, the latter \emph{is not} the topological interior of $\M$ and our notation is slightly abusive.
In the following table we present notations, definitions and conventions used
consistently throughout the paper:\\
\begin{center}
\begin{tabular}{ l@{\hskip 0.5in} l }
 $B_\rho(x),\,\overline{B}_\rho(x)$ & open and closed geodesic balll of radius $\rho$ and center $x$ in $\M$;  \\ 
$\partial B_\rho(x)$ & geodesic sphere of radius $\rho$ and center $x$ in $\M$\\
 ${\rm Int}\, (U)$ & ``interior'' of the open set $U$, namely $U\setminus \partial \M$;\\
 Inj($\M$) & injectivity radius of $\M$;\\
 $\text{An}(x,\tau,t)$ & open annulus $B_t(x)\setminus \overline{B}_{\tau}(x)$;\\
 $\mathcal{AN}_{r}(x)$ & the set $\{\text{An}(x,\tau,t) \text{ with } 0<\tau<t<r\}$;\\
 diam($G$) & diameter of a subset $G \subset \M$;
\end{tabular}
\end{center}
\begin{center}
\begin{tabular}{ l@{\hskip 0.5in} l }
 $\haus^k$ & $k$-dim Hausdorff measure in $\M$;\\
 $\omega_k$ & volume of the unit ball in $\RR^k$;\\
$\nu$ & unit normal to $\partial\M$, pointing {\em inwards}\\
$\supp$ & support (of a function, vector field, varifold, current, etc.);\\
$\mathfrak{X}_c (U)$ & smooth vector fields $\chi$ with $\supp (\chi)\subset U$\\
& (note that such $\chi$ do not necessarily vanish on $\partial \M$!);\\
 $\mathfrak{X}_c^0(U)$ & $\chi\in \mathfrak{X}_c (U)$ which vanish on $\partial \M$;\\
 $\mathfrak{X}^t_c(U)$ & $\chi\in \mathfrak{X}_c (U)$ tangent to $\partial \M$, i.e. $\chi\cdot \nu =0$;\\ 
$\mathfrak{X}^-_c (U)$ & $\chi\in\mathfrak{X}_c (U)$ pointing inwards at $\partial \M$ (i.e. $\chi \cdot \nu \geq 0$);\\
 $\va^k(U), \va(U)$ & vector space of $k$-varifolds in $U$;\\
 $G^k(U),G(U)$ & Grassmanian bundle of unoriented $k$-planes on $U$;\\
$\llbracket S \rrbracket$ & (rectifiable) current induced by the $k$-dimensional \\& submanifold $S$ (taken with multiplicity 1);\\ 
& the same notation is used for the corresponding rectifiable varifold\\
$\mass(S)$ & mass norm of a current $S$;\\
$\fl(S)$ & flat norm of a current $S$;\\
$\mathrm{v}(R,\theta)$ & varifold induced by the $k$-rectifiable set $R$, with multiplicity $\theta$;\\
 $W_0$ & the set $\{(x_1, \ldots , x_{n+1})\in \mathbb R^{n+1} : |x_{n+1}| \leq x_1 \tan \theta \}$ with $\theta\in ]0, \frac{\pi}{2}[$\\ 
& which we will refer to as the {\em canonical wedge with opening angle $\theta$}.
\end{tabular}
\end{center}

Note that all of the different spaces of vector fields introduced above (namely $\mathfrak{X}_c (U)$, $\mathfrak{X}_c^0(U)$, $\mathfrak{X}^t_c(U)$ and $\mathfrak{X}^-_c (U)$) coincide when $U\cap \partial \M = \emptyset$. Otherwise  we have the inclusions
\[
\mathfrak{X}_c^0(U)\subset \mathfrak{X}^t_c(U)\subset \mathfrak{X}^-_c (U)\subset \mathfrak{X}_c (U)\, ,
\]
which are all proper.  Additional clarifications on the differences are provided in the next section. For the notation and terminology about currents and rectifiable varifolds we will follow \cite{simon-gmt}. However we warn the reader that, unless we specify that a given varifold is integer rectifiable, in general it will be not and will be understood as a suitable measure on the space of Grassmanians, according to \cite[Chapter 8]{simon-gmt}. 

\section{Existence of stationary varifolds}\label{s:pull-tight}

The first step in the min-max construction consists of finding a {\it nice} minimizing sequence having the property that {\it any} min-max sequence belonging to it converges to a stationary varifold. From now on we will denote the subset of stationary varifolds by $\va_s (\M)$ (or simply $\va_s$): the latter is the space of varifolds $V$ such that $\delta V (\chi) =0$ for any vector field $\mathfrak{X}_c^0 (\M)$, where $\delta V$ denotes, as usual, the first variation of $V$ (we refer to \cite{simon-gmt} for the relevant definition). We will however consider two slightly smaller subclasses of $\va_s$, depending on whether we are dealing with the constrained or unconstrained problem. To get an intuition consider the one-parameter families of smooth maps $\Phi_\tau$ generated by vector fields in $\mathfrak{X}_c$ following their flows and observe first that such family of smooth maps are keep fix any point $x\not \in U$. Concerning the points in $U$ (more precisely those in $U\cap \partial\M$), we have the following different behaviors:
\begin{itemize}
\item[(C)] If $\chi\in \mathfrak{X}_c^0 (U)$ then, for every $\tau$, $\Phi_\tau$ is a diffeomorphism of $\M$ onto itself which is the identity on $\partial \M$;
\item[(T)] If $\chi\in \mathfrak{X}_c^t (U)$ then, for every $\tau$, $\Phi_\tau$  is a diffeomorphism of $\M$ onto itself which maps $\partial \M \cap U$ {\em onto} itself;
\item[(I)] If $\chi\in \mathfrak{X}^-_c (U)$ then $\Phi_\tau$ is a well-defined map for $\tau\geq 0$, but not necessarily for $\tau<0$; moreover, for each $\tau\geq 0$, $\Phi_\tau$ is a diffeomorphism of $\M$ with $\Phi_\tau (\M)\subset \M$, but in general $\Phi_\tau (\M)$ will be a proper subset of $\M$, i.e. $\Phi_\tau$ rather than mapping $\partial \M$ into itself might ``push it inwards''.
\end{itemize}
It is thus clear that $\mathfrak{X}_c^t$ is a natural class of variations for the uncostrained problem, whereas a vector field in $\mathfrak{X}^-_c$ gives a natural (one-sided) variation for the constrained problem if we impose that it vanishes on the fixed boundary $\gamma$. This motivates the following

\begin{definition}\label{d:vairations}
In the ``constrained'' min-max problem, where the boundary constraint is $\gamma$, we introduce the set $\va^c_s (\M, \gamma)$ (or
shortly $\va^c_s (\gamma)$) which consists of those varifolds satisfying the condition
\begin{equation}\label{e:add_constr}
\delta V (\chi) \geq 0 \qquad \mbox{{ for all $\chi \in \mathfrak{X}^-_c (\M)$ which vanish on $\gamma$.}}
\end{equation}
In the ``unconstrained'' min-max problem we introduce the set $\va_s^u$ which consists of those varifolds which are stationary for all variations in $\chi\in \mathfrak{X}_c^t (\M)$:
\begin{equation}\label{e:add_unconstr}
\delta V (\chi) = 0 \qquad \mbox{for all $\chi \in \mathfrak{X}^t_c (\M)$.}
\end{equation}
\end{definition}

Clearly, since $\mathfrak{X}^0_c (\M)\subset \mathfrak{X}^t_c (\M)$, $\va_s^u$ is a subset of the stationary varifolds $\va_s$. Note moreover that, if $\chi \in \mathfrak{X}^0_c (\M)$, { then both $\chi$ and $-\chi$ belong to $\mathfrak{X}^-_c (\M)$ and vanish on $\gamma$}: therefore we again conclude $\va_s^c (\gamma) \subset \va_s$. 

For the purpose of this section, we will consider the subset $\va(\M, 4m_0)$ of varifolds with mass bounded by $4m_0 = 4 m_0 (X)$ (the latter being the minmax value of Theorem \ref{t:main}). Recall that the weak* topology on this set is metrizable, and we choose a metric $\mathcal{D}$ which induces it.  We are now ready to state the main technical proposition of this section which, as already mentioned, will be proved using the classical pull-tight procedure of Almgren. 

\begin{propos}\label{p:pull-tight}
Let $X$ be a homotopically closed set of smooth families parametrized \mbox{by $\P$}, such that \eqref{mountain pass family} is satisfied. Then:
\begin{itemize}
\item[(C)] In the problem constrained by $\gamma$ there exists a minimizing sequence $\{\{\Gamma_t\}^\ell\} \subset X$ such that, if $\{\Gamma^{\ell}_{t_{\ell}}\}$ is a min-max sequence, then $\mathcal{D}(\Gamma^{\ell}_{t_{\ell}}, \va^c_s (\gamma)) \to 0$;
\item[(U)] In the unconstrained problem there exists a minimizing sequence $\{\{\Gamma_t\}^\ell\} \subset X$ such that, if $\{\Gamma^{\ell}_{t_{\ell}}\}$ is a min-max sequence, then $\mathcal{D}(\Gamma^{\ell}_{t_{\ell}}, \va^u_s) \to 0$.
\end{itemize}
\end{propos}

\begin{proof} In what follows we will use $\va_s^c$ in place of $\va_s^c (\gamma)$ for the constrained case. 
{ In order to simplify our discussion, we introduce the notation $\mathfrak{X}^- (U, \gamma)$ for the class of vector fields which belong to $\mathfrak{X}^-_c (U)$ and vanish on $\gamma$}.
We can repeat the first two steps in the proof of \cite[Proposition 4.1]{c-dl} verbatim, the only exception being that we consider vector fields in $\mathfrak{X}^t_c(\M)$ { or $\mathfrak{X}^-(\M, \gamma)$} and thus we replace $\va_s$ with $\va_s^u$ and $\va_s^c$ in the respective cases. Since both these sets of vector fields are convex subsets of $\mathfrak{X} (\M)$, the vector field $H_V$ produced in \cite[Proof of Proposition 4.1, Step 1]{c-dl} will also belong to the same class. This way we obtain a map $\Psi$ as in \cite[Proof of Proposition 4.1, Step 3]{c-dl}, which sends each varifold $V \in \va(\M, 4m_0)$ to a 1-parameter family of maps $\Psi_V:[0,\infty)\times\M \to \M$, and a continuous, strictly increasing function $L:\RR \to \RR$ s.t. $L(0)=0$, with the property that
\begin{equation*}
\text{if } \eta = \mathcal{D}(V,\va^\square_s)>0, \text{ then } ||\Psi_V(1,\cdot)_{\sharp}V||(\M) \leq ||V||(\M) - L(\eta)\, ,
\end{equation*}
where $\square$ is either $u$ or $c$, depending on the case considered. The map $\Psi_V (s, \cdot)$ is given by following the one-parameter family of diffeomorphisms generated by $T (V) H_V$ where the ``speed'' $T$ is a nonnegative continuous function as in \cite[Proof of Proposition 4.1, Step 2] {c-dl} with $\va^\square_s$ replacing the space of all stationary varifolds. In particular $T$ is zero on $\va^\square_s$ and thus $\Psi_V (s, \cdot)_\sharp V = V$ for any varifold in $\va^\square_s$. 
Note moreover that in the unconstrained case $\Psi_V (s, \cdot)$ is a diffeomorphism from $\M$ to itself, whereas in the constrained case it is a diffeomorphism of $\M$ with $\Psi_V (s, \M)$ which however keeps $\gamma$ fixed. 

At this point we diverge slightly from \cite{c-dl} and introduce the set 
\[
\va_{\partial}:=\{ \Xi_t \, | \, t \in \partial \P, \{\Xi_t\} \in X \}\, ,
\] 
which is a closed subset of $\va(\M)$. Note that, according to our notion of homotopy in Definition \ref{homotopies}, this definition is independent of the family $\{\Xi_t\}\in X$ we choose. We define $b(V):=\min\{\mathcal{D}(V,\va_{\partial}),1\}$ for $V \in \va(\M)$, and remark that $b: \va(\M) \to \RR$ is a continuous function. A quick computation as in \cite[Proof of Proposition 4.1, Step 2]{c-dl} yields
\begin{equation}\label{tightening}
||\Psi_V(b(V),\cdot)_{\sharp}V||(\M) \leq ||V||(\M) - b(V)L(\mathcal{D}\big(V,\va^{\square}_{s})\big).
\end{equation}
We now renormalize the diffeomorphisms $\Psi_V$ by setting
\begin{equation*}
\Omega_V(s,\cdot) = \Psi_V (b(V)s,\cdot), \quad s\in [0,1]\, .
\end{equation*}
We proceed exactly the same as in the rest of \cite[Proof of Proposition 4.1, Step 3]{c-dl}, only with $\Omega$ instead of $\Psi$, and a different parameter space $\P$. Hence, if we start with a sequence of families $\{\{\Sigma_t\}^{\ell}\} \subset X$ such that $\max_{t\in\P} \haus^n (\Sigma^\ell_t) \leq m_0(X) + \frac{1}{\ell}$, we consider for each $\ell$ the map
\begin{equation*}
h^\ell: \P \to { \mathfrak{X}^-(\M , \gamma)} \quad (\text{or } \mathfrak{X}^t_c(\M)),
\end{equation*} 
given by 
$h^\ell_t=b(\Sigma_t^{\ell})T\big(\mathcal{D}(\Sigma_t^{\ell},\va^\square_{s})\big)H_{\Sigma_t^{\ell}}$: such smooth vector field  generates $\Omega_{\Sigma_t^{\ell}}$. Note that $h^\ell_t=0$ for $t \in \partial \P$. Moreover the map $h^\ell$ is continuous if 
$\mathfrak{X}_c(\M)$ is endowed with the topology of $C^k$-seminorms. We next smooth the mapp $t\mapsto h^\ell_t$ by keeping it $0$ on $\partial \P$. Consider now the $1$-parameter family of maps generated by such smoothing, which (by a slight abuse of notation) we still denote by $\Omega^\ell_t (s, \cdot)$. We are ready to define a new family $\Gamma^{\ell}_t = \Omega_t^\ell (1, \Sigma_t^\ell)$
Since $\Omega^\ell_t (s, \cdot)$ is the identity for $t\in \partial \P$,  the new family $\{\Gamma^\ell_t\}_t$ is homotopic to $\{\Sigma^l_t\}_t$. By the rest of the construction and \eqref{tightening}, assuming that the smoothing of $h^\ell_t$ is sufficiently close to it, we then have
\begin{equation}\label{tightening 2}
\haus^n(\Gamma_t^{\ell}) \leq \haus^n(\Sigma_t^{\ell}) - \frac{b(\Sigma_t^{\ell})L\big(\mathcal{D}(\Sigma_t^{\ell},\va_{s})\big)}{2}.
\end{equation}
Moreover, there will be an increasing continuous map $\lambda: \RR^+ \to \RR^+$ with $\lambda(0)=0$ and
\begin{equation}\label{tightening 3}
\mathcal{D}(\Sigma_t^{\ell}, \va^\square_{s}) \geq \lambda \big(\mathcal{D}(\Gamma_t^{\ell},\va^\square_{s})\big). 
\end{equation}
Finally, we claim that for every $\epsilon$, there exist $\delta >0$ and $N \in \NN$ such that
\begin{equation}\label{tightening 4}
\text{if } \left\{ \begin{array}{c}
k>N\\
\text{and } \haus^n(\Gamma^{\ell}_{t_{\ell}}) > m_0 - \delta
\end{array} \right\}, \quad \text{then } \mathcal{D}(\Gamma^{\ell}_{t_{\ell}},\va^\square_{s})<\epsilon.
\end{equation}
Let us therefore fix $\epsilon>0$. Considering that $b(W)=0 \>\>\forall W \in \va_{\partial}$, the continuity of mass of varifolds clearly implies that, if we set $\xi:=\frac{m_0(X)-bM_0(X)}{2}$, then for all $V \in \va(\M)$ with $\haus^n(V) \geq m_0-\xi=bM_0(X) + \xi$ we have $b(V) \geq c(\xi)>0$. We will choose $0<\delta<\xi$ and $N \in \NN$ satisfying
\begin{equation*}
\frac{c(\xi)L(\lambda(\epsilon))}{2} - \delta > \frac{1}{N}.
\end{equation*}
Assume now, contrary to \eqref{tightening 4}, there are $k>N$ and $t \in \P$ such that
\begin{equation*}
\haus^n(\Gamma_{t_{\ell}}^{\ell})>m_0-\delta \quad \text{ and } \quad \mathcal{D}(\Gamma_{t_{\ell}}^{\ell},\va^\square_{s}) > \epsilon.
\end{equation*}
Then, by \eqref{tightening 2}, \eqref{tightening 3} and the fact that $\haus^n(\Sigma_{t_{\ell}}^{\ell})\geq \haus^n(\Gamma_{t_{\ell}}^{\ell})>m_0 - \delta > m_0 - \xi$, we get
\begin{align*}
\haus^n(\Sigma_{t_{\ell}}^{\ell}) &\geq \haus^n(\Gamma_{t_{\ell}}^{\ell}) + \delta + \frac{c(\xi)L(\lambda(\epsilon))}{2} - \delta\\
&\geq m_0 + \frac{1}{N}.
\end{align*}
This contradicts $\max_{t\in\P} \haus^n (\Sigma^\ell_t) \leq m_0(X) + \frac{1}{\ell}$, and thus completes the proof of claim \eqref{tightening 4}, which in turn implies the proposition.
\end{proof}

\section{Almost minimizing property}\label{s:AM}

Following its introduction by Pitts \cite{pitts}, an important concept to achieve regularity for stationary varifolds produced by min-max theory is that of \emph{almost minimizing surfaces}. Roughly speaking, a surface is almost minimizing if any area-decreasing deformation must eventually pass through some surface with sufficiently large area. The precise definition we require here is the following:
\begin{definition}\label{amm property}
Let $\epsilon >0$, $U \in \M$ be an open subset, and fix $m \in \NN$. A surface $\Sigma$ is called \emph{$\epsilon$-almost minimizing} in $U$ if there is \underline{no} family of surfaces $\{\Sigma_t\}_{t \in [0,1]}$ satisfying the properties:
\begin{align}
\label{amm prop1}&\emph{(s1), (s2)} \text{ and } \emph{(s3)} \text{ of Definition 0.1 hold};\\ 
\label{amm prop2}&\Sigma_0 = \Sigma \text{ and } \Sigma_t \setminus U = \Sigma \setminus U \text{ for every } t \in [0,1];\\\label{amm prop3}
&\haus^n(\Sigma_t) \leq \haus^n(\Sigma) + \frac{\epsilon}{2^{m+2}} \text{ for all } t \in [0,1]; \\\label{amm prop4}
&\haus^n(\Sigma_1) \leq \haus^n(\Sigma) - \epsilon
\end{align}
A sequence $\{\Omega^i\}$ of surfaces is called almost minimizing (or a.m.) in $U$ if each $\Omega^i$ is $\epsilon_i$-almost minimizing in $U$ for some sequence $\epsilon_i \to 0$ (with the same $m$).
\end{definition}
\begin{remark}\label{fixing m}
The definition above is practically the same as the one given in \cite{dl-t} when $m=1$. The generalization is due to the more general parameter space $\P$. To be precise, we will henceforth fix $m \in \NN$ such that $\P$ can be smoothly embedded into $\RR^m$. 
\end{remark}
The main goal of this section is to prove an existence result regarding almost minimizing property in annuli:
\begin{propos}\label{a.m. in annuli}
Le $X$ be a homotopically closed set of (constrained or unconstrained) families in $\M^{n+1}$, parameterized by a smooth, compact $k$-dimensional manifold $\P$ (with or without boundary), and satisfying the condition \eqref{mountain pass family}. Then there is a function $r: \M \to \RR_{+}$ and a min-max sequence $\{\Gamma^k\} = \{\Gamma^k_{t_k}\}$ such that
\begin{itemize}
\item $\{\Gamma^k\}$ is a.m. in every  An $\in \mathcal{AN}_{r(x)}(x)$ with $x \in \M$ \\
\item $\{\Gamma^k\}$ converges to a stationary varifold $V$ as $k \to \infty$; the varifold $V$ belongs to $\mathcal{V}^c_s (\gamma)$ in the constrained case, whereas it belongs to $\mathcal{V}^u_s$ in the unconstrained case.
\end{itemize}
\end{propos}

An important corollary of the above proposition is the interior regularity, for which we refer to \cite{dl-t}. We record the consequence
here

\begin{propos}\label{interior}
The varifold $V$ of Proposition \ref{a.m. in annuli} is a regular embedded minimal surface in ${\rm Int}\, (\M)$, except for a
set of Hausdorff dimension at most $n-7$. 
\end{propos}

\begin{remark}\label{r:differente_2} As already noticed in Remark \ref{r:differente_1}, there is indeed a difference between the families considered here and the ones of \cite{dl-t}. For this reason, one cannot literally apply the statements in \cite{dl-t} to conclude Proposition \ref{interior} from Proposition \ref{a.m. in annuli}. However, this difference only requires a small technical adjustment, which is illustrated in Section \ref{s:differente_3}.
\end{remark}

In order to prove Proposition \ref{a.m. in annuli} we will be following the strategy laid out in Section 5 of \cite{c-dl} (see also Section 3 of \cite{dl-t}), which contains a similar statement. In fact, the main difference is the significant generalization of the parameter space $\P$. The case of higher dimensional cubes was covered in the master thesis of Fuchs \cite{fuchs}, and in this paper, some necessary modifications were made. The key ingredient of the proof is a combinatorial covering argument, a variant of the original one by Almgren and Pitts (see \cite{pitts}), and which we therefore refer to as the Almgren-Pitts combinatorial lemma. We will use it to prove Proposition \ref{a.m. in annuli} at the end of this section, and its proof will be provided in the next one. 
\begin{definition}
Let $d \in \NN$ and $U^1, \ldots U^d$ be open sets in $\M$. A surface $\Sigma$ is said to be \emph{$\epsilon$-almost minimizing in $(U^1, \ldots, U^d)$} if it is $\epsilon$-a.m. in at least one of the open sets $U^1, \ldots U^d$.\\
Furthermore, we define
\begin{equation*}
\emph{dist}(U,V):= \inf_{u \in U, v \in V} d_g(u,v)
\end{equation*}
as the distance between the two sets $U$ and $V$ ($d_g$ being the Riemannian distance).\\
Finally, for any $d \in \NN$ we denote by $\mathcal{CO}_d$ the set of $d$-tuples $(U^1, \ldots, U^d)$, where $U^1, \ldots, U^d$ are open sets with the property that
\begin{equation*}
\emph{dist}(U^i, U^j) \geq 4\cdot \min\{\emph{diam}(U^i),\emph{diam}(U^j)\}
\end{equation*}
for all $i,j \in \{1,\ldots,d\}$ with $i \neq j$.
\end{definition}
We require also the following lemma as preparation:
\begin{lemma}\label{combinatorial}
Let $p \in \NN$. Then there exists $\omega_p \in \NN$ with the following property:
\begin{itemize}
\item[(CA)] Assume $\mathcal{F}_1=(U_1^1,\ldots,U_1^{\omega_p}),\ldots,\mathcal{F}_{2^p}=(U_{2^p}^1,\ldots,U_{2^p}^{\omega_p})$ are $2^p$ families of open sets with the property that
\begin{equation}\label{dist assumption}
\emph{dist}(U_i^j, U_i^{j'}) \geq 2\cdot \min\{\emph{diam}(U_i^j),\emph{diam}(U_i^{j'})\}
\end{equation} 
for all $i \in \{1,\ldots,2^p\}$ and for all $j,j' \in \{1,\ldots,\omega_p\}$ with $j \neq j'$.\\
Then we can extract $2^p$ subfamilies $\mathcal{F}_1^{sub} \subset \mathcal{F}_1,\ldots,\mathcal{F}_{2^p}^{sub} \subset \mathcal{F}_{2^p}$ such that
  \begin{itemize}\renewcommand{\labelitemii}{$\bullet$}
  \item $\emph{dist}(U, V)>0$ for all $U \in \mathcal{F}_i^{sub}, V \in \mathcal{F}_j^{sub}$ with $i,j \in \{1,\ldots,2^p\}$ and $i \neq j$;\\
  \item $\mathcal{F}_i^{sub}$ contains at least $2^p$ open sets for every $i \in \{1,\ldots,2^p\}$.
  \end{itemize}
\end{itemize}
\end{lemma}
\begin{proof}
Let $\mathcal{F}_1,\ldots,\mathcal{F}_{2^p}$ be as in the assumption (CA), with $\omega_p$ some (natural) number, to be fixed later. First note that, if $U \in \mathcal{F}_i$ and $V^1, \ldots V^l \in \mathcal{F}_s$ with $i \neq s$ and $\diam(U) \leq \diam(V^j), \> j \in \{1,\ldots,l\}$, then there is at most one $j \in \{1,\ldots,l\}$ with dist$(U,V^j) = 0$. Otherwise, assuming there are two such sets $V^{j_1}, V^{j_2}$ with dist$(V^{j_1},U)=0$, dist$(V^{j_2},U)=0$ and w.l.o.g. $\diam(V^{j_1}) \leq \diam(V^{j_2})$, we would get
\begin{equation*}
\dist(V^{j_1},V^{j_2}) \leq \diam(U) \leq \diam (V^{j_1}),
\end{equation*} 
which contradicts the assumption \eqref{dist assumption}.
Now, in order to produce the subfamilies, one can employ the following algorithm:
\begin{itemize}
\item[--] take all the sets in all the families and arrange them in an ascending order with respect to their diameters, left to right (from smallest to largest). In the first step, fix the leftmost set;
\item[--] at each step of the process, remove all the sets to the right of the fixed set which are at distance zero with respect to it. Furthermore, if to the left of the currently fixed set there are $2^p -1$ remaining sets from the same family $\mathcal{F}_i$, remove all the sets to the right which belong to the same family (the latter operation will be called, for convenience, ``clearing of the family $\mathcal{F}_i$'');
\item[--] move on to the first (remaining) set to the right of the previously fixed set, fix it, and repeat the step above.
\end{itemize}
We claim that the remaining sets build the desired subfamilies.
Firstly, it is obvious from the construction that for any two remaining sets $U,V$ we have $\dist(U,V)>0$. Secondly, we see from the consideration at the beginning of the proof that at each step we remove at most one set from each family to which the fixed set does not belong to (and none from the same family, due to \eqref{dist assumption}). Finally, since any family that reaches $2^p$ remaining elements is removed from the process, it can account for no more than $2^p$ removed elements from any other family. Hence, if for some $\mathcal{F}_i$ we do not reach the stage at each we ``clear'' $\mathcal{F}_i$, we have removed at most $2^p(2^p -1)$ elements from $\mathcal{F}_i$ and retained at most $2^p -1$. Hence, if we choose any $\omega_p \geq 4^p$ we can ensure that the clearing process happens for every family and thus that we have selected at least $2^p$ elements from each. 
\end{proof}

\begin{propos}\label{almgren-pitts}
\emph{(Almgren-Pitts combinatorial lemma)} Let $X$ be a homotopically closed set of families as in Proposition \ref{a.m. in annuli}. Assume $\P$ is smoothly embedded into $\RR^m$, and let $\omega_m$ be as in Lemma \ref{combinatorial}. Then there exists a min-max sequence $\{\Gamma^N\}=\big\{\Gamma^N_{t_N}\big\}$ such that
\begin{itemize}
\item $\{\Gamma^N\}$ converges to a stationary varifold $V$, which belongs to $\mathcal{V}^c_s (\gamma)$ in the constrained case and to $\mathcal{V}^u_s$ in the unconstrained case;\\
\item for any $(U^1,\ldots,U^{\omega_m}) \in \mathcal{CO}_{\omega_m}, \> \Gamma^N$ is $\frac{1}{N}$-a.m. in $(U^1,\ldots,U^{\omega_m})$, for $N$ large enough.\\
\end{itemize}
\end{propos}

We can now prove the main proposition as a corollary of the above.\\

\noindent {\bf\emph{Proof of Proposition \ref{a.m. in annuli}.}} We will show that a subsequence of $\{\Gamma^j\}$ in Proposition \ref{almgren-pitts} satisfies the requirements. For each positive $r_1< \text{Inj} (\M)$ and for each choice of $r_2, \ldots, r_{\omega}$ with the property that $r_i < \frac{1}{9} r_{i-1}$  consider the tuple $(U^1_{r_1}(x),\ldots,U^{\omega_m}_{r_{\omega_m}}(x))$ given by
\begin{align}
&U^1_{r_1}(x):=\M \setminus \overline{B}_{r_1}(x); \\
&U^l_{r_l}(x):=B_{\tilde{r}_{l}}(x)\setminus \overline{B}_{r_l}(x) \text{ where } \tilde{r}_l:= \frac{1}{9}r_{l-1}\, ;\\
&U^{\omega_m}_{r_{\omega_m}}:=B_{r_{\omega_m}}(x) \text{ where } r_{\omega_m} \leq \frac{1}{9}r_{\omega_m-1}.
\end{align}
Then, by definition, $(U^1_{r_1}(x),\ldots,U^{\omega_m}_{r_{\omega_m}}(x)) \in \mathcal{CO}_{\omega_m}$ and $\Gamma^j$ is therefore (for $j$ large enough) $\frac{1}{j}$-a.m. in at least one $U^l_{r_l}(x), \> 1 \leq l \leq \omega_m$. Having fixed $r_1>0$, one of the following options holds:\\
\begin{itemize}
\item[(a)] either $\{\Gamma^j\}$ is (for $j$ large) $\frac{1}{j}$-a.m. in $(U^2_{r_2}(y),\ldots,U^{\omega_m}_{r_{\omega_m}}(y))$ for every $y \in \M$ and every choice of $r_2, \ldots, r_{\omega}$ compatible with the requirement $r_i < \frac{1}{9} r_{i-1}$;\\
\item[(b)] or, for each $K \in \NN$, there exists some $s_K \geq K$ and a point $x^{s_K}_{r_1} \subset \M$ such that $\Gamma^{s_K}$ is $\frac{1}{s_K}$-a.m. in $\M \setminus \overline{B}_{r_1}(x_{r_1}^{s_K})$.\\
\end{itemize}

Assume there is no $r_1>0$ such that (a) holds. Thus, choosing option (b) with $r_1=\frac{1}{j}$ and $K=j$ for each $j \in \NN$, we obtain a subsequence $\{\Gamma^{s_j}\}_{j \in \NN}$, and a sequence of points $\{x_j^{s_j}\}_{j \in \NN} \subset \M$ such that $\Gamma^{s_j}$ is $\frac{1}{s_j}$-a.m. in $\M \setminus \overline{B}_{\frac{1}{j}}(x_{j}^{s_j})$. Since $\M$ is compact, there exists some $x \in \M$ such that $x_j^{s_j} \to x$. We conclude that, for any $N \in \NN$, $\Gamma^{s_j}$ is $\frac{1}{s_j}$-a.m. in $\M \setminus \overline{B}_{\frac{1}{N}}(x)$ for $j$ large enough. Consequently, if $y \in \M \setminus \{x\}$, we can choose $r(y)$ such that $B_{r(y)}(y) \subset \subset \M \setminus \{x\}$, whereas $r(x)$ can be chosen arbitrarily: with such choice $\{\Gamma^{s_j}\}$ is a.m. in any annullus of $\mathcal{AN}_{r(z)} (z)$ for any $z\in \M$.\\\
Assume now that there is some fixed $r_1 >0$ such that (a) holds. Note that, in this case, for any $x$ there is a $J$ (possibly depending on $x$) such that $\Gamma^j$ is \underline{not} $\frac{1}{j}$-a.m. in $U^1_{r_1}$ for all $j\geq J$. Due to compactness, we can divide the manifold $\M$ into finitely many, nonempty, closed subsets $\M_1,\ldots,\M_N \subset \M$ such that
\begin{itemize}
\item $0 < \diam(\M_i) < \tilde{r}_2 =\frac{1}{9}r_1$ for every $i \in \{1,\ldots,N\}$;
\item $\M = \cup \M_i$.
\end{itemize}
Similar to the reasoning above, for each $\M_i$, starting with $\M_1$, we consider two mutually exclusive cases:\\
\begin{itemize}
\item[(a)] either there exists some fixed $r_{2,i}>0$ such that $\{\Gamma^j\}$ must be (for $j$ large) $\frac{1}{j}$-a.m. in $(U^3_{r_3}(y),\ldots,U^{\omega_m}_{r_{\omega_m}}(y))$ for every $y \in \M_i$ and every choice of radii $r_3, \ldots , r_{\omega_m}$ with $r_3 < \frac{1}{9} r_{2,i}$ and $r_j < \frac{1}{9} r_{j-1}$;\\
\item[(b)] or we can extract a subsequence $\{\Gamma^j\}$, not relabeled, and a sequence of points $\{x_{i,j}\} \subset \M_i$ such that $\Gamma^j$ is $\frac{1}{j}$-a.m. in $B_{\tilde{r}_2}(x_{i,j}) \setminus \overline{B}_{\frac{1}{j}}(x_{i,j})$
\end{itemize}
Again, if (b) holds, we know $x_{i,j}\to x_i \in \M_i$, and we can choose $r(x_i) \in (\diam(\M_i),\tilde{r}_2)$. Accordingly, for any other $y \in \M_i$, we can choose $r(y)$ such that $B_{r(y)}(y) \subset \subset B_{r(x_i)}(x_i) \setminus \{x_i\}$. We proceed onto $\M_{i+1}$, where either (a) gets chosen, or we possibly extract a futher subsequence, and define further values of the function $r$. For the subsets $\M_{i_1},\ldots,\M_{i_l}$ where option (a) holds, we define $r_2:=\min\{r_{2,i_1},\ldots,r_{2,i_l}\}$, and then continue iteratively, by first subdividing the sets and then considering the relevant cases. Finally, note that if in the last instance of the iteration we choose option (a) for certain subsets, it means that, in those sets, $\Gamma^j$ must be (for $j$ large) $\frac{1}{j}$-a.m. in $B_{r_{\omega_m}}(y)$ for some $r_{\omega_m}>0$ and all $y$, hence we can choose $r(y)=r_{\omega_m}$, and we are done.
\qed

\section{Almgren-Pitts combinatorial lemma}\label{s:comb}

In this section, we turn to proving Proposition \ref{almgren-pitts}, which will be done by contradiction. Assuming no min-max sequence (extracted from an appropriate minimizing sequence) with the required property exists, we are able to construct a competitor minimizing sequence $\{\Sigma_t\}^N$ with energy (i.e. $\max_{t \in \P}\{\Sigma_t\}^N$), lowered by a fixed amount, thus reaching a contradiction to the minimality of the original sequence. This will be done using two main ingredients. The first is a technical lemma which enables us to use the "static" variational principle in Definition \ref{amm property} for a single, fixed time slice to construct a "dynamic" competitor family of surfaces. This is achieved by using a tool called "freezing", introduced in \cite{dl-t} (see Lemma 3.1). The statement and proof we present here are slightly different. In the rest of this section we will use the notation
$Q (t_0, r)$ for the $p$-dimensional cube centered at $t_0$ with sidelength $2r$, namely
\begin{equation*}
Q(t_0,r):=\big\{t=(t^1,\ldots,t^p) \in [0,1]^p \>\> | \> t_0^i -r < t^i < t_0^i + r\> \forall i \in \{1,\ldots,p\} \big\}\, .
\end{equation*}

\begin{lemma}\label{freezing lemma}
Let $U \subset \subset U' \subset \M$ be two open sets, and $\{\Xi_t\}_{t \in [0,1]^p}$ be a smooth family parameterized by $[0,1]^p$, with $p \in \NN$ fixed. Given an $\epsilon >0$ and $t_0 \in (0,1)^p$, suppose $\{\Sigma_s\}_{s \in [0,1]}$ is a $1$-parameter family of surfaces satisfying properties (\ref{amm prop1})-(\ref{amm prop4}), with $\Sigma_0 = \Xi_{t_0}$ and $m=p$. Then there is an $\eta>0$ such that the following holds for every $a',a$ with $0<a'<a<\eta$:\\ There is a competitor (smooth) family $\{\Xi'_t\}_{t \in [0,1]^p}$ such that
\begin{align}
& \Xi_t=\Xi'_t \>\> for \>\> t \in [0,1]^p \setminus Q(t_0,a), \>\> and \>\> \Xi_t \setminus U' = \Xi'_t \setminus U' \>\> for \>\> t \in Q(t_0,a);\label{freeze1} \\
& \haus^n(\Xi'_t) \leq \haus^n(\Xi_t) + \frac{\epsilon}{2^{p+1}} \>\> for \> every \>\> t \in [0,1]^p;\label{freeze2} \\
& \haus^n(\Xi'_t) \leq \haus^n(\Xi_t) - \frac{\epsilon}{2} \>\> for \> every \>\> t \in Q(t_0,a')\, . \label{freeze3}
\end{align}
Moreover, $\{\Xi'_t\}$ is homotopic to $\{\Xi_t\}$.
\end{lemma}

\begin{proof}
{\bf Step 1: Freezing}. First we will choose open sets $A_1, A_2$ and $B_1,B_2$ satisfying
\begin{equation*}
U \subset \subset A_1 \subset \subset A_2 \subset \subset B_1 \subset \subset B_2 \subset \subset U',
\end{equation*}
and such that $\Xi_{t_0} \cap \tilde{C}$ is a smooth surface, where $\tilde{C} := B_2\setminus \bar{A}_1$, which is possible since $\Xi_{t_0}$ contains only finitely many singularities. In a tubular $\delta$-neighborhood (w.r.t the normal bundle) of $\Xi_{t_0} \cap \tilde{C}$ we fix normal coordinates $(z,\sigma) \in (\Xi_{t_0} \cap \tilde{C}) \times (-\delta,\delta)$ (from now on we use the notation $\Omega \times (\alpha, \beta)$ to identify those points of the tubular neighborhood which lie on top of $\Omega\subset \Xi_{t_0}$ and have signed distance bounded between $\alpha$ and $\beta$).  By choosing $\delta$ small enough and$/$or redefining $A_i$-s and $B_i$-s, we can ensure that $(\Xi_{t_0} \cap A_2) \times (-\delta,\delta) \subset \subset B_1$ and $(\Xi_{t_0} \cap B_1) \times (-\delta,\delta) \subset \subset B_2$. Now, after defining the open sets $A:= A_1 \cup \big( (\Xi_{t_0} \cap A_2) \times (-\delta,\delta)\big) $, and $B:= \big((B_1 \cap \Xi_{t_0}) \times (-\delta,\delta)\big) \cup \big( B_2 \setminus ((\bar{B}_2 \cap\Xi_{t_0})\times[-\delta,\delta]) \big)$, we set $C:=B \setminus \bar{A}$ and deduce the following properties:
\vspace{1mm}

\begin{itemize}
\item [(a)]$U \subset \subset A \subset \subset B \subset \subset U'$;\\
\item [(b)]$\Xi_{t_0} \cap C$ is a smooth surface;\\
\item [(c)]we can fix $\eta>0$ such that $\Xi_t \cap C$ is the graph of a function $g_t$ over $\Xi_{t_0} \cap C$ for $t \in Q(t_0,\eta)$. 
\end{itemize}  
Note that the slightly complicated definitions above are only to ensure the property (c), or in other words, that the set $C$ is "cylindrical" near $\Xi_{t_0}$ so that $\Xi_t \cap C$ can in fact be entirely represented as a graph over $\Xi_{t_0} \cap C$, i.e. $\Xi_t \cap C= \{(z,\sigma)| \sigma = g_t(z), \> z \in \Xi_{t_0} \cap C\} \,\, \forall t \in Q(t_0,\eta).$ \\
Next, we fix two smooth functions $\varphi_A$ and $\varphi_B$ such that
\begin{itemize}
\item $\varphi_A + \varphi_B = 1$;\\
\item $\varphi_A \in C_c^{\infty}(B), \> \varphi_B \in C_c^{\infty}(\M \setminus \bar{A})$
\end{itemize}
We then introduce the functions
\begin{equation*}
g_{t,\tau}:=\varphi_B g_t + \varphi_A g_{\tau}, \quad t,\tau \in Q(t_0,\eta), s \in [0,1]
\end{equation*}
Since $g_t$ converges smoothly to $g_{t_0}(=0 )$ as $t \to t_0$, we can make $\sup_{\tau}\|g_{t,\tau} - g_t\|_{C^1}$ arbitrarily small by choosing $\eta$ small. Moreover, if we express the area of the graph of a function $g$ over $\Xi_{t_0} \cap C$ as an integral functional of $g$, we know that it only depends on $g$ and its first derivatives. Thus, if $\Gamma_{t,\tau}$ is the graph of $g_{t,\tau}$, we can find $\eta$ small enough such that
\begin{equation}\label{graph area}
 \haus^n(\Gamma_{t,\tau}) \leq \haus^n(\Xi_t \cap C) + \frac{\epsilon}{2^{p+3}}.
\end{equation}
Now, given $0<a'<a<\eta$, we choose $a'' \in (a',a)$ and fix a 
smooth function $\vartheta:Q(t_0,a) \to Q(t_0,\eta)$ which is equal to the identity in a neighborhood of $\partial Q(t_0,a)$ and equal to $t_0$ in $Q(t_0,a'')$.
We now define a new family $\{\Delta_t\}$ as follows:
\begin{itemize}
\item $\Delta_t=\Xi_t$ for $t \notin Q(t_0,a)$;\\
\item $\Delta_t \setminus \bar{B} =\Xi_t \setminus \bar{B}$ for all $t$;\\
\item $\Delta_t \cap A = \Xi_{\gamma(t)} \cap A$ for $t \in Q(t_0,a)$;\\
\item $\Delta_t \cap C = \{(z,\sigma)\,|\,\sigma = g_{t,\vartheta (t)},\, z \in \Xi_{t_0}\cap C\}$ for $t \in Q(t_0,a)$.
\end{itemize}
Note that $\{ \Delta_t\}$ is a smooth family homotopic to $\{\Xi_t\}$, they both coincide outside of $B$ (and hence outside of $U'$) for \emph{every} $t$, and that in $A$ (and hence in $U$) we have $\Delta_t = \Xi_{\vartheta (t)}$ for $t \in Q(t_0,a)$. Since $\vartheta (t)$ is equal to $t_0$ for $t \in Q(t,a'')$, it follows that $\Delta_t \cap U = \Xi_{t_0} \cap U$ for $t \in Q(t_0,a'')$, or in other words, $\Delta \cap U$ is \emph{frozen} in $Q(t_0,a'')$. Furthermore, because of \eqref{graph area}, 
\begin{equation}\label{area est.1}
\haus^n(\Delta_t \cap C) \leq \haus^n(\Xi_t \cap C) + \frac{\epsilon}{2^{p+3}} \quad \text{for } t \in Q(t_0,a).
\end{equation}
{\bf Step 2: Dynamic competitor.} We fix a smooth function $\chi:Q(t_0,a'') \to [0,1]$ which is identically $0$ in a neighborhood of $\partial Q(t_0,a'')$, and identically $1$ on $Q(t_0,a')$. We then define a competitor family $\{\Xi_t'\}$ in the following way:
\begin{itemize}
\item $\Xi_t' = \Delta_t$ for $t \notin Q(t_0,a'')$;\\
\item $\Xi_t' \setminus A = \Delta_t \setminus A$ for $t \in Q(t_0,a'')$;\\
\item $\Xi_t' \cap A = \Sigma_{\chi(t)} \cap A$ for $t \in Q(t_0,a'')$.
\end{itemize}
The new family $\{\Xi_t'\}$ is also a smooth family, which is obviously homotopic in the sense of Definition \ref{homotopies} to $\{\Delta_t\}$ and hence to $\{\Xi_t\}$, so long as we ensure $a$ is small enough that $Q(t_0,a) \subseteq (0,1)^p$. We can now start estimating $\haus^n(\Xi_t')$.\\
For $t \notin Q(t_0,a)$, we have $\Xi_t' = \Delta_t = \Xi_t$, so
\begin{equation}\label{area est2}
\haus^n(\Xi_t') = \haus^n(\Xi_t) \quad \text{for } t \notin Q(t_0,a).
\end{equation}
For $t \in Q(t_0,a)$, we have $\Xi_t \setminus \bar{B} = \Xi_t' \setminus \bar{B}$ and hence $\Xi_t' \setminus U' = \Xi_t \setminus U'$. This shows property \eqref{freeze1} of the lemma.\\
In the set $C$ we have $\Xi_t' = \Delta_t$ for $t \in Q(t_0,a)$, thus owing to \eqref{area est.1},
\begin{align}\label{area est3}
\haus^n(\Xi_t') - \haus^n(\Xi_t) &= [\haus^n(\Delta_t \cap C) - \haus^n(\Xi_t \cap C)] +[\haus^n(\Xi_t' \cap A) - \haus^n(\Xi_t \cap A)]\nonumber\\
&\overset{\eqref{area est.1}}{\leq} \frac{\epsilon}{2^{p+3}} + [\haus^n(\Xi_t' \cap A) - \haus^n(\Xi_t \cap A)].
\end{align}
Next, we want to estimate the area in $A$ for $t \in Q(t_0,a)$. To do so, we consider several cases separately:
\begin{itemize}
\item[(i)] Let $t \in Q(t_0,a) \setminus Q(t_0,a'')$. Then $\Xi_t' \cap A = \Delta_t \cap A = \Xi_{\gamma(t)} \cap A$. However, $t, \gamma(t) \in Q(t_0,\eta)$ and, having chosen $\eta$ small enough, we can assume that
\begin{equation}\label{area est4}
| \haus^n(\Xi_s \cap A) - \haus^n(\Xi_{\sigma} \cap A)| \leq \frac{\epsilon}{2^{p+3}} \quad \text{for every } \sigma,s \in Q(t_0,\eta).
\end{equation}
Hence, we deduce with \eqref{area est3} that
\begin{equation}\label{area est5}
\haus^n(\Xi_t') \leq \haus^n(\Xi_t) + \frac{\epsilon}{2^{p+2}}.
\end{equation}
\item[(ii)] Let $t \in Q(t_0,a'') \setminus Q(t_0,a')$. Then $\Xi_t' \cap A = \Sigma_{\chi(t)} \cap A$. Therefore, with \eqref{area est3} it follows
\begin{align}\label{area est6}
\haus^n(\Xi_t') - \haus^n(\Xi_t) &\leq \frac{\epsilon}{2^{p+3}} + [ \haus^n(\Xi_{t_0} \cap A) - \haus^n(\Xi_t \cap A)]\nonumber\\ 
&+ [ \haus^n(\Sigma_{\chi(t)} \cap A) - \haus^n(\Xi_{t_0} \cap A)] \nonumber \\
&\overset{\eqref{amm prop3},\eqref{area est4}}{\leq} \frac{\epsilon}{2^{p+3}} + \frac{\epsilon}{2^{p+3}} + \frac{\epsilon}{2^{p+2}} = \frac{\epsilon}{2^{p+1}}.
\end{align}
\item[(iii)] Let $t \in Q(t_0,a')$. Then we have $\Xi_t' \cap A = \Sigma_1 \cap A$. Using \eqref{area est3} again, we have
\begin{align}\label{area est7}
\haus^n(\Xi_t') - \haus^n(\Xi_t) &\leq \frac{\epsilon}{2^{p+3}} + [ \haus^n(\Sigma_1 \cap A) - \haus^n(\Xi_{t_0} \cap A)]\nonumber\\
&+ [ \haus^n(\Xi_{t_0} \cap A) - \haus^n(\Xi_t \cap A)]\nonumber\\
&\overset{\eqref{amm prop4},\eqref{area est4}}{\leq} \frac{\epsilon}{2^{p+3}} - \epsilon + \frac{\epsilon}{2^{p+3}} < -\frac{\epsilon}{2}.
\end{align}
\end{itemize}
Gathering the estimates \eqref{area est2}, \eqref{area est5}, \eqref{area est6} and \eqref{area est7}, we finally obtain the properties \eqref{freeze2} and \eqref{freeze3} of the lemma, which concludes the proof.
\end{proof}
By retracing the steps of the previous proof, we can see that it allows for (at least) two generalizations, which will be useful.
\begin{remark}\label{freezing remark}
\emph{(i)} Note that the choice to have the cubes $Q(t_0,a')$ and $Q(t_0,a)$ centered at $t_0$ is unnecessary and only for the sake of notational simplicity. Indeed, with the appropriate choice of cut-off functions $\psi, \gamma$ and $\chi$, the proof is almost identical if we replace them with cubes $Q(t_1,a')$ and $Q(t_2,a)$ (or even more general sets) that are nested inside each other, i.e. $Q(t_1,a') \subset \subset Q(t_2,a) \subset \subset Q(t_o,\eta)$.\\
\emph{(ii)} The lemma also works with minimal modifications if the family $\{\Xi_t\}_{t \in [0,1]^p}$ is parameterized by a $k$-dimensional smooth submanifold $\P$ of $[0,1]^p$, with $\partial\P \cap (0,1)^p = \emptyset$ in case it has a boundary. One can simply take restrictions of the relevant subsets of $[0,1]^p$ to their intersection with $\P$, both in the statement and the proof.
\end{remark}

In order to use the previous lemma to construct the aforementioned competitor minimizing sequence and prove the Almgren-Pitts lemma, we will require a combinatorial covering argument, which is the second main ingredient. The idea is to decrease areas of certain "large-area" slices by a definite amount, while simultaneously keeping the potential area increase for other slices under control.\\

\noindent {\bf \emph{Proof of Proposition \ref{almgren-pitts}.}}
Let $\{\{\Gamma^\ell_t\}\}^\ell \subset X$ be a minimizing sequence which satisfies Proposition \ref{p:pull-tight}, and such that $\mathcal{F}(\{\Gamma^\ell_t\}):=\underset{t \in \P}{\max}\,\haus^n(\Gamma^\ell_t) < m_0(X) + \frac{1}{2^{m+2}\ell}$. The following claim  clearly implies the proposition:\\

{\bf Claim}: {\it For every $N$ large enough there exists $t_N \in \P$ such that $\Gamma^N:=\Gamma^N_{t_{N}}$ is $\frac{1}{N}$-a.m. in every $(U^1,\ldots,U^{\omega_m}) \in \mathcal{CO}_{\omega_m}$ and $\haus^n(\Gamma^N) \geq m_0(X) - \frac{1}{N}$.}

\vspace{2mm}
We define
\begin{equation}
K_N:= \Big\{ t \in \P \> | \> \haus^n(\Gamma^N_t) \geq m_0(X) - \frac{1}{N} \Big\}
\end{equation}
and suppose, contrary to the claim, that there is some subsequence $\{N_j\}_j$ such that for every $t \in K_{N_j}$ there exists an $\omega_m$-tuple $(U^1,\ldots,U^{\omega_m})$ such that $\Gamma^{N_j}_t$ is not $\frac{1}{N_j}$-a.m. in it. After a translation and$/$or dilation, we can assume, without loss of generality, that $\P \subset [0,1]^m$ (in the embedding). Note that, if we assume $N$ to be large enough that $m_0(X) - 1/N>bM_0(X)$, the set $K_N$ will surely lie in the interior of $\P$. In fact, in everything that follows, it is tacitly assumed that the subsets of $\P$ we choose stay away from $\partial \P$, in order to comply with our definition of homotopic families.\\
By a slight abuse of notation, from now on we do not rename the subsequence, and also drop the super- and subscript $N$ from $\Gamma_t^N$ and $K_N$. Thus for every $t \in K$ there is a $\omega_m$-tuple of open sets $(U_{1,t},\ldots,U_{\omega_m,t}) \in \mathcal{CO}_{\omega_m}$ and $\omega_m$ families $\{\Sigma_{i,t,\tau}\}_{\tau \in [0,1]}$ such that the following properties hold for every $i \in \{1,\ldots,\omega_m\}$:
\begin{itemize}
\item $\Sigma_{i,t,0} = \Gamma_t$;\\
\item $\Sigma_{i,t,\tau}\setminus U_{i,t}=\Gamma_t \setminus U_{i,t}$;\\
\item $\haus^n(\Sigma_{i,t,\tau}) \leq \haus^n(\Gamma_t) +  \frac{1}{2^{m+2}N}$;\\
\item $\haus^n(\Sigma_{i,t,1}) \leq \haus^n(\Gamma_t) - \frac{1}{N}$.
\end{itemize}
By recalling the definition of $\mathcal{CO}_{\omega_m}$, for every $t \in K$ and every $i \in \{1,\ldots,\omega_m\}$ we  can choose an open set $U'_{i,t}$ such that $U_{i,t} \subset \subset U'_{i,t}$ and
\begin{equation}
\dist(U'_{i,t},U'_{j,t}) \geq 2 \cdot \min \{\diam(U'_{i,t}),\diam(U'_{j,t})\}
\end{equation}
for all $i,j \in \{1,\ldots,\omega_m\}$ with $i \neq j$. Next, we apply Lemma \ref{freezing lemma} with $\Xi_t = \Gamma_t$, $U=U_{i,t}$, $U'=U'_{i,t}$ and $\Sigma_{\tau} = \Sigma_{i,t,\tau}$. Hence, for every $t \in K$ and $i \in \{1,\ldots,\omega_m\}$ we get a corresponding constant $\eta_{i,t}$ given by the statement of the lemma. 

\vspace{2mm}
\noindent {\bf Step 1: Initial covering.}
We first assign to each $t \in K$ exactly one constant $\eta_t$, by setting $\eta_t:=\underset{i \in \{1,\ldots,\omega_m\}}{\min}\eta_{i,t}$. We would like to initially decompose the cube $[0,1]^m$ into a grid of small, slightly overlapping cubes, such that we might be able to apply the constructions in Lemma \ref{freezing lemma} to each of those (after discarding the ones which have empty intersection with $K$). For this, we would like their size to be smaller than the size of the cube given by the lemma for any point lying in the center of one of these cubes. Therefore, we choose a covering of $K$:
\begin{equation*}
\left\{ Q\big(((2r_1+1)\tilde{\eta},\ldots,(2r_m+1)\tilde{\eta}),\eta\big) \left| \begin{array}{l} r_1,\ldots,r_m \in \{1,\ldots,\xi\} \\ Q\big(((2r_1+1)\tilde{\eta},\ldots,(2r_m+1)\tilde{\eta}),\eta\big) \cap K \neq \emptyset  \end{array}\right. \right\},
\end{equation*}
where $\tilde{\eta}= \frac{9}{10}\eta, \> \xi = \min\{n \in \NN_0 \,|\, (2n+1)\tilde{\eta}>1-\eta\}$, and $\eta$ is yet to be determined. \\
Ideally, we would like $\eta$ to be smaller than any $\eta_t$. The problem, however, is that for each $t \in K$, the constant $\eta_t$ (which is determined by the proof of Lemma \ref{freezing lemma}) depends also on the sets $U_{i,t}$, so one might not in general expect to prove lower boundedness. Nevertheless, using \mbox{Remark \ref{freezing remark}(i)}, we deduce that if $t_0 \in K$, then for any $t \in Q(t_0,\frac{\eta_{i,t_0}}{2})$, the conclusions of the lemma hold with $\eta = \frac{\eta_{i,t_0}}{2}$ ($t$ being the center of the cubes now), and $U=U_{i,t_0}$. Therefore, for $t\,(\in K)$ close enough to $t_0$, we can replace $(U_{1,t},\ldots,U_{\omega_m,t})$ by $(U_{1,t_0},\ldots,U_{\omega_m,t_0})$ if necessary. Now, we can start by covering $K$ with $Q(t,\frac{\eta_t}{2}),\, t \in K$. Since $K$ is compact, it suffices to pick finitely many $t_0,\ldots,t_l$ with $K \subset \bigcup Q(t_i,\frac{\eta_{t_i}}{2})$. We then set:
\begin{equation}
\eta':= \min_{j \in \{0,\ldots,l\}}\frac{\eta_{t_j}}{2}
\end{equation}
Also note that for $N$ large enough, because of condition \eqref{mountain pass family}, the set $K$ lies in the interior of $\P$ (in case it has a boundary). That means there exists some $\eta'' >0$ such that for any cube $Q(t,\eta'')$ intersecting $K$ we have $\partial \P \cap Q(t,\eta'') = \emptyset$.\\
We define $\eta:= \min\{\frac{\eta'}{4},\eta''\}$, which determines the size of the cubes in the covering. Furthermore, we set
\begin{align*}
&\mathfrak{r}:=(r_1,\ldots,r_m); \\
&t_{\mathfrak{r}}:=((2r_1+1)\tilde{\eta},\ldots,(2r_m+1)\tilde{\eta})\\
&\mathfrak{Q_r}:=Q\big(((2r_1+1)\tilde{\eta},\ldots,(2r_m+1)\tilde{\eta}),\eta\big)
\end{align*}
To each $\mathfrak{Q_r}$ with $t_{\mathfrak{r}} \in K$ we can assign a corresponding $\omega_m$-tuple $(U_{1,t_{\mathfrak{r}}},\ldots,U_{\omega_m,t_{\mathfrak{r}}}) \in \mathcal{CO}_{\omega_m}$ by assumption. On the other hand, to any cube $\mathfrak{Q_r}$ in the covering (i.e. $\mathfrak{Q_r} \cap K \neq \emptyset$) where the center $t_{\mathfrak{r}} \notin K$, owing to \mbox{Remark \ref{freezing remark}(i)} and the choice of $\eta$ above, we can also assign $(U_{1,\tilde{t}},\ldots,U_{\omega_m,\tilde{t}})$ belonging to some $\tilde{t} \in K$,  where we are able to apply Lemma \ref{freezing lemma}. With a slight abuse of notation, we will denote this tuple by $(U_{1,t_{\mathfrak{r}}},\ldots,U_{\omega_m,t_{\mathfrak{r}}})$.

\vspace{2mm}
\noindent {\bf Step 2: Refinement of the covering.} Our aim is to find a refinement $\{\mathfrak{Q_r(a)}\}, \> \mathfrak{a} \in \{-\frac{2}{5},\frac{2}{5}\}^m$ of the initial covering, such that
\begin{itemize}
\item[(i)] $\mathfrak{Q_r(a)} \subset \mathfrak{Q_r}$ for any $\mathfrak{a}$;\\
\item[(ii)] for every $\mathfrak{r}$ and every $\mathfrak{a}$ there is a choice of $U_{\mathfrak{a},t_{\mathfrak{r}}}$ such that
\begin{itemize}
\item $U'_{\mathfrak{a},t_{\mathfrak{r}}} \in \{U'_{1,t_{\mathfrak{r}}},\ldots, U'_{\omega_m,t_{\mathfrak{r}}}\}$,\\
\item $\dist(U'_{\mathfrak{a},t_{\mathfrak{r}}},U'_{\mathfrak{a'},t'_{\mathfrak{r}}}) >0$ if $\mathfrak{Q_r(a)} \cap \mathfrak{Q_{r'}(a')} \neq \emptyset$;\\
\end{itemize}
\item[(iii)] every point $t \in [0,1]^m$ is contained in at most $2^m$ cubes $\mathfrak{Q_r(a)}$.
\end{itemize}
To do this, we cover each cube $\mathfrak{Q_r}$ with $2^m$ smaller cubes in the following way:
\begin{equation}\label{refined covering}
\left\{Q\Big(\big((2r_1+1)\tilde{\eta} + a_1\eta,\ldots,(2r_m+1)\tilde{\eta}+a_m\eta \big),\frac{3}{5}\eta\Big) \>|\> a_1,\ldots,a_m \in \left\{-\frac{2}{5},\frac{2}{5}\right\} \right\}
\end{equation}
We simplify the notation by setting
\begin{align*}
&\mathfrak{a}:=(a_1,\ldots,a_m) \in \left\{-\frac{2}{5},\frac{2}{5}\right\};\\
&\mathfrak{Q_r(a)}:=Q\Big(\big((2r_1+1)\tilde{\eta} + a_1\eta,\ldots,(2r_m+1)\tilde{\eta}+a_m\eta \big),\frac{3}{5}\eta\Big).
\end{align*}
Note that this choice of the refinement, as well as that of the initial covering, immediately guarantees properties (i) and (iii).\\

\noindent After assigning a family of open sets to each cube of the initial covering in the  previous step, we now want to assign a subfamily to every cube of the refined covering. Consider a cube $\mathfrak{Q_{r_1}(a)} \subset \mathfrak{Q}_{\mathfrak{r}_1}$ of the refinement. Assume that $\mathfrak{Q_{r_1}(a)}$ intersects $1 \leq j \leq 2^m-1$ different cubes of the initial covering, say $\mathfrak{Q}_{\mathfrak{r}_2},\ldots,\mathfrak{Q}_{\mathfrak{r}_j}$, and let
\begin{equation*}
\mathcal{F}_{\mathfrak{r}_1}:=(U'_{1,t_{\mathfrak{r}_1}},\ldots,U'_{\omega_m,t_{\mathfrak{r}_1}}),\ldots, \mathcal{F}_{\mathfrak{r}_j}:=(U'_{1,t_{\mathfrak{r}_j}}\ldots,U'_{\omega_m,t_{\mathfrak{r}_j}})
\end{equation*}
be the corresponding tuples of open sets. Applying Lemma \ref{combinatorial}, we extract subfamilies $\mathcal{F}_{\mathfrak{r}_i}^{sub} \subset \mathcal{F}_{\mathfrak{r}_i}$ for every $i \in \{1,\ldots,j\}$, each containing at least $2^m$ open sets such that
\begin{equation}\label{subfamily}
\dist(U,V)>0 \quad \forall \> U \in \mathcal{F}_{\mathfrak{r}_a}^{sub},\> V \in \mathcal{F}_{\mathfrak{r}_b}^{sub}.
\end{equation}

We then assign to $\mathfrak{Q_{r_1}(a)}$ the subfamily $\mathcal{F}_{\mathfrak{r}_1}^{sub}$, which we now denote by $\mathcal{F}_{\mathfrak{r}}(\mathfrak{a})$. We can do this for every cube in the refinement. By construction, the property \eqref{subfamily} surely holds for each two subfamilies $\mathcal{F}_{\mathfrak{r}_i}(\mathfrak{a}), \mathcal{F}_{\mathfrak{r}_j}(\mathfrak{a}')$ assigned to cubes $\mathfrak{Q}_{\mathfrak{r}_i}(a), \mathfrak{Q}_{\mathfrak{r}_j}(a')$, such that \mbox{$\mathfrak{Q}_{\mathfrak{r}_i}(a) \cap \mathfrak{Q}_{\mathfrak{r}_j}(a') \neq \emptyset$} and $\mathfrak{Q}_{\mathfrak{r}_i} \neq \mathfrak{Q}_{\mathfrak{r}_j}$. On the other hand, the subfamilies assigned to two cubes belonging to the same cube of the initial covering (i.e. $\mathfrak{Q}_{\mathfrak{r}_i} = \mathfrak{Q}_{\mathfrak{r}_j}$), are not necessarily different. Note however, that each subfamily contains at least $2^m$ open sets, and every cube $\mathfrak{Q}_{\mathfrak{r}}$ of the initial covering is covered by exactly $2^m$ cubes of the refinement. Hence we can assign to each of those a distinct open set $U'_{\mathfrak{a},t_{\mathfrak{r}}} \in \mathcal{F}_{\mathfrak{r}}(\mathfrak{a})$.\\
Thus we have a refinement of the covering $\{\mathfrak{Q}_{\mathfrak{r}}(\mathfrak{a})\}$ and corresponding open sets $U'_{\mathfrak{a},t_{\mathfrak{r}}}$ which have all the three properties listed in the beginning of Step 2. Moreover, since $K$ is compact, and $\mathfrak{Q}_{\mathfrak{r}}(\mathfrak{a})=Q(t_{\mathfrak{r}} + \mathfrak{a}\eta,\frac{3}{5}\eta)$ according to \eqref{refined covering}, we can choose a $\delta >0$ such that every $t \in K$ is contained in at least one of the cubes $Q(t_{\mathfrak{r}} + \mathfrak{a}\eta,\frac{3}{5}\eta - \delta)$.\\

For the sake of simplicity, let us now rename the refinement $\{\mathfrak{Q}_{\mathfrak{r}}(\mathfrak{a})\}$ and call it $\{P_\alpha\}$, the corresponding smaller cubes $Q(t_{\mathfrak{r}} + \mathfrak{a}\eta,\frac{3}{5}\eta - \delta)$ we call $P_\alpha^{\delta}$, and the associated open sets we now denote by $U_\alpha$ and $U'_\alpha$. In particular $\alpha$ stands for the pair $(\mathfrak{a},t_{\mathfrak{r}})$. 

\vspace{2mm}
\noindent {\bf Step 3: Conclusion.}
In order to deduce the existence of a family $\{\Gamma_{\alpha,t}\}$ with the properties
\begin{itemize}
\item $\Gamma_{\alpha,t}=\Gamma_t$ if $t \notin P_\alpha$ and $\Gamma_{\alpha,t} \setminus U'_\alpha = \Gamma_t \setminus U'_\alpha$  if  $t \in P_\alpha$;\\
\item $\haus^n(\Gamma_{\alpha,t}) \leq \haus^n(\Gamma_t) + \frac{1}{2^{m+1}N}$ for every $t$;\\
\item $\haus^n(\Gamma_{\alpha,t}) \leq \haus^n(\Gamma_t) - \frac{1}{2N}$ if $t \in P_\alpha^{\delta}$,
 \end{itemize}
 we apply Lemma \ref{freezing lemma} for $\Xi_t = \Gamma_t, U=U_\alpha, U'=U'_\alpha$ and $\Sigma_{\tau} = \Sigma_{i,s,\tau}$, where $(i,s) = \alpha$.\\
 Recall that from the construction of the refined covering $\{P_\alpha\}$ and the choice of $U'_\alpha$ it follows that, if $P_\alpha \cap P_\beta \neq \emptyset$ for $\alpha \neq \beta$, then $\dist(U'_\alpha,U'_\beta) >0$. We can therefore define a new family $\{\Gamma'_t\}_{t \in \P}$ with
 \begin{itemize}
 \item $\Gamma'_t = \Gamma_t$ if $t \notin \cup P_\alpha$;\\
 \item $\Gamma'_t = \Gamma_{\alpha,t}$ if $t$ is contained in in a single $P_\alpha$;\\
 \item $\Gamma'_t = \big[\Gamma_t \setminus (U'_{\alpha_1} \cup \ldots \cup U'_{\alpha_s})\big] \cup \big[\Gamma_{\alpha_1,t} \cap U'_{\alpha_1}\big] \cup \ldots \cup \big[\Gamma_{\alpha_s,t}\cap U'_{\alpha_s} \big]$ if $t \in P_{\alpha_1} \cap \ldots \cap P_{\alpha_s}, \> s \geq 2$.
 \end{itemize}
 This family is clearly homotopic to $\{\Gamma_t\}$ and hence belongs to $X$.\\
 We now want to estimate $\mathcal{F}(\{\Gamma'_t\})$. If $t \notin K$, then $t$ is contained in at most $2^m \> P_{\alpha}$'s and $\Gamma'_t$ can therefore increase at most $2^m \cdot \frac{1}{2^{m+1}N}$ in area:
 \begin{equation}
 t \notin K \implies \haus^n(\Gamma'_t) \leq \haus^n(\Gamma_t) + 2^m \cdot \frac{1}{2^{m+1}N} \leq m_0(X) - \frac{1}{2N}.
 \end{equation}
Note that the last inequality is due to the definition of $K$. If $t \in K$, then $t$ is contained in at least one cube $P_\alpha^{\delta}$ and at most $2^m - 1$ other cubes $P_{\alpha_1},\ldots,P_{\alpha_{2^m-1}}$. Hence the area of $\Gamma'_t$ looses at least $\frac{1}{2N}$ in the first cube and increases at most $\frac{1}{2^{m+1}N}$ in the remaining ones, which are no more than $2^m-1$. Thus,
\begin{equation}
t \in K \implies \haus^n(\Gamma'_t) \leq \haus^n(\Gamma_t) + (2^m-1)\cdot\frac{1}{2^{m+1}N} - \frac{1}{2N} \leq m_0(X) - \frac{1}{2^{m+2}N},
\end{equation}
where the last inequality holds since $\haus^n(\Gamma_t) \leq \mathcal{F}(\{\Gamma_s^N\}_{s \in \P}) \leq m_0(X) + \frac{1}{2^{m+2}N}$ by assumption.\\
From the preceding inequalities we conclude 
\begin{equation*}
\mathcal{F}(\{\Gamma'_t\}) \leq m_0(X) - \frac{1}{2^{m+1}N},
\end{equation*}
which is a contradiction to $m_0(X) = \underset{X}{\inf}\mathcal{F}$. This finishes the proof.
\qed

\section{Boundary behavior of stationary varifolds}\label{s:boundary}

\subsection{Maximum principle} The first important tool which we recall is the following classical maximum principle for the constrained case.

\begin{propos}[Maximum principle]\label{p:White-max}
Let $\M$ be a smooth $(n+1)$-dimensional submanifold satisfying Assumption \ref{a:(C)} and $U\subset \M$ an open set. If $V\in \va_s^c (U, \gamma)$ for some $C^{2,\alpha}$ $(n-1)$-dimensional submanifold $\gamma$ of $\partial \mathcal{M}$ (namely $\delta V (\chi) \geq 0$ for every $\chi \in \mathfrak{X}_c^- (U)$ which vanishes on $\gamma$), then $\supp (V) \cap \partial \M \subset \gamma$. 
\end{propos}

The above proposition is classical if we were to consider $\M$ as a subset of a larger manifold $\tilde{\M}$ without boundary and we had a varifold $V$ which were stationary in $\tilde \M\setminus \gamma$. For a proof we refer the reader to White's paper \cite{white-max}. However it is straightforward to check that the proof in \cite{white-max} works in our setting, since the condition $\delta V (\chi) \geq 0$ for the class of vector fields $\mathfrak{X}^-_c (U\setminus \gamma)$ pointing ``inwards'' is what White really uses in his proof. 

\begin{remark}\label{extr vs. intr}
While one can in principle work with objects defined intrinsically on $\M$, it is often more convenient to embed $\M$ (smoothly) isometrically into some Euclidean space $\RR^{N}$. In fact, by possibly choosing a larger $N$, one can do this so that $\M$ is a compact subset of a closed $(n+1)$-dimensional manifold $\tilde{\M}$.
\end{remark}

As a corollary to Proposition \ref{p:White-max} we obtain the following

\begin{corol}\label{c:allard_possible}
Let $\M$ be a smooth $(n+1)$-dimensional Riemannian manifold isometrically embedded in a Euclidean space $\mathbb R^N$ and satisfying Assumption \ref{a:(C)}. If $U'$ is an open subset of $\mathbb R^N$ and $V$ a varifold in $\va_s^c (U'\cap \M, \gamma)$ for 
some $n-1$-dimensional $C^{2,\alpha}$ submanifold $\gamma$ of $\partial \M$, then { the restricion of $V$ to $U' \setminus \gamma$} has, as a varifold in $U'\setminus \gamma$, bounded generalized mean curvature in the sense of Allard: in particular all the conclusions of Allard's boundary regularity theory in \cite{allard-bdry} are applicable. 
\end{corol}

The proof is straightforward: after viewing $\M$ as a subset of a closed submanifold $\tilde\M$, Proposition \ref{p:White-max} implies the stationarity of $V$ in $\tilde\M\setminus \gamma$ and reduces the statement to a classical computation (see for instance \cite[Remark 16.6(2)]{simon-gmt}).

\subsection{Monotonicity formulae} An important tool in regularity theory for stationary varifolds is the monotonicity formula. For $x \in$ Int($\M$) it says that there exists a constant $\Lambda$ (depending on the ambient Riemannian manifold $\M$, and which is $0$ if the metric is flat, see \cite{simon-gmt}) such that the function
\begin{equation}\label{monotonicity}
f(\rho):=e^{\Lambda \rho}\dfrac{||V||(B_{\rho}(x))}{\omega_n\rho^{n}}
\end{equation}
is non-decreasing for every $x \in \M$ and every $\rho < \text{min}\{\text{Inj}(\M),\text{dist}(x, \partial \M)\}$. A similar conclusion assuming the existence of a "boundary" was reached by Allard \cite{allard-bdry}. { However, in order to apply Allard's conclusion to our case, we need to first show that in our case $\|V\| (\gamma) =0$. This is achieved in the following Lemma.

\begin{lemma}\label{c:gamma_meas_0}
Let $V\in \va^c_s (U, \gamma)$. Then $\|V\| (\gamma) =0$. In particular, a varifold $V$ which is a.m. in annuli for the constrained problem as in Proposition \ref{a.m. in annuli} is integer rectifiable in the whole $\M$.
\end{lemma}}
{
\begin{proof}
We split the varifold $V$ into two parts: $V^r$ is the restriction of $V$ to $G (\M\setminus \gamma)$, and $V^s$ is the ``restriction of $V$ to $\gamma$'', namely $V^s = V - V^r$. We first claim that 
\begin{equation}\label{e:regular_vanishes}
\delta V^r (\chi) =0 \qquad \mbox{for all $\chi\in \mathfrak{X}^t_c (\M)$ which vanish on $\gamma$.}
\end{equation}
First of all recall that $\delta V (\chi) =0$ for any $\chi \in \mathfrak{X}^t_c (\M)$ which vanishes on $\gamma$, because in this case both $\chi$ and $-\chi$ belong to $\mathfrak{X}^- (\M, \gamma)$. Secondly, observe that from the formula for the first variation, namely
\begin{equation}\label{e:first_var} 
\delta V (\chi) = \int {\rm div}_\pi\, \chi (x)\, dV (x, \pi)\, ,
\end{equation}
we conclude easily that $\delta V (\chi) = \delta V^r (\chi)$ whenever $\chi \in \mathfrak{X}_c (\M\setminus \gamma)$. 
Fix therefore a $\chi \in \mathfrak{X}^t_c (\M)$ which vanishes on $\gamma$ and let $\varphi_\delta$ be a family of functions
with the following properties:
\begin{itemize}
\item $\varphi_\delta\in C^\infty_c (\M\setminus \gamma)$;
\item $\varphi_\delta$ is identically equal to $1$ outside the $2\delta$-tubular neighborhood of $\gamma$;
\item $\|\nabla \varphi_\delta\|_0 \leq C \delta^{-1}$, where the constant $C$ is independent of the parameter $\delta$.
\end{itemize}
Note that, since $\chi$ vanishes on $\gamma$, $\|\chi\|_0 \leq C \delta$ in the $2\delta$-tubular neighborhood of $\gamma$. Hence it is straightforward to check that $\|\nabla (\varphi_\delta \chi)\|\leq C$, where $C$ is a constant independent of $\delta$. Hence, the formula for the first variation and the dominated convergence theorem yield
\begin{align*}
\delta V^r (\chi) & = \int {\rm div}_\pi\, \chi (x)\, dV^r (x, \pi)\\
& = \lim_{\delta\downarrow 0} \int {\rm div}_\pi\, (\varphi_\delta \chi) (x)\, dV^r (x, \pi) = 
\lim_{\delta\downarrow 0} \delta V^r (\varphi_\delta \chi) = 0\, .
\end{align*}
This shows \eqref{e:regular_vanishes}, which in turn, recalling that $V^s = V - V^r$, yields
\begin{equation}\label{e:singular_vanishes}
\delta V^s (\chi) =0 \qquad \mbox{for all $\chi\in \mathfrak{X}^t_c (\M)$ which vanish on $\gamma$.}
\end{equation}
Let now $U$ be a sufficiently small tubular neighborhood of $\gamma$ and for each $p\in U\setminus \gamma$ consider the nearest point $q\in \gamma$ and the geodesic segment connecting $p$ and $q$ in $\tilde\M$. We then let $\chi (p)$ be the vector field tangent to such geodesic segment, pointing towards $q$ and with length equal to the geodesic distance of $q$ to $p$. Extend it then to $\gamma$ by setting it $0$ there. $\chi$ is then a smooth vector field on a tubular neighborhood of $\gamma$ inside $\tilde{M}$ and it also has the following property:
\begin{itemize}
\item[(N)] If $e_1, \ldots, e_{n-1}, e_n, e_{n+1}$ is a smooth orthonormal frame defined over $\gamma$ with the property that $e_1, \ldots, e_{n-1}$ are tangent to $\gamma$, then we have $\nabla_{e_1} \chi (q)= \ldots = \nabla_{e_{n-1}} \chi (q) =0$, $\nabla_{e_n} \chi (q) = -e_n$ and $\nabla_{e_{n+1}} \chi (q) = -e_{n+1}$ for any $q\in \gamma$. 
\end{itemize}
Clearly, $\chi$ is not tangent to $\partial \M$. It is however easy to see that if $q\in \partial \M$, then the projection of $\chi (q)$ onto $T_q \partial \M$ is bounded by $C ({\rm dist}\, (q, \gamma))^2$. For this reason $\chi$ can be modified so that:
\begin{itemize}
\item it is tangent to $\partial \M$;
\item it vanishes on $\gamma$;
\item it retains property (N) above. 
\end{itemize}
Moreover, multiplying it by a suitable cut-off function, it can be suitably extended outside a neighborhood of $\gamma$ to the whole manifold $\M$, in order to obtain a globally defined vector field $\chi \in \mathfrak{X}^t (\M)$ which vanishes on $\gamma$. For this reason it is an admissible test for \eqref{e:singular_vanishes}, namely we must have
\begin{equation}\label{e:ci_siamo}
0 = \delta V^s( \chi) = \int {\rm div}_\pi\, \chi (x)\, dV^s (x, \pi)\, .
\end{equation}
On the other hand, the integral on the left hand side takes place for $x\in \gamma$. For any such $x$, fix any 
$n$-dimensional plane $\pi \subset T_x \M$ and recall that, 
\[
{\rm div}_\pi \chi (x) = \sum_{i=1}^n g (\nabla_{f_i} \chi, f_i)\, , 
\]
where $f_1, \ldots , f_n$ is an orthonormal basis for $\pi$. Now, property (N) above ensures that ${\rm div}_\pi \chi (x) \leq -1$. Hence we find $\delta V^s (\chi) \leq - \|V^s\| (\M) = - \|V\| (\gamma)$. Thus \eqref{e:ci_siamo} implies $\|V\| (\gamma)=0$ and concludes our proof. 
\end{proof}
}

{ Lemma \ref{c:gamma_meas_0}} combined with Corollary \ref{c:allard_possible} and with the results in 
\cite{allard-bdry} gives the following 

\begin{propos}\label{p:All-monot}
Consider an open subset $U\subset \M$ and a varifold $V\in \va_s^c (U, \gamma)$ which is a.m. in annuli as in Proposition \ref{a.m. in annuli} in the constrained case (where $\gamma$ is a $C^{2,\alpha}$ submanifold $\gamma$ of $\partial \M$). 
Then, for every $x \in \gamma$ there exists a $\rho_0 > 0$ and a (smooth) function $\Phi(\rho)$ with $\Phi(\rho) \to 0$ as $\rho \to 0$, such that the quantity 
\begin{equation}\label{e:all-monot}
f (\rho) = e^{\Phi (\rho)} \dfrac{||V||(B_{\rho}(x))}{\omega_n\rho^{n}}
\end{equation}
is a monotone non-decreasing function of $\rho$ as long as $0 < \rho < \rho_0$.
\end{propos}

In particular, we conclude that the limit $\dfrac{||V||(B_{\rho}(x))}{\omega_n\rho^{n}}$ exists and it is finite at any point $x\in \gamma$.

The case with free boundaries has been addressed by Gr\"uter and Jost in \cite{Gruter1,gru-jost-1,gru-jost}, who proved a suitable version of the monotonicity formula. The results in these papers were proved in the Euclidean space, but they are easily extendable to the case of stationary varifolds in compact Riemannian manifolds using the embedding trick of Remark \ref{extr vs. intr}.  We summarize the conclusion in the following

\begin{propos}\label{p:GJ-monot}
Assume $\M\subset \tilde\M\subset \mathbb R^N$, where $\tilde \M$ is a closed manifold, let $U\subset \M$ be an open set and $V$ a varifold in $\va_s^u (U)$. Then for each $x \in U$, there exists an $r < \text{dist}(x,\partial U)$, and a constant $c(x,r)$, with $c(x,r) \to 1$ as $r \to 0$, such that
\begin{equation}
\dfrac{||V||(B_{\sigma}(x)) + ||V||(\tilde{B}_{\sigma}(x)) }{\omega_k \sigma^k} \leq c(x,r)\dfrac{||V||(B_{\rho}(x)) + ||V||(\tilde{B}_{\rho}(x)) }{\omega_k \rho^k}
\end{equation}
for all $0< \sigma < \rho < r$. Here, $\tilde{B}_{\sigma}(x)$ denotes the reflection of the ball $B_{\sigma}(x)$ across the boundary $\partial \mathcal{M}$, as defined in \cite[Section 2]{gru-jost-1}.
\end{propos}

Note that, for points in ${\rm Int}\, (U)$ and $r< \text{dist} (x, \partial \M)$, the monotonicity formula of Gr\"uter and Jost reduces to  \eqref{monotonicity}. 
 
\begin{remark}\label{r:Gru-Jost-valid}
Proposition \ref{p:GJ-monot} is indeed proved in \cite[Section 3]{gru-jost-1} under the additional assumption that the varifold $V$ is rectifiable. In our case it is however crucial that their argument can be extended to general varifolds. Indeed, since the monotonicity is derived by testing the stationarity with a suitable vector field, the adaptation of the argument to general varifolds is straightforward and the reader may consult the lecture notes of Leon Simon, more precisely \cite[Chapter 8]{simon-gmt}, where he shows how to adapt to general varifolds the proof of the interior monotonicity formula presented in \cite[Chapter 4]{simon-gmt} under the rectifiability assumption. 
\end{remark}

An important consequence of the monotonicity in all of the above cases is the existence of the {\em density function} of the varifold under consideration:
\begin{equation}\label{e:density}
\Theta(V, x) = \lim_{r \to 0}\dfrac{||V||(B_r(x))}{\omega_n r^n}
\end{equation}
is well defined at all points $x \in U$. Moreover, in the case $V\in \va_s^u$, one can conclude that the function
\begin{equation*}
\tilde{\Theta}(V,x):= \left\{
        \begin{array}{ll}
            \Theta(V, x) & \quad x \in \text{Int}(\M)\cap U\\
            2\Theta(V, x) & \quad x \in \partial \M \cap U
        \end{array}
    \right.
\end{equation*}
is upper semicontinuous in $U$. In the constrained case we conclude instead that the density function is upper semicontinuous in ${\rm Int}\, (U)$ and in $\partial \M\cap U$ {\em separately}.  

\medskip

A direct corollary of Proposition \ref{p:GJ-monot} is then the rectifiability of any varifold $V\in \va_s^u (\M)$ obtained in Proposition \ref{a.m. in annuli}

\begin{corol}\label{c:Gru-Jost-rect}
Let $V\in \mathcal{V}_s^u$ be a varifold which is a.m. in annuli as in Proposition \ref{a.m. in annuli}. Then $V$ is a rectifiable varifold.
\end{corol}
 
\begin{proof} As already remarked, the a.m. property gives the integrality of the varifold in the interior. The monotonicity formula of Proposition \ref{p:GJ-monot} gives that the upper density $\Theta (V,x)$ is everywhere finite, and thus we have $\|V\|\res \partial \M = \Theta \mathcal{H}^n \res \M$ by standard measure theoretic arguments. It remains to show that for $\mathcal{H}^n$-a.e. $x\in \partial \M$ the varifold $V$ has $T_x \partial \M$ as approximate tangent. Arguing as in the proof of \cite[Theorem 8.5.5]{simon-gmt} we know that for $\mathcal{H}^n$-a.e. $x\in \partial \M$ any varifold tangent to $V$ at $x\in \partial M$ is of the form $C = \eta \Theta (V, x) \mathcal{H}^{n-1}\res T_x \partial \M$, where $\eta$ is a probability measure on the Grassmanian $G$ of $n$-dimensional planes of $T_x \M$. We just need to show that $\eta = \delta_{T_x \partial \M}$: as argued in the proof of \cite[Theorem 8.5.5]{simon-gmt} this would imply the rectifiability of 
$V$. 

First of all, by standard arguments, the fact that $V\in \mathcal{V}^u_c$ implies that $\delta C (\chi) =0$ whenever $\chi$ is a vector field which is tangent to $T_x \partial \M$. 
In particular we can conclude that $\eta = \delta_{T_x \partial \M}$ following the argument used in a similar situation in the proof of the Constancy Theorem in \cite[Chapter 8]{simon-gmt}. Notice that the proof in there is achieved by testing the first variation condition with a vector field of the form $\chi = f \nabla f$, where $f$ is a function vanishing on $T_x \partial \M$: in particular in our case $\delta C (\chi) =0$ because $\chi$ actually vanishes on $T_x \partial \M$. 
\end{proof}
 
Finally, we record here a simple consequence of the argument in \cite{gru-jost-1} proving Proposition \ref{p:GJ-monot}, which has a crucial role in a later section.

\begin{lemma}\label{l:GJ-again}
Let $S$ be an $n$-dimensional varifold in $\{x\in \mathbb R^{n+1} : x_1 \leq 0\}$ such that:
\begin{itemize}
\item $\delta S (\chi) =0$ for every vector field which is tangent to $\{x_1=0\}$;
\item $\rho^{-n} \|S\| (B_\rho (0)) = r^{-n} \|S\| (B_r (0))$ for two distinct radii $\rho < r$.
\end{itemize}
Then $s^{-n} \|S\| (B_s (0)) = \rho^{-n} \|S\| (B_\rho (0))$ for every $s\in [\rho, r]$. 
\end{lemma}

\subsection{Blow-up and tangent cones} In this section we recall the usual ``rescaling'' procedure which allows to blow-up minimal surfaces at a given point.
Following Remark \ref{extr vs. intr}, we adhere to the standard procedure of first embedding the Riemannian manifold $\M$ into $\RR^{N}$.  We will use the term $n+1$-dimensional wedge of opening angle $\theta\in ]0, \frac{\pi}{2}[$ for any closed subset $W$ of the form $R (W_0)$, where $R\in SO (n+1)$ is an orientation-preserving isometry of $\mathbb R^{n+1}$ and we recall that
that $W_0$ is the canonical wedge with opening angle $\theta$, namely the set 
\[
\{(x_1, \ldots , x_{n+1})\in \mathbb R^{n+1} : |x_{n+1}| \leq x_1 \tan \theta\}\, .
\]
The half-hyperplane $R (\{x_{n+1} = 0, x_1 > 0\})$ will be called the {\em axis} of the wedge and the $n-1$-dimensional plane $\ell := R (\{x_{n+1} = x_1 =0\})$ will be called the {\em tip of the wedge}. As stated above, when $W= W_0$, we call it the {\em canonical wedge with opening angle $\theta$.}

\begin{definition}\label{meeting with an angle}
Let $\M$ be a smooth $(n+1)$-dimensional manifold with boundary satisfying Assumption \ref{a:(C)} and $\gamma$ a $C^{2,\alpha}$ $(n-1)$-dimensional submanifold of $\partial\M$. We say that a closed set $K\subset \M$ meets $\partial \M$ in $\gamma$ with opening angle at most $\theta$ if the following holds:
\begin{itemize}
\item $\gamma = K \cap \partial \M$;
\item for any $x\in \gamma$, let $\tau \in T_x \partial \M$ be a unit vector orthogonal to $T_x \gamma$ and $\nu\in T_x \M$ be the unit vector orthogonal to $T_x \partial\M$ and pointing inwards; then for every $C^1$ curve $\sigma : [0,1]\to K$, with $\sigma (0)=x$ and parameterized by arc length, we have 
\begin{equation}
|\langle \dot\sigma (0), \tau \rangle|\leq  \langle \dot{\sigma} (0), \nu \rangle\tan \theta\, .
\end{equation}
\end{itemize}
\end{definition} 

\begin{figure}[htbp]
\begin{center}
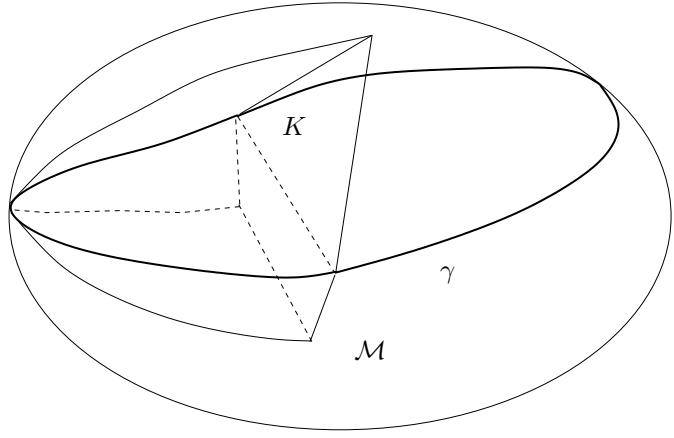
\end{center}
\caption{A set $K$ meeting $\gamma$ at some angle at most $\theta< \frac{\pi}{2}$.}
\end{figure}

We are now ready to state the blow-up procedure which we will use in the rest of the note, especially at boundary points. Recall that, for $x\in \partial \M$, $\nu$ is the unit vector of $T_x \M$ orthogonal to $T_x \partial \M$ and pointing inwards. 

\begin{lemma}\label{l:bu}
Let $\M\subset \mathbb R^N$ be a smooth Riemannian manifold satisfying Assumption \ref{a:(C)}, $U\subset \M$ an open set and $V$ a rectifiable varifold which is stationary in ${\rm Int}\, (U)$. Given a point $x\in \spt (V)\subset \M$ we introduce the map $\iota_{x,r} : \RR^N \to \RR^N$ defined by $\iota_{x,r} (y) := (y-x)/r$ and let $\M_{x,r}:= \iota_{x,r} (\M)$ and $V_{x,r} := (\iota_{x,r})_\sharp V$.

\begin{itemize}
\item[(I)] If $x\in{\rm Int}\, (U)$, then $\M_{x,r}$ converges, as $r \to 0$, locally in the Hausdorff sense, to $T_x \M$ (which is identified with the corresponding linear subspace of $\RR^N$). If $V$ is integral, then up to subsequences $V_{x,r}$ converges, in the sense of varifolds, to a stationary varifold $S$ which is integral and is a cone.
\item[(B)] If $x\in \partial \M$, then $\M_{x,r}$ converges, locally in the Hausdorff sense, to $T^+_x \M := T_x\M \cap \{y: \nu \cdot y \geq 0\}$. If $V$ is integral and belongs to $\mathcal{V}^c_s (\gamma)$, then $V_{x,r}$ converges, in the sense of varifolds, to an integral varifold $S$ which  a cone, it is supported in $T^+_x \M$ and it is stationary in $T_x \M \cap \{y: y\cdot \nu >0\}$. 
\item[(W)] If $x\in \partial \M$, $V$ is as in (B) and $\spt (V)$ is contained in a closed $K$ which meets $\partial \M$ at a $C^{2,\alpha}$ submanifold $\gamma$ with opening angle at most $\theta< \frac{\pi}{2}$, then each such $S$ (as in statement (B)) is supported in the wedge $W\subset T_x \M$ of opening angle $\theta$ with tip $T_x \gamma$ and axis orthogonal to $T_x \partial \mathcal{M}$.
\end{itemize}
\end{lemma}

The lemma is a straightforward consequence of the theory of varifolds developed in \cite{allard-bdry}. 

\begin{definition}\label{d:tan_cones}
At every point $x$ where $\Theta (V, x) < \infty$ we denote by $\Tan (x, V)$ the set of varifolds $W$ which are limits of subsequences (with $r_k\downarrow 0$) of $\{V_{x,r}\}_r$ and which will be called {\em tangent varifolds to $V$ at $x$}. 
If a tangent varifold is a cone, then it will be called {\em tangent cone}.
\end{definition}

\begin{remark}\label{r:coni_non_banale}
We observe moreover that, when $V$ is stationary and a $x$ a point where it satisfies the monotonicity formula, then $W$ is stationary and
\begin{equation}\label{e:density2}
\Theta (V, x) = \Theta (W,0) = \frac{\|W\| (B_r (0))}{\omega_n r^n} \qquad \forall W\in \Tan (x, V), \forall r>0\, .
\end{equation}
In order to conclude that $W$ is a cone one needs however some additional information. The rectifiability of the varifold is enough in the interior, cf. \cite[Chapter 8]{simon-gmt}.
\end{remark}

\subsection{White's curvature estimate at the boundary}\label{ss:white_est} We close this section by introducing the most
important tool in the boundary regularity theory which we will develop in the sequel. The tool is a suitable curvature estimate at the boundary, suggested to us by Brian White, which is valid for  stable smooth hypersurfaces constrained in a wedge. A varifold will be called stable (in an open set $U$) if the second variation $\delta^2 V$ is nonnegative when evaluated at every vector field compactly supported in ${\rm Int}\, (U)$. Strict stability will mean that the second variation is actually strictly positive, except for the trivial situation where the vector field vanishes everywhere on the support of the varifold.

\begin{theorem}\label{t:boundary_est}
Let $\M$ be an $(n+1)$-dimensional smooth Riemannian manifold satisfying Assumption \ref{a:(C)}, $\gamma \subset \partial \M$ a $C^{2,\alpha}$ submanifold of $\partial \M$ and $r\in ]0,1[$. Denote by $D$ the inverse of the distance between the closest pair of points in $\gamma$ which belong to distinct connected components; if there is a single connected component, set $D=0$. For every $M >0$ and $\theta \in [0, \pi[$ there are positive constants $C (D, M, \M, \gamma, n, \theta)$ and $\delta (D, M, \M, \gamma, n, \theta)$ with the following property: Assume that
\begin{itemize}
\item[(CE1)] $x_0\in \gamma$ and $\Sigma$ is a stable, minimal hypersurface in $B_{2r} (x_0)$ such that:
\begin{itemize}
\item $\haus^n (\Sigma) \leq M r^{n}$, $\partial \Sigma \subset \partial B_{2r} (x_0) \cup \partial \M$ and
$\partial \Sigma \cap\partial\M = \gamma$;
\item $\Sigma$ is $C^1$ apart from a closed set $\sing (\Sigma)$ with $\haus^{n-2} (\sing (\Sigma)) =0$ and \\ $\gamma \cap \sing (\Sigma) = \emptyset$;
\item $\Sigma$ is contained in a closed set $K$ meeting $\partial \M$ in $\gamma$ with opening angle at most
$\theta$.
\end{itemize}
\end{itemize}
Then $\Sigma$ is $C^{2,\alpha}$ in $B_{\delta r} (x_0)$ and 
\begin{equation}\label{e:bounded_curv_est}
|A|\leq C r^{-1} \qquad \mbox{in $B_{\delta r} (x_0)$.}
\end{equation}
Furthermore, $\Sigma\cap B_{\delta r} (x_0)$ consists of a single connected component. 
\end{theorem}

The proof requires two elementary but important lemmas, which we state immediately.

\begin{lemma}\label{l:tangent_cones_wedge}
Let $V$ be an integer $n$-dimensional rectifiable varifold in $\RR^{n+1}$ such that
\begin{itemize}
\item[(a)] $V$ is stationary in a wedge $W_0$ of opening angle $\theta$;
\item[(b)] $\delta V = (w_1,w_2,0,\ldots,0) \haus^{n-1} \res \ell$ for some Borel vector field\\ $w = (w_1,w_2) \in L^1_{loc} (\haus^{n-1} \res \ell; \RR^2)$.
\end{itemize}
Then for $\haus^{n-1}$-a.e. $x\in \ell$ we have the representation
\[
w (x) = \sum_{i=1}^m v_i (x)
\]
where
\begin{itemize}
\item $m = 2 \Theta (x, V)$;
\item each $v_i$ is of the form $(- \cos \theta_i, -\sin \theta_i)$ for some $\theta_i\in [-\theta, \theta]$.
\end{itemize}
\end{lemma}

\begin{lemma}\label{l:simple_geom_obs}
Let $k\in \mathbb N \setminus \{0\}$ and $v_i=(\cos \theta_i, \sin \theta_i)$ $2k+1$ unit vectors in the plane with $-\frac{\pi}{2} <\theta_i < \frac{\pi}{2}$. Then the
sum $v_1 + \ldots + v_{2k+1}$ has length strictly larger than $1$. 
\end{lemma}

The simple proofs of the lemmata will be postponed to the end of the section, while we first deal with the proof of the main Theorem (given the two lemmata).

\begin{proof}[ \bf \emph{Proof of Theorem \ref{t:boundary_est}}]
We will in fact prove that the constants $\delta$ and $C$ depend on the $C^{2,\alpha}$ regularity of $\gamma$, $\M$ and
$\partial \M$. First of all we focus on the curvature estimate. 

Without loss of generality, we again assume that $\M$ is isometrically embedded in a euclidean space $\mathbb R^N$. Observe that the dimension $N$ can be estimated by $n$ and thus we can assume that $N$ is some
fixed number, depending only on $n$. Upon rescaling we can also assume that $r=1$: the rescaling would just lower the $C^{2,\alpha}$ norm of $\M$, $\partial\M$ and $\gamma$ and increase the distance $D^{-1}$ between different connected components of $\gamma$. 

Assuming by contradiction that the statement does not hold, we would find a sequence of manifolds $\M_k$, boundaries $\gamma_k$, minimal surfaces $\Sigma_k$ and points $p_k\in \Sigma_k$ with the properties that:
\begin{itemize} 
\item $|A_{\Sigma_k}| (p_k) \uparrow \infty$, or $p_k$ is a singular point, and the distance between $p_k$ and $\gamma_k$ converges to $0$
\item $\M_k, \Sigma_k$ and $\gamma_k$ satisfy the assumptions of the Theorem with $r=1$, with a uniform bound on
the $C^{2,\alpha}$ regularities of both $\gamma_k$ and $\M_k$ and with a uniform bound on $M$ and $\theta$.
\end{itemize}
We let $q_k\in \gamma_k$ be the closest point to $p_k$ and, w.l.o.g. we translate the surfaces so that $q_k =0$.
We next rescale them by a factor $\rho_k^{-1}$ where $\rho_k$ is the maximum between $|p_k|$ and $|A_{\Sigma_k} (p_k)|^{-1}$ (where
we understand the latter quantity to be $0$ if $p_k$ is a singular point).
We denote by $\bar\gamma_k$, $\bar\M_k$, $\bar \Sigma_k$ and $\bar p_k$ the corresponding rescaled objects.
It turns out that, up to subsequences,
\begin{itemize}
\item[(a)] the rescaled manifolds $\bar\M_k$ are converging, locally in $C^{2,\alpha}$, to a half $(n+1)$-dimensional plane, that w.l.o.g. we can assume to $\RR^{n+1}_+ = \{x: x_{n+2} = \ldots = x_N =0,\, x_1\geq 0\}$;
\item[(b)] the rescaled manifolds $\partial\bar\M_k$ are converging, locally in $C^{2,\alpha}$, to an $n$-dimensional plane, namely $\{x: x_1 = x_{n+2} = \ldots = x_N = 0\}$;
\item[(c)] the rescaled surfaces $\bar\gamma_k$ are converging, locally in $C^{2,\alpha}$, to an $n-1$-dimensional plane, that w.l.o.g. we can assume to be $\ell := \{x_1=x_{n+1} = x_{n+2} = \ldots = x_N=0\}$;
\item[(d)] the points $\bar p_k$ are converging to some point $\bar p$ and $\liminf_k |A_{\bar\Sigma_k}| > 0$;
\item[(e)] the surfaces $\bar{\Sigma}_k$ are converging, in the sense of varifolds, to an integral varifold $V$, which is supported in the standard wedge $W$ contained in $\RR^{n+1}_+$ with tip $\ell$, axis $\pi_+ = \{x_1 > 0, x_{n+1}=x_{n+2} = \ldots = x_N =0\}$;
\item[(f)] the integral varifold $V$ is stationary inside $W\setminus \ell$ and in fact
$|\delta V|\leq \haus^{n-1} \res \ell$.
\end{itemize}
All these statements are simple consequences of elementary considerations and of the theory of varifolds. For (f), observe that $|\delta \llbracket\bar\Sigma_k\rrbracket| \leq \haus^{n-1} \res \gamma_k + \|A_{\M_k}\|_{C^0} \haus^n \res \Sigma_k$ and use the semicontinuity of the total variation of the first variations under varifold convergence.

\medskip

We next show that the varifold $V$ is necessarily half of an $n$-dimensional plane $\tau$ bounded by $\ell$ and lying in $W$. This would imply, by Allard's regularity theorem, that the surfaces $\Sigma_k$ are in fact converging in $C^{2,\alpha}$ to 
$\tau$, contradicting (d).

\medskip

We first start to show that the density $2\Theta (V,x)$ is odd at $\haus^{n-1}$ a.e. $p\in \ell$. By White's stratification theorem, see Theorem 5 of White \cite{white_strata}, at $\haus^{n-1}$-a.e. point $x\in \ell$ there is a tangent cone
$V_\infty$ to $V$ which is invariant under translations along $\ell$. This implies that $V_\infty$ is necessarily given by
\[
\sum_{i=1}^m \llbracket \pi_i \cap W_0\rrbracket
\]
for some family of $n$-dimensional planes (possibly with repetitions) containing $\ell$, where 
\[
m = m(x) = 2 \Theta (V_\infty, 0) = 2 \Theta (V,x)\, .
\]
Observe that any such halfplane $\pi_i\cap W_0$ is contained in the wedge $W_0$. Without loss of generality we can assume that each $\pi_i\cap W_0$ makes an angle smaller than $\frac{\pi}{2}$ with $\pi_1\cap W_0$

Let $B\subset \pi_1\cap W_0$ be a compact connected set not intersecting $\ell$. By a simple diagonal argument, $V_\infty$ is also the limit of an appropriate sequence of rescalings of the surfaces $\bar\Sigma_k$, namely
$(\bar\Sigma_{k(j)})_{0, r_j}$. If $k(j)$ converges to infinity sufficiently fast, we keep the convergence conclusions in (a), (b), (c), (e) and (f) even when we replace $\bar\Sigma_k$, $\bar \M_k$, $\partial \bar\M_k$ and $\bar\gamma_k$ with the corresponding rescalings $(\bar\Sigma_{k(j)})_{0, r_j}$, $(\bar\M_{k(j)})_{0, r_j}$, $(\partial\bar\M_{k(j)})_{0, r_j}$
and $(\bar\gamma_{k(j)})_{0, r_j}$. For notational simplicity, let us keep the label $\bar\Sigma_k$ even for the rescaled surfaces.

Note that, since the support of the varifold $V_\infty$ is a finite number of affine graphs over the set $B$ (and $m$ is the sum of the multiplicities, including the one of $\pi_1$), the Schoen-Simon theorem implies smooth convergence of the $\bar\Sigma_k$. Thus the $\bar\Sigma_k$ will also be a union of $m$ graphs over $B$ (distinct, because the $\Sigma_k$ are surfaces with multiplicity $1$). Let $\varkappa = \pi_1^\perp$ and, after giving compatible orientations to $\pi_1$ and $\varkappa$, for every $x\in B$ where $\varkappa +x$ intersects $\bar{\Sigma}_k$ transversally, 
we define the degree
\[
d (x) := \sum_{y\in \varkappa \cap \bar{\Sigma}_k} \varepsilon (T_y \bar\Sigma_k, \varkappa),
\]
where $\varepsilon (T_y \bar\Sigma_k, \varkappa)$ takes, respectively, the value $1$ or $-1$ according to whether the two transversal planes have compatible or non-compatible orientation. For $k$ large enough $\gamma_k$ does not intersect $B+ \varkappa$ and thus $d$ is constant on $B$. Moreover, it turns out that $d$ is either $1$ or $-1$. To see this, one can for instance consider $\bar\Sigma_k$ as integral currents and project them onto $\pi_1$. Due to (c), for $k$ large enough (and inside some large ball around the origin), the projection of $\bar\gamma_k$'s will have multiplicity one, and since the projection and the boundary operator commute, the projection of $\bar\Sigma_k$'s onto $\pi_1$ inside $B$ will be simply $\pm\llbracket \pi_1 \rrbracket \res B$. Thus the number of intersections of $y+\varkappa$ with $\bar{\Sigma_k}$ must be odd for a.e. $y\in B$. This obviously implies that $m$ is odd.

\medskip

We next infer that $2\Theta (V,x)$ must be $1$ at $\haus^{n-1}$-a.e. $x\in \ell$. Apply indeed Lemma \ref{l:tangent_cones_wedge} and, using the Borel maps $w$ and $v_i$ defined in there, consider the Borel function
\[
f (x) := |w(x)| = \left|\sum v_i (x)\right|\, .
\] 
We then have $|\delta V |= f \haus^{n-1}\res \ell$ and from 
Lemma \ref{l:simple_geom_obs}, we conclude
that $f (x) > 1$ at every point $x$ where $2\Theta (V,x)$ is an odd number larger than $1$. Since by the previous step such number is odd a.e., we infer our claim by using item (f) from above.
By Allard's regularity theorem, any point $x$ as above (i.e. where there is at least one tangent cone invariant under translations along $\ell$) is then a regular point.

\medskip

Hence, it turns out that
\begin{itemize}
\item the set of interior singular points of $V$ has Hausdorff dimension at most $n-7$, by the Schoen-Simon compactness theorem;
\item the set of boundary singular points has Hausdorff dimension at most $n-2$.
\end{itemize}
Consequently, there is only one connected component of the regular set of $V$ whose closure contains $\ell$. Thus there cannot be any other connected component, because its closure would not touch $\ell$ and would give a stationary varifold contained in the wedge $W$, violating the maximum principle. Hence we infer that any interior regular point of $V$ can be connected with a curve of regular points to a regular boundary point. In turn this implies that the varifold $V$ has density $1$ at every regular point. So $V$ can be given the structure of a current and in particular we conclude that the $\Sigma_k$'s are converging to $V$ as a current.

\medskip

Consider next that, 
\[
\lim_{R\uparrow \infty} \frac{\|V\| (B_R (0))}{R^{n}}
\]
is bounded uniformly, depending only on the constant $M$. Thus, by the usual monotonicity formula, there is a sequence
$R_k\to \infty$ such that $V_{0, R_K}$ converges to a cone $V_\infty$ stationary in $W\setminus \ell$. Again, by a diagonal argument, $V_\infty$ is also the limit of a sequence of rescalings $(\bar\Sigma_{k(j)})_{0, R_j}$, and if $k(j)$ converges to infinity sufficiently fast, we retain the conclusions in (a), (b), (c), (e) and (f) when we replace $\bar\Sigma_k$, $\bar \M_k$, $\partial \bar\M_k$ and $\bar\gamma_k$ with the corresponding rescalings $(\bar\Sigma_{k(j)})_{0, R_j}$, $(\bar\M_{k(j)})_{0, R_j}$, $(\partial\bar\M_{k(j)})_{0, R_j}$
and $(\bar\gamma_{k(j)})_{0, R_j}$.

\medskip

All the conclusions inferred above for $V$ are then valid for $V_\infty$ as well, namely:
$V_\infty$ has multiplicity $1$ a.e., it can be given the structure of a current and the surfaces $(\bar\Sigma_{k(j)})_{0, R_j}$ are converging to it in the sense of currents. In particular the boundary of $V_\infty$ (as a current) is given by $\ell$ (with the appropriate orientation).
We can then argue as in \cite[Lemma 5.2]{allard-bdry} to conclude that the current $V_\infty$ is in fact the union of finitely many half-hyperplanes meeting at $\ell$. But since $\ell$ has many regular points, where the multiplicity must be $\frac{1}{2}$, we conclude that indeed $V_\infty$ consists of a single plane.
 
\medskip

 In particular we infer from the argument above that $\Theta (V_\infty, 0) = \frac{1}{2}$.  This in turn implies
\[
\lim_{R\uparrow \infty} \frac{\|V\| (B_R (0))}{R^{n}} = \frac{\omega_n}{2}\, .
\]
On the other hand
\[
\lim_{r\downarrow 0} \frac{\|V\| (B_r (0))}{r^n} = \omega_n \Theta (V, 0) \, .
\]
But the upper semicontinuity of the density and the fact that $\Theta (V,x) = \frac{1}{2}$ for $\mathcal{H}^{n-1}$-a.e. $x\in \ell$ implies implies that $\Theta (V, 0) \geq \frac{1}{2}$.

Since $\ell$ is flat, Allard's monotonicity formula implies that
\[
r\mapsto  \frac{\|V\| (B_r (0))}{r^n}
\]
is monotone and thus constant. Again the monotonicity formula implies that such function is constant if and only if $V$ is itself a cone. This means that $V$ coincides with $V_\infty$ and is half of a hyperplane, as desired.

\medskip

We now come to the claim that, choosing $\delta$ possibly smaller, the surface $\Sigma$ has a single connected component in $B_{\delta r} (x_0)$. Again this is achieved by a blow-up argument.
Given the estimate on the curvature, for every sufficiently small $\eta$ we have that $x_0$ belongs to a connected component of $\Sigma$ which is the graph of a function $f$ for some given system of coordinates in $B_{2\eta} (x_0)$. Let us denote by $\Gamma$ such a connected component. For $\eta$ small we can assume that the tangent to $\Gamma$ is as close to $T_{x_0} \Gamma$ as we desire and thus we can assume that the connected component is actually a graph of a function $f: T_{x_0} \Gamma \to T_{x_0} \Gamma^\perp$, with gradient smaller than some $\varepsilon>0$, whose choice we specify in a moment. From now on in all our discussion we assume to work in normal coordinates based at $x_0$. In fact it is convenient to consider a closed manifold $\tilde{\M}$ which contains $\partial \M$ and
from now on we let $\tilde{B}_r (x_0)$ be the corresponding geodesic balls.

Assume now, by contradiction, that $B_{\delta \eta} (x_0)$ contains another point $y_0\in \Sigma$ which does not belong to $\Gamma$, where $\delta$ is a small parameter, depending on the maximal opening angle $\theta$ with which the set $K$ can meet $\partial\M$. Thus $y_0$ belongs to a second connected component $\Gamma'$. By the curvature estimates we can assume that $\Gamma'$ as well is graphical and more precisely it is a graph over some plane $\pi$ of a function $g$ with gradient smaller than $\varepsilon$ and height smaller than $\varepsilon \eta$. Moreover, without loss of generality, we can assume that $\pi$ passes through the point $y_0$. 

\medskip

Observe that by assumption (CE1) $\Gamma'$ cannot intersect $\partial \M$, hence any point in $\partial \Gamma'$ is at distance $2\eta$ from $x_0$.
Since $\|g\|_0 \leq \varepsilon \eta$, it turns out that any point $z_0\in \pi \cap \tilde{B}_{(2-2\varepsilon) \eta} (x_0)$ must be in the domain of $g$,
which we denote by ${\rm Dom}\, (g)$. To see this observe first that, since $\pi \cap \tilde{B}_{(2-2\varepsilon) \eta} (x_0)$ is convex, we can join $y_0$ and $z_0$ with a path $\gamma$ lying
in $\pi \cap \tilde{B}_{(2-2\varepsilon)\eta} (x_0)$. Assume that $\gamma$ is parametrized over $[0,1]$ and that $\gamma (0) = y_0$. For a small $\varepsilon$ we know that $\gamma ([0, \varepsilon])\subset {\rm Dom}\, (g)$. If $\gamma (1)\in {\rm Dom}\, (g)$ we are finished. Otherwise we let $\tau$ be the infimum of $\{t: \gamma (t)\not \in {\rm Dom}\, (g)\}$. Obviously the point $p$ in the closure of the graph of $g$ lying over $\gamma (\tau)$ is a boundary point for $\Gamma'$. On the other hand, since $\gamma (\tau)\in \tilde{B}_{(2-2\varepsilon)\eta} (x_0)$ and
$\|g\|\leq \varepsilon \eta$, clearly $p$ cannot be at distance $2\eta$ from $x_0$. This is a contradiction and thus we have proved the 
conclusion
\[
\pi \cap \tilde{B}_{(2-2\varepsilon) \eta} (x_0) \subset {\rm Dom}\, (g)\, .
\]
In particular we conclude that $\pi \cap \tilde{B}_{(2-4\varepsilon) \eta} (x_0)$ cannot meet $\partial \M$: if the intersection were not empty, then there would be a point $q$ contained in $\pi \cap (\tilde{B}_{(2-2\varepsilon) \eta} (x_0)\setminus \M) $ which lies at distance at least $\frac{3}{2} \varepsilon \eta$ from $\partial \M$. In particular the point of the graph of $g$ lying on top of $q$ could not belong to $\M$, although it would be a point of $\Gamma'$. 

By a similar argument, we conclude that $\pi \cap \tilde{B}_{(2-6\varepsilon) \eta} (x_0)$ cannot intersect $T_{x_0} \Sigma$, otherwise we would have nonempty intersection between the graphs of $f$ and $g$, i.e. a point belonging to $\Gamma \cap \Gamma'$, which we know to be different connected components of $\Sigma \cap B_{2\eta} (x_0)$, hence disjoint. 

At this point we choose $\varepsilon = \frac{1}{12}$. Summarizing, the plane $\pi$ has the following properties:
\begin{itemize}
\item[(a)] $\pi$ contains a point $y_0 \in B_{\delta \eta} (x_0)$;
\item[(b)] $\pi$ does not intersect $\partial \M \cap \tilde{B}_{3 \eta/2}$;
\item[(c)] $\pi$ does not intersect $T_{x_0} \Sigma \cap \tilde{B}_{3\eta/2}$;
\item[(d)] $T_{x_0} \Sigma$ meets $T_{x_0} \partial \M$ at an opening angle at most $\theta$. 
\end{itemize}
It is now a simple geometric property that, if $\delta$ is chosen sufficiently small compared to $\theta$, then the plane $\pi$ cannot
exist, cf. Figure \ref{f:bad_plane}.

\begin{figure}[htbp]
\begin{center}
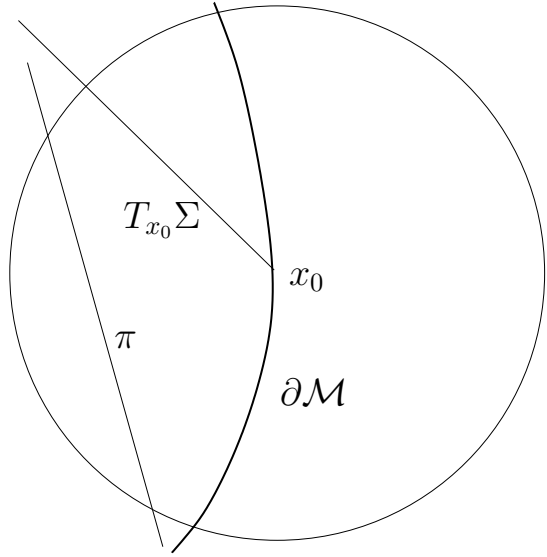
\end{center}
\caption{If two planes $\pi$ and $T_{x_0} \Sigma$ satisfy the assumption (b), (c) and (d), then $\pi$ cannot contain a point which is $\delta \eta$ close to $x_0$.}\label{f:bad_plane}
\end{figure}

\end{proof}

\begin{proof}[\bf \emph{Proof of Lemma \ref{l:tangent_cones_wedge}}]
By White's stratification theorem, see Theorem 5 of White \cite{white_strata}, at $\haus^{n-1}$-a.e. point $x\in \ell$ there is a tangent cone
$V_\infty$ to $V$ which is invariant under translations along $\ell$. This implies that $V_\infty$ is necessarily given by
\[
\sum_{i=1}^m \llbracket \pi_i \cap W_0\rrbracket
\]
for some family of $n$-dimensional planes containing $\ell$, where $m = m(x) = 2 \Theta (V_\infty, 0) = 2 \Theta (V,x)$. 
It is therefore obvious that
\[
\delta V_\infty = \sum_{i=1}^m v_i \haus^{n-1} \res \ell
\]
where each $v_i = v_i (x)$ is the unit vector contained in $\pi_i\setminus W_0$ and orthogonal to $\ell$: therefore for each $i$ we have $v_i =(- \cos \theta_i, -\sin \theta_i)$ for some $\theta_i\in [-\theta, \theta]$. 

Let $r_k\downarrow 0$ be a sequence such that the rescaled varifolds $V_{x, r_k}$ converge weakly to $V_\infty$. Then
$\delta V_{x,r}$ converges to $\delta V_\infty$ in the sense that $\delta V_{x,r_k} (\varphi) \to \delta V_\infty (\varphi)$ for any smooth
compactly supported vector field on $\mathbb R^{n+1}$. On the other hand for a.e. $x$ we have 
$\delta V_{x,r} \weaks w(x) \haus^{n-1} \res \ell$. This completes the proof since for a.e. $x$ we must have $w (x) = \sum_i^{m (x)} v_i (x)$.
\end{proof}

\begin{proof}[\bf \emph{Proof of Lemma \ref{l:simple_geom_obs}}]
We order the vectors so that $\theta_1\leq \theta_2 \leq \ldots \leq \theta_{2k+1}$. For each $i\leq k$, the sum $w_i$ of
the pair $v_i + v_{2k+2-i}$ is a positive multiple of 
\[
\left(\cos \textstyle{\frac{\theta_i+\theta_{2k+2-i}}{2}}, \sin \textstyle{\frac{\theta_i+\theta_{2k+2-i}}{2}}\right)\, .
\]
Since $\theta_i \leq \theta_{k+1} \leq \theta_{2k+2-i}$, it is easy to see that the vectors $w_i$ and $v_{k+1}$ form an angle strictly smaller than $\frac{\pi}{2}$. We therefore have $\langle w_i, v_{k+1}\rangle >0$ and we can estimate
\begin{align*}
| v_1 + \ldots + v_{2k+1}| &\geq \sum_{j=1}^{2k+1} \langle v_j, v_{k+1}\rangle = 1 + \sum_{i=1}^k \langle w_i, v_{k+1}\rangle >1 \, .
\end{align*}
\end{proof}

\section{Stability and compactness}\label{s:stability}

 Since the ground-breaking works of Schoen \cite{schoen}, Schoen-Simon-Yau \cite{s-s-y} and Schoen-Simon \cite{SS}, it is known that, roughly speaking, all the smoothness and compactness results which are valid for hypersurfaces (resp. integer rectifiable hypercurrents) which minimize the area are also valid (in the form of suitable \textsl{a priori} estimates) for stable hypersurfaces.

\subsection{Interior compactness and regularity} We recall here the fundamental compactness/regularity theorem of Schoen and Simon (cf. \cite{SS}) for stable minimal surfaces.

\begin{theorem}\label{t:SS}
Let $\{\Sigma^k\}$ be a sequence of stable minimal hypersurfaces in some open subset $U\subset \M\setminus \partial \M$ and assume that
\begin{itemize}
\item[(i)] each $\Sigma^k$ is smooth except for a closed set of vanishing $\haus^{n-2}$-measure;
\item[(ii)] $\Sigma^k$ has no boundary in $U$;
\item[(iii)] $\sup_k \haus^n (\Sigma^k) <\infty$.
\end{itemize}
Then a subsequence of $\{\Sigma^k\}$ (not relabeled) converges, in the sense of varifolds, to an integer rectifiable varifold $V$ such that
\begin{itemize}
\item[(a)] $V$ is, up to multiplicity, a stable minimal hypersurface $\Gamma$ with ${\rm dim}\, ({\rm Sing}\, (\Gamma))\leq n-7$;
\item[(b)] at any point $p\not \in {\rm Sing}\, (\Gamma)$ the convergence is smooth, namely there is a neighborhood $U'$ of $p$ such that, for $k$ large enough, $\Sigma^k\cap U'$ can be written as the union of $N$ distinct smooth graphs over (the normal bundle of ) $\Gamma \cap U'$, converging smoothly (where the number $N$ is uniformly controlled by virtue of (iii)). 
\end{itemize}
\end{theorem}

In fact the Theorem of Schoen and Simon gives a more quantitative version of the smooth convergence, since for every point $p\not\in {\rm Sing}\, (\Gamma)$ the second fundamental form of $\Sigma^k$ at $p$ can be bounded, for $k$ large enough, by 
$C \max \{\dist (p, {\rm Sing}\, (\Gamma))^{-1}, \dist (p, \partial U)^{-1}\}$, where the constant $C$ is independent of $k$.

\subsection{Boundary version for free boundary surfaces} In \cite{gru-jost} the fundamental result of Schoen and Simon has been extended to the case of free boundary minimal surfaces, under a suitable convexity assumption in the case of $n=2$ in the Euclidean case. However, it can be readily checked that the arguments presented in \cite{gru-jost} to adapt the proof of Schoen and Simon in \cite{SS} to the free boundary case are independent both of the dimensional assumption $n=2$ and of the assumption that $\M$ is a convex subset of the Euclidean space. We state the resulting theorem below, where we need the following stronger stability condition, which we will call {\em stability for the free boundary problem}. For a more general result, where the
convexity assumption on $\partial \M$ is removed, see the recent work of Li and Zhou, \cite{LiZhou1}. 

\begin{definition}
Let $\M$ be a smooth $(n+1)$-dimensional Riemannian manifold and $U\subset \M$ an open set. A varifold $V\in \va_s^u (U)$ is said to be stable for the free boundary problem if $\delta^2 V (\chi) \geq 0$ for every $\chi\in \mathfrak{X}_c^t (U)$. 
\end{definition}

\begin{theorem}\label{t:SS-GY}
Let $\M$ be a smooth $(n+1)$-dimensional Riemannian manifold which satisfies Assumption \ref{a:(C)}. 
Let $\Sigma^k$ be a sequence of stable minimal hypersurfaces in some open subset $U\subset \M$ and assume that
\begin{itemize}
\item[(i)] each $\Sigma^k$ is smooth except for a closed set of vanishing $\haus^{n-2}$-measure;
\item[(ii)] $\partial \Sigma^k\cap U$ is contained in $\partial \M$ and $\Sigma^k$ meets $\partial \M$ orthogonally (thus, $\Sigma^k$ is stationary for the free boundary problem);
\item[(iii)] $\Sigma^k$ is stable for the free boundary problem;
\item[(iv)] $\sup_k \haus^n (\Sigma^k) <\infty$.
\end{itemize}
Then a subsequence of $\Sigma^k$ (not relabeled) converges, in the sense of varifolds, to an integer rectifiable varifold $V$ such that
\begin{itemize}
\item[(a)] $V$ is, up to multiplicity, a stable minimal hypersurface $\Gamma$ with ${\rm dim}\, ({\rm Sing}\, (\Gamma))\leq n-7$;
\item[(b)] at any point $p\not \in {\rm Sing}\, (\Gamma)$ the convergence is smooth;
\item[(c)] $\Gamma$ meets $\partial \M$ orthogonally, thus $V\in \va_s^u (U)$;
\item[(d)] $V$ is stable for the free boundary problem.
\end{itemize}
\end{theorem}

\subsection{Boundary version for the constrained case} We close this section by combining Theorem \ref{t:boundary_est} with the interior estimates of Schoen and Simon to get a compactness theorem for stable minimal hypersurfaces which have a fixed given boundary $\gamma$ and meet $\partial \M$ transversally in a suitable quantified way. 

\begin{theorem}\label{t:SS-bdry}
Let $\M$ be an $(n+1)$-dimensional smooth Riemannian manifold which satisfies Assumption \ref{a:(C)}, $\gamma \subset \partial \M$ a $C^{2,\alpha}$ submanifold of $\partial \M$, $U$ an open subset of $\M$ and $K\subset U$ a set which meets $\partial\M$ in $\gamma$ at an opening angle smaller than $\frac{\pi}{2}$. 
Let $\Sigma^k$ be a sequence of stable minimal hypersurfaces in $U\subset \M$ and assume that
\begin{itemize}
\item[(i)] each $\Sigma^k$ is smooth except for a closed set of vanishing $\haus^{n-2}$ measure and $\gamma \cap \sing (\Sigma) = \emptyset$;
\item[(ii)] $\partial \Sigma^k\cap U = \gamma \cap U$;
\item[(iii)] $\sup_k \haus^n (\Sigma^k) <\infty$;
\item[(iv)] $\Sigma^k\subset K$. 
\end{itemize}
Then a subsequence of $\Sigma^k$, not relabeled, converges, in the sense of varifolds, to an integer rectifiable varifold $V$ such that
\begin{itemize}
\item[(a)] $V$ is, up to multiplicity, a stable minimal hypersurface $\Gamma$ with ${\rm dim}\, ({\rm Sing}\, (\Gamma))\leq n-7$;
\item[(b)] at any point $p\not \in {\rm Sing}\, (\Gamma)$ the convergence is smooth;
\item[(c)] ${\rm Sing}\, (\Gamma)\cap \partial \M = \emptyset$ and $\partial \Gamma = \gamma$ (in particular, the multiplicity of any connected component of $\Gamma$ which intersects $\partial \M$ must be $1$). 
\end{itemize} 
\end{theorem}
\begin{proof}
First of all, after extraction of a subsequence we can assume that $\Sigma^k$ converges to a varifold $V$. Observe that $V$ is stationary in ${\rm Int}\, (U)$ and thus it is integer rectifiable in there, by Allard's compactness theorem. Note also that each $\Sigma^k$ belongs to $\va_s^c (U, \gamma)$ and thus, by continuity of the first variations, $V$ belongs as well to $\va_s^c (U, \gamma)$. Thus, by the maximum principle of Proposition \ref{p:White-max} we conclude that $\|V\| (\partial \M) = \|V\| (\gamma)$. In particular, { as argued for Lemma \ref{c:gamma_meas_0}} $\|V\| (\partial \M)= 0$ and that $V$ is integer rectifiable in $U$.

Next observe that in $U\setminus \partial \M$ we can apply the Schoen-Simon compactness theorem: thus, except
for a set $K'$ in $U\setminus \partial \M$, the smooth convergence holds at every point $x_0\in U\setminus (\gamma \cup K')$
and ${\rm dim}\, (K') \leq n-7$. 
As for the points $x \in \gamma$, consider first an open subset $U'$ which has positive distance from $\partial U \setminus \partial \M$. By the boundary curvature estimates of Theorem \ref{t:boundary_est}, there is an $r_0>0$ and a constant $C_0$, both independent of $k$, such that
$|A_{\Sigma_k}|\leq C_0$ in any ball $B_{r_0} (x)$ with center $x\in \gamma \cap U'$. This implies that, in a fixed neighborhood $U''$ of $\gamma$, $\Sigma^k$ consists of a single smooth component which is a graph at a fixed scale, independent of $k$. The estimate on the curvature in Theorem \ref{t:boundary_est} gives then the convergence of these graphs in $C^{1,\alpha}$ for every $\alpha <1$. Since the limit turns out to be (locally) graphical and a solution of an elliptic PDE, classical Schauder estimates imply its smoothness and the smooth convergence. 
\end{proof}

\subsection{Varying the ambient manifolds}\label{s:changing_manifolds} In all the situations above, we can allow also for the manifolds $\M$ to vary in a controlled way, namely to change as $\M_k$ along the sequence. One version which is particularly useful is when the 
$\M_k$ are embedded in a given, fixed, Euclidean space and they are converging smoothly to a $\M$. All the compactness statements above still hold in this case and in particular the corresponding obvious modifications (left to the reader) will be used at one occasion in the very simple situation where the $\M_k$ are rescalings of the same $\M$ at a given point, thus converging to the tangent space at that point, cf. Section \ref{s:vartan} and Section \ref{s:remove} below.

\section{Wedge property}\label{s:wedge}
In this section we use the maximum principle to prove that, given a smooth $\gamma$ any stationary varifold $V\in 
\va^c_s (U, \gamma)$ meets $\gamma$ ``transversally'' in a quantified way, namely it lies in suitable wedges that have a controlled angle.  This property is necessary to apply to the compactness Theorem \ref{t:SS-bdry}. The precise formulation is the following

\begin{lemma}\label{wedge property}
Let $\M$ be a smooth $(n+1)$-dimensional submanifold satisfying Assumption \ref{a:(C)}, $\gamma$ be a $C^2$ $(n-1)$-dimensional submanifold of $\M$ and $U'\subset\subset U$ two open subsets of $\M$. Then there is a constant $\theta_0 (U, U', \gamma)< \frac{\pi}{2}$ and a compact set $K\subset \overline{U'}$ with the following properties:
\begin{itemize}
\item[(a)] $K$ meets $\gamma$ at an opening angle at most $\theta_0$;
\item[(b)] $\supp (V)\cap U' \subset K$ for every varifold $V\in \va_s^c (U, \gamma)$.
\end{itemize}
\end{lemma}

Note that in the special case of $U=\M$, we are allowed to choose $U'=\M$ and thus we conclude a uniform transversality property
for any varifold in $\va_s^c (\M, \gamma)$, in particular for the varifold $V$ of Proposition \ref{a.m. in annuli}. On the other hand we do
need the local version above for several considerations leading to the regularity of $V$ at the boundary. When $\M$ is a subset of the Euclidean case, the lemma above follows easily from the following two considerations:
\begin{itemize}
\item[(i)] By the classical maximum principle, $\supp (V)$ is contained in the compact subset $K$ which is the convex hull of
$(\gamma \cap \overline{U}) \cup (\partial U\setminus \partial \M)$, see for instance the reference \cite{simon-gmt};
\item[(ii)] Such convex hull $K$ meets $\gamma$ at an opening angle which is strictly less than $\frac{\pi}{2}$ at every point $x\in \gamma \cap U$ (here the $C^2$ regularity of $\gamma$ is crucially used, cf. the elementary Lemma \ref{l:c_hull} below). 
\end{itemize}
The uniform (upper) bound on the angle is then obtained in $U'\subset\subset U$ simply by compactness. 

Unfortunately, although the extension of (i) above to general Riemannian manifolds is folklore among the experts, we do not know of
a reference that we could invoke for Lemma \ref{wedge property} without some additional technical work. This essentially amounts to reducing to the Euclidean situation by a suitable choice of coordinates.

\subsection{Wedge property and convex hull} We start by recording the following elementary fact, which in particular proves claim (ii) above. 

\begin{lemma}\label{l:c_hull}
Consider a bounded, open, smooth, uniformly convex set $\M\subset \mathbb R^{n+1}$ and a $C^2$ $(n-1)$-dimensional connected submanifold $\gamma\subset B_r (0) \cap \partial \M$ passing through the origin. 
Then there is a wedge $W$ containing $\gamma$ such that:
\begin{itemize}
\item[(a)] The axis of $W$ is orthogonal to $T_0 \partial \M$;
\item[(b)] The tip of $W$ is $T_0 \gamma$;
\item[(c)] The opening angle is bounded away from $\frac{\pi}{2}$ in terms of the principal curvatures of $\partial \M$ and of those of $\gamma$.
\end{itemize}
\end{lemma}

\begin{proof}
For simplicity fix coordinates so that $T_0 \gamma = \{x_1=x_{n+1} = 0\}$, $T_0 \partial \M = \{x_1=0\}$ and $\M\setminus \{0\}$ is lying in 
$\{x_1 > 0\}$. For every $\theta< \frac{\pi}{2}$ let $\M_\theta$ be the portion of $\M$ lying in $\{x_{n+1} > x_1 \tan \theta\}$, and consider $r>0$ such that the open ball
\[
B_r ((0,0, \dots, 0, r))
\]
contains $\M_\theta$. Let $\rho (\theta)$ be the smallest such radius. $\rho(\theta)$ is a non-increasing function of $\theta$ and by the uniform convexity of $\M$, $\rho (\theta)\to 0$ as 
$\theta \uparrow \frac{\pi}{2}$. On the other hand we know that if $\rho (\theta) < \|A_\gamma \|_{\infty}^{-1}$, then $B_\rho ((0,0, \ldots, 0, \rho))$ is an osculating ball for $\gamma$ at $0$ and cannot contain any point of $\gamma$. This shows that for all $\theta$ sufficiently close to $\frac{\pi}{2}$, $\gamma$ is contained in $\{x_{n+1}\leq x_1 \tan \theta\}$. By a simple reflection argument we obtain the same property with $\{- x_{n+1} \leq x_1 \tan \theta\}$, which completes the proof of the lemma. 
\end{proof}

\subsection{Proof of Lemma \ref{wedge property}} First of all, we observe that by a simple covering argument it suffices to show the lemma in a sufficiently small neighborhood $U$ of any point $p\in \gamma$, since we already know by the maximum principle in Proposition \ref{p:White-max} that $\supp (V)\cap \partial M \subset \gamma$. 

Recall that we can assume that $\M$ is a subset of a closed Riemannian manifold $\tilde{\M}$, cf. Remark \ref{extr vs. intr}.
Let $p \in \gamma$, and $\tilde U$ a normal neighborhood of $p$ in $\tilde \M$. We then consider normal coordinates on $\tilde\M$ centered at $p$, given by the chart $\varphi:=E \circ \expo_p^{-1} : U \to \RR^{n+1}$, where the isomorphism $E:T_p\tilde \M \to \RR^{n+1}$ is chosen so that $E(T_p\partial\M)=\{x \in \RR^{n+1} : \, x_1=0\}$, and $E(T_p\gamma) = \{x \in \RR^{n+1}: \,x_1=x_{n+1}=0\}$. 

Now, if we let $A$ denote the second fundamental form of $\partial \M$ in $\M$ with respect to the unit normal $\nu$ pointing inside $\M$, $B$ the second fundamental form of $\varphi(\partial\M)$ in $\RR^{n+1}$ with respect to the unit normal $n$ pointing inside $\varphi(\M)$, and $\nabla$, $\bar{\nabla}$ the ambient Riemannian and Euclidean connection respectively, we immediately see that
\begin{equation*}
A(X,Y)\big|_{p} = - g( \bar{\nabla}_{X}\nu,Y)\big|_{p} = g(\nu,\bar{\nabla}_XY)\big|_{p} = \langle n,\nabla_XY\rangle\big|_{0} = -\langle\nabla_Xn,Y\rangle\big|_{0} = B(X,Y)\big|_{0},
\end{equation*}
since $\nu(p)=n(0)$, $g(.,.)\big|_{p}=\langle.,.\rangle$ and $\nabla\big|_{p}=\bar{\nabla}\big|_{0}$ by the properties of the exponential map.
Hence, it follows from Assumption \ref{a:(C)} that  $B \succeq \xi \, \text{Id}$ at $0$. Thus, if we represent $\varphi(\partial\M)$ as a graph of a function $f$ over its tangent plane $\{x \in \RR^{n+1} |\, x_1=0\}$ at $0$, the Hessian of $f$ is equal to $B$ at $0$, and hence there are some Cartesian coordinates $(y_2,\ldots,y_{n+1})$ on this plane such that $f$ has the form
\begin{equation}\label{paraboloid}
f(y_2,\ldots,y_{n+1}) =\frac{1}{2}(\kappa_2y_2^2+\ldots+\kappa_{n+1}y_{n+1}^2) + O(|y|^3),
\end{equation}
where $\kappa_2,\ldots,\kappa_{n+1} > \xi>0$ are principal curvatures (w.r.t. inward pointing normal at $0$).

\medskip

In particular we can assume that $U$ is chosen so small that $f$ is uniformly convex in the Euclidean sense, namely that $D^2 f > 0$ everywhere on $\varphi (\tilde U)$. By abuse of notation we keep using $\tilde U$ for $\varphi (\tilde U)$, $\M$ for $\varphi (\M)$ and thus $V$ for the varifold $\varphi_\sharp V$.
Since we can now regard $\M$ as a convex subset of the euclidean space, we can apply Lemma \ref{l:c_hull} and conclude that
$\gamma$ is contained in a wedge $W$ of the form $\{|x_{n+1}|\leq \tan \theta x_1\}$. However we cannot apply the maximum principle to conclude that $\supp (V) \subset W$ because $V$ is not {\em stationary in the euclidean metric}. Our aim is however to show that, if we enlarge slightly $\theta$, but still keep it smaller than $\frac{\pi}{2}$, then $\supp (V) \subset W$. The resulting $\theta$ will depend on the manifold $\M$, the submanifold $\gamma$ and the size of $\tilde U$, but not on the point $p$. Thus this argument completes the proof, since the set $K$ can be taken to be, in a neighborhood of $\gamma\cap U'$, the union of the corresponding wedges for $p$
(intersected with the corresponding neighborhoods $\tilde{U}$) as $p$ varies in $\gamma \cap U'$. 

\medskip

Recall that, in our notation, $\M$ is in fact the set $\{y_1 \geq f (y_2, \ldots , y_{n+1})\}$. 
For each $\lambda\geq 0$ consider now the function 
\[
f_{\lambda} (y_2, \ldots, y_{n+1}) = (1-\lambda) f (y_2, \ldots, y_{n+1}) + \lambda \frac{y_{n+1}}{\tan \theta}\, .
\]
For $\lambda \downarrow 0$, the function $f_{\lambda}$ converges in $C^2$ to the function $f$. Thus the set
$\M_{\lambda} = \{y_1 \geq f_{\lambda} (y_2, \ldots , y_{n+1})\}$ is uniformly convex in the Riemannian manifold $\tilde \M$
as soon as $\lambda \leq \varepsilon$. 

Observe next that all the graphs of all the functions $f_\lambda$ intersect in an $n-1$-dimensional submanifold, which is indeed
the intersection of the graphs of $f_1$ and $f_0 = f$. Consider now the region 
\[
R = \{ f_0 (y_2, \ldots , y_{n+1}) \leq  y_1 \leq f_\varepsilon (y_2, \ldots, y_{n+1})\}\, ,
\]
cf. Fig. \ref{f:fol}. Since the graph of $f_0$ is in fact $\partial \M$, we know from Proposition \ref{p:White-max} that $\supp (V) \cap \partial \M \cap R \subset \gamma$ and from the choice of the wedge $W$ we thus know that $\supp (V) \cap \partial \M \cap R = \{0\}$. Assume now by contradiction that $R$ contains another point $p\in \supp (V)$. Then this point does not belong to $\gamma$. On the other hand there must be a minimum $\delta$ such that the graph of $f_\delta$ contains this point. But then, by the fact that $\M_\delta$ is uniformly convex in $\tilde{\M}$, this would be a contradiction to the maximum principle of Proposition \ref{p:White-max}. 

We thus conclude that the region $R$ intersects the support of $V$ only in the origin. On the other hand recall that $f_\delta$ is convex also in the Euclidean sense. Thus its graph lies above its tangent at $0$, which is given by 
$\{y_{n+1} = y_1 \delta^{-1} \tan \theta\}$. This implies that the support of $V$ intersected with $\tilde{U}$ is in fact contained in 
\[
\{y_{n+1}\leq y_1 \delta^{-1} \tan \theta\}\, .
\]
Symmetrizing the argument we find the new desired wedge in which the support of our varifold is contained. \qed

\begin{figure}[htbp]
\begin{center}
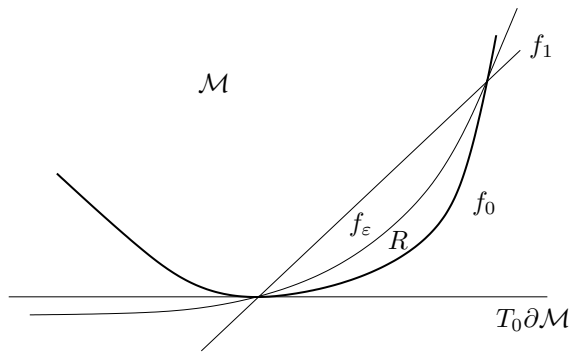
\end{center}
\caption{The region $R$ foliated by the graps of $f_\lambda$.}
\label{f:fol}
\end{figure}

\section{Replacements at the boundary}\label{s:replacements}

We have now all the tools for proving the boundary regularity of the varifold $V$ in Proposition \ref{a.m. in annuli} and we can start with the argument leading to

\begin{theorem}\label{t:Bdry_regularity}
The varifold $V$ of Proposition \ref{a.m. in annuli} has all the properties claimed in Theorem \ref{t:main}.
\end{theorem}

The argument is indeed split into two main steps. In the first one we employ another important concept first developed by Pitts, called a \emph{replacement}.
\begin{definition}\label{d:replacement}
Let $V \in \va(\M)$ be a stationary varifold in $\M$, belonging to one of the classes $\mathcal{V}^u_s$ and $\mathcal{V}^c_s (\mathcal{M}, \gamma)$, and $U \subset \M$ an open set. A stationary varifold $V' \in \va(\M)$ (belonging to one of the two corresponding classes) is called a replacement for $V$ in $U$ if $V = V'$ on $G(\M \setminus U)$, $\|V\|(\M) = \|V'\|(\M)$, and $V' \res U$ is a stable minimal hypersurface $\Gamma$. In the constrained case we require that $\partial \Gamma \cap U = \gamma \cap U$ (in particular the connected components of $\Gamma$ that intersect $\gamma$ will arise with multiplicity $1$ in the varifold $V$). In the unconstrained case the surface $\Gamma \cap U$ meets $\partial \M$ orthogonally.
\end{definition}
Our goal now is to show that the almost minimizing property of the sequence $\{\Gamma^j\}$ from Proposition \ref{a.m. in annuli} is sufficient to prove the existence of a replacement for the varifold $V$. More precisely, we prove:

\begin{propos}\label{existence of replacements}
Let $\{\Gamma^j\},V \text{ and } r$ be as in Proposition \ref{a.m. in annuli}. Fix $x \in \M$ and consider an annulus An $ \in \mathcal{AN}_{r(x)}(x)$. Then there exist a varifold $\tilde{V}$, a sequence $\{\tilde{\Gamma}^j\}$ and a function $r':\M \to \mathbb{R}^{+}$ such that
\begin{itemize}
\item $\tilde{V}$ is a replacement for $V$ in An and $\tilde{\Gamma}^j$ converges to $\tilde{V}$ in the sense of varifolds;\\
\item $\tilde{\Gamma}^j$ is almost minimizing in every $An' \in \mathcal{AN}_{r'(y)}(y)$ with $y \in \M$;\\
\item $r(x) = r'(x)$.
\end{itemize}
\end{propos}

\subsection{Homotopic Plateau's problem}
Let us fix a point $x \in \M$ and $An \in \mathcal{AN}_{r(x)}(x)$ from now on. If $x\in {\rm Int}\, \M$, then the statement above is indeed proved in \cite{dl-t}, except for a small technical adjustment which we explain in Section \ref{s:differente_3} below. We fix therefore $x\in \partial \mathcal{M}$. The strategy of the proof will be analogous to the one in \cite{dl-t} and 
follows anyway the pioneering ideas of Pitts: in $An$ we will indeed replace the a.m. sequence $\Gamma^j$ with a suitable $\tilde{\Gamma}^j$, which is a minimizing sequence for a suitable (homotopic) variational problem. 

As a starting point for the proof we consider for each $j \in \NN$ the following class and the corresponding variational problem:

\begin{definition}
Let $U\subset \M$ be an open set and for each $j\in \mathbb N$ consider the class 
$\haus_c (\Gamma^j,U)$ (resp. $\haus_u (\Gamma^j, U)$) of surfaces $\Xi$ such that there is a constrained (resp. unconstrained) family of surfaces $\{\Gamma_t\}$  satisfying $\Gamma_0 = \Gamma^j$, $\Gamma_1 = \Xi$, (\ref{amm prop1}), (\ref{amm prop2}), and (\ref{amm prop3}) for $\epsilon=1/j$ (recall that $m$ is fixed by Remark \ref{fixing m}). The subscript $c$ (resp. $u$) will be dropped when clear from the context. A minimizing sequence in $\haus (\Gamma^j, U)$ is a sequence $\Gamma^{j,k}$ for which the volume of $\Gamma^{j,k}$ converges towards the infimum. 
\end{definition}

We will call the variational problem above the $(2^{m+2}j)^{-1}$ - \emph{homotopic Plateau} problem.
Next, we take a minimizing sequence $\{\Gamma^{j,k}\}_{k \in \mathbb{N}}\subset \haus (\Gamma^j, An)$. Up to subsequences, we have that
\begin{itemize}
\item as integral currents, $\llbracket \Gamma^{j,k} \rrbracket$ converge weakly to an integral current $Z^j$ (in the constrained case the current is integral in $\M$, including the boundary $\partial \M$, because $\partial Z^j \res \partial \M = \llbracket \gamma \rrbracket$; in the unconstrained case the current $Z^j$ is a-priori only integral {\em in the interior}; however the regularity proved later in Corollary \ref{c:regularity} will actually imply that it is integral even when including the boundary $\partial \M$);
\item as varifolds, $\Gamma^{j,k}$ converge to a varifold $V^j$;
\item $V^j$, along with a suitable diagonal sequence $\tilde{\Gamma}^j = \Gamma^{j,k(j)}$ converges to a varifold $\tilde{V}$.
\end{itemize}
The rest of the section will then be devoted to prove that the varifold $\tilde{V}$ is in fact the replacement of Proposition \ref{existence of replacements} and that the sequence $\tilde{\Gamma}^j$ satisfies the requirements of the same proposition.

\begin{remark}\label{r:integrality_again}
Note that $V^j\in \mathcal{V}^c_s (An, \gamma)$ and that $V^j$ is a.m. in annuli (in fact it has a much stronger minimizing property!). For this reason we can apply Lemma \ref{c:gamma_meas_0} and conclude that $V^j$ is an integer rectifiable varifold. 
\end{remark}

The proof is split into two steps. In the first one we will show that, at all sufficiently small scales, the current $Z^j$ is indeed a minimizer of the area in the corresponding variational problem (constrained and unconstrained) without {\em any restriction} on the competitors.
More precisely we show that 

\begin{lemma}\label{replacements first step lemma}
Let $j \in \NN$ and $y \in An$. Then there are a ball $B = B_{\rho}(y) \subset An$ and a $k_0 \in \NN$ such that every set $\Xi$ with the following properties (satisfied for some $k\geq k_0$) belongs to the class $\haus(\Gamma^j,An)$:
\begin{itemize}
\item $\Xi$ is a smooth hypersurface away from a finite set;
\item $\partial \Xi \cap B = \gamma \cap B$ in the constrained problem, whereas $\partial \Xi \cap {\rm Int}\, (B) = \emptyset$ in the unconstrained problem;
\item $\Xi \setminus B = \Gamma^{j,k} \setminus B$;
\item $\haus^n(\Xi) < \haus^n(\Gamma^{j,k})$.
\end{itemize} 
\end{lemma}

As a simple corollary, whose proof will be given later, using the regularity theory for area minimizing currents for a given prescribed boundary (and the corresponding regularity theory for the minimizers in the free boundary case, as developed by Gr\"uter in \cite{Gruter1}) we then get the following 

\begin{corol}\label{c:regularity}
Let $\tilde B$ be the ball concentric to the ball $B$ in Lemma \ref{replacements first step lemma} with half the radius.
In the constrained case the current $Z^j$ has boundary $\gamma$ in $B$ and any competitor $Z^j + \partial S$, where $S$ is an integer rectifiable current supported in $\tilde B$, cannot have mass smaller than that of $Z^j$. In the unconstrained case $Z^j$ is a minimizer with respect to free boundary perturbations, namely any current $Z^j + T$ with $\supp (\partial T)\subset B\cap \partial \M$ and $\supp (T) \subset \tilde B$, cannot have mass smaller than that of $Z^j$.

Thus, $Z^j\res An = V^j\res An = \bar{\Gamma}^j$ is a regular, minimal, embedded hypersurface except for a closed set ${\rm Sing}\, (\bar\Gamma^j)$ of dimension at most $n-7$. In the unconstrained case it meets the boundary $\partial \M$ orthogonally and it is stable for the free  boundary problem. In the constrained case ${\rm Sing}\, (\bar \Gamma^j)$ does not intersect $\partial \M$ and $\partial \bar \Gamma^j = \gamma$ (in $An$; in particular any connected component of $\Gamma^j$ that intersects $\partial \M$ must have multiplicity $1$). 
\end{corol}

The second step in the proof of Proposition \ref{existence of replacements} takes advantage of the compactness theorems in Section \ref{s:stability} to pass into the limit in $j$ and conclude that $\tilde{V}$ has the desired regularity properties.

\subsection{Proof of Lemma \ref{replacements first step lemma}} We focus on the constrained case, since the proof in the unconstrained case follows the same line and it is indeed easier.

We will exhibit a suitable homotopy between $\Gamma^{j,k}$ and $\Xi$ by first deforming $\Gamma^{j,k}$ inside $B$ to a cone with vertex $y$ and base $\Gamma^{j,k} \cap \partial B$, and then deforming this cone back to $\Xi$, without increasing the area by more than $(2^{m+2}j)^{-1}$, which will prove the claim. To this end, we borrow the "blow down -blow up" procedure from \cite{dl-t}, which in turn is borrowed from Smith \cite{smith} (see also Section 7 of \cite{c-dl}) and we only need to modify the idea because $x\in\partial \M$.

Fix $y \in An \cap \partial \M$, and $j \in \NN$. If $y \notin \gamma$, by considering $\M$ as a subset of $\tilde{\M}$ as in Remark \ref{extr vs. intr}, and simply making sure to choose $\rho$ small enough that $B_{\rho}(y)\subset \subset \tilde{\M} \setminus \gamma$, we can reduce to the interior case. Note that we also make use of the convexity assumption on $\M$ to make sure all the surfaces in the homotopy stay inside $\M$. Therefore, we are left to prove the case $y \in An \cap \partial \M \cap \gamma$.\\
First, in a small neighborhood around $y$, we can find (smooth) diffeomorphisms 
\begin{align*}
&\Psi_1: \RR^{n-1} \times \RR \to \partial \M , \quad  \Psi_1^{-1}(\gamma) \subset \RR^{n-1}\times \{0\}, \quad \Psi_1(0) = \iota(y); \\
&\Psi_2: \partial \M \times \RR^+ \to  \M , \quad  \Psi_2 (x,t) = \text{exp}_x(t\nu(x)), 
\end{align*}
with $\iota: \partial \M \to \M$ the inclusion map, and $\nu(x)$ unit normal to $\partial \M$. By taking $\Psi_2(\Psi_1(x),t)$ and composing it with a linear map if necessary, we get a (smooth) local coordinate chart $\Psi: U \subseteq \RR^{n+1} \to V$ in a neighborhood $V \subset An \subset \M$ of $y$, with $\Psi(0) = y,\, \partial \M \cong \RR^n \times \{0\}, \gamma \cong \RR^{n-1} \times \{0\} \times \{0\}$ , and $D\Psi_0 =\text{Id}$. In the following, $B^e_r(0)$ and $\haus^{n,e}$ are used to denote the ball of radius $r$ and the Hausdorff measure w.r.t the euclidean metric in the given coordinates. We will choose $\tau>0$ small enough, that $B^e_{2\tau}(0) \subseteq U$. The required radius $\rho$ of the geodesic ball $B=B_\rho(y)$ will be fixed later, but chosen small enough that $\Psi^{-1}(B_{\rho}(y)) \subset \subset B^e_{\tau}(0)$ (and, of course, smaller than the injectivity radius). Furthermore, by choosing $U$ (and consequently $\tau$) small enough, we can ensure for any surface $\Sigma \subset B_{2\tau}^e(0)$ that
\begin{equation}\label{metric comparison}
\frac{1}{c}\haus^{n,e}(\Sigma) \leq \haus^n(\Sigma)\leq c\haus^{n,e}(\Sigma),
\end{equation}
where c depends on the metric, and $c\to 1$ for $\tau \to 0$. From now on, we will use the same symbols to denote sets and their representations in the coordinates given by $\Psi$.

\medskip

\noindent {\bf Step 1: Stretching $\Gamma^{j,k} \cap \partial B^e_r(0)$.} First of all , we will choose $r \in (\tau, 2\tau)$ such that, for every $k$,
\begin{equation}\label{transversal slice}
\begin{split}\Gamma^{j,k} \text{ is regular in a neighborhood of } \partial B^e_r(0)\\ \text{ and intersects it transversally } \end{split}
\end{equation}
This is implied by Sard's lemma, since each $\Gamma^{j,k}$ has only finitely many singularities. We let $K$ be the cone
\begin{equation*}
K = \{\lambda z \,| \,0 \leq \lambda < 1, z \in \partial B_r^e(0) \cap \Gamma^{j,k}\}
\end{equation*}
We now show that $\Gamma^{j,k}$ can be homotopized through a family $\tilde{\Omega}_t$ to a surface $\tilde{\Omega}_1$ in such a way that
\begin{itemize}
\item $\max_t \haus^{n,e}(\tilde{\Omega}_t) - \haus^{n,e}(\Gamma^{j,k})$ can be made arbitrarily small; 
\item $\tilde{\Omega}_1$ coincides with $K$ in a neighborhood of $\partial B_r^e(0)$
\end{itemize}
To this end, we consider a smooth function $\varphi:[0,2\tau] \to [0,2\tau]$ with
\begin{itemize}
\item $|\varphi(s) - s| \leq \epsilon$ and $0 \leq \varphi' \leq 2$;\\
\item $\varphi(s) = s$ if $|s - r| > \epsilon$ and $\varphi \equiv r$ in a neighborhood of $r$.
\end{itemize}
Set $\Phi(t,s):=(1-t)s + t\varphi(s)$. If $A$ is any set, we use $\lambda A$ as usual to denote the set $\{\lambda x \, | \, x \in A\}$. We can now define $\tilde{\Omega}_t$ in the following way:
\begin{itemize}
\item $\tilde{\Omega}_t \setminus An^e(0,r-\epsilon, r+\epsilon) = \Gamma^{j,k} \setminus An^e(0,r-\epsilon,r+\epsilon)$;\\
\item $\tilde{\Omega}_t \cap \partial B^e_s(0) = \frac{s}{\Phi(t,s)}\big(\Gamma^{j,k} \cap \partial B^e_{\Phi(t,s)}\big)$ for every $s \in (r-\epsilon,r+\epsilon)$,
\end{itemize}
where the annuli (with the superscript $e$) are with respect to the euclidean metric. Note that our choice of coordinates ensures that $\gamma$ is preserved as the boundary. Furthermore, the surfaces are smooth (with the exception of a finite number of singularities), since $\Gamma^{j,k}$ is regular in a neighborhood of $\partial B^e_r(0)$ Moreover, owing to \eqref{metric comparison} and \eqref{transversal slice}, and for $\epsilon$ sufficiently small, $\tilde{\Omega}_t$ will have the desired properties. Finally, since $\Xi$ coincides with $\Gamma^{j,k}$ on $\M \setminus B_{\rho}(y)$ (and in particular, outside $B^e_{\tau}(0)$), the same argument can be applied to $\Xi$. This shows that
\begin{equation}\label{loc. stretching to a cone}
\begin{split}
\text{w.l.o.g. we can assume } K = \Xi = \Gamma^{j,k}\\
\text{ in a neighborhood of } \partial B^e_r(0)
\end{split}
\end{equation}

\medskip

\noindent {\bf Step 2: The homotopy.} We now construct the required homotopy mentioned in the beginning of the proof, as the family $\{\Omega_t\}_{t \in [0,1]}$ of hypersurfaces which satisfy:
\begin{itemize}
\item $\Omega_t \setminus \bar{B}_r^e(y) = \Gamma^{j,k} \setminus \bar{B}_r^e(y)$ for every $t$;
\item $\Omega_t \cap An^e(0, |1-2t|r,r) = K \cap An(y,|1-2t|r,r)$ for every $t$;
\item $\Omega_t \cap  \bar{B}^e_{(1-2t)r}(0) = (1-2t)(\Gamma^{j,k} \cap \bar{B}^e_r(0))$ for $t \in [0,\frac{1}{2}]$;
\item $\Omega_t \cap  \bar{B}^e_{(2t-1)r}(0) = (2t-1)(\Xi \cap \bar{B}^e_r(0))$ for $t \in [\frac{1}{2},1]$.
\end{itemize}
Note again that, because of the way we chose our coordinates and deformations, and consequently \eqref{loc. stretching to a cone}, this satisfies the properties of a smooth constrained family. The only property left to check is that
\begin{equation}
\max_t \haus^n(\Omega_t) \leq \haus^n(\Gamma^{j,k}) + \frac{1}{2^{m+2}j} \quad \forall k \geq k_0
\end{equation}
holds for a suitable choice $\rho, r$ and $k_0$.\\
First we observe the following standard facts, for every $r <2\tau$ and $\lambda \in [0,1]$:
\begin{align}
& \haus^{n,e} (K) = \frac{r}{n} \haus^{n-1,e}\big(\Gamma^{j,k} \cap \partial B^e_r(0)\big);\label{euclid.area est 1}\\
& \haus^{n,e} \big(\lambda(\Gamma^{j,k} \cap \bar{B}^e_r(0))\big) = \haus^{n,e} \big(\lambda(\Gamma^{j,k} \cap B^e_r(0))\big) \leq \haus^{n,e}\big(\Gamma^{j,k} \cap B^e_r(0)\big);\label{euclid.area est 2}\\
& \haus^{n,e} \big(\lambda(\Xi \cap \bar{B}^e_r(0))\big) = \haus^{n,e} \big(\lambda(\Xi \cap B^e_r(0))\big) \leq \haus^{n,e}\big(\Xi \cap B^e_r(0)\big);\label{euclid.area est 3}\\
& \int_{0}^{2\tau} \haus^{n-1,e}\big(\Gamma^{j,k}\cap\partial B^e_s(0)\big) \, ds \leq \haus^{n,e} \big( \Gamma^{j,k} \cap B^e_{2\tau}(0)\big),\label{euclid.area est 4}
\end{align}
where the equalities in \eqref{euclid.area est 2} and \eqref{euclid.area est 3} are due to \eqref{transversal slice}. From \eqref{metric comparison} and the assumption on $\Xi$ we conclude $\haus^{n,e}(\Xi \cap B^e_{2\tau}(0)) \leq c^2 \haus^{n,e}(\Gamma^{j,k}\cap B^e_{2\tau}(0))$, which together with \eqref{euclid.area est 1}, \eqref{euclid.area est 2} and \eqref{euclid.area est 3} gives us the estimate
\begin{align}
\max_t\haus^n(\Omega_t) - \haus^n(\Gamma^{j,k}) &\leq c\haus^{n,e}\big(\Omega_t \cap B^e_{2\tau}(0)\big) \nonumber\\
&\leq c^3\haus^{n,e}\big(\Gamma^{j,k} \cap B^e_{2\tau}(0)\big) + r \haus^{n-1,e}\big(\Gamma^{j,k} \cap \partial B^e_r(0)\big)
\end{align}
By \eqref{euclid.area est 4} we can find $r \in (\tau,2\tau)$ which, in addition to \eqref{transversal slice} (and consequently \eqref{loc. stretching to a cone}), satisfies
\begin{equation}\label{slice from coarea}
\haus^{n-1,e} \big(\Gamma^{j,k} \cap \partial B^e_r(0)\big) \leq \frac{2}{\tau} \haus^{n,e} \big( \Gamma^{j,k} \cap B^e_{2\tau}(0) \big).
\end{equation}
Hence,
\begin{equation}\label{final est 6.3(1)}
\max_t\haus^n(\Omega_t) \leq \haus^n(\Gamma^{j,k}) + (4+c^2)\haus^{n,e}\big(\Gamma^{j,k} \cap B^e_{2\tau}(0)\big).
\end{equation}
By a metric comparison argument similar to \eqref{metric comparison} relating the lenghts of curves inside $B^e_{2\tau}(0)$, we can obtain the inclusions $B_{\rho}(y) \subset \subset B^e_{\tau}(0) \subset B^e_{2\tau}(0) \subset \subset B_{\bar{c}\rho}(y)$, where the constant $\bar{c}$ depends on the metric, assuming of course that $\tau$ is initially chosen small enough. Next, by the convergence of $\Gamma^{j,k}$ to the stationary varifold $V^j$, we can choose $k_0$ such that
\begin{equation}\label{final est 6.3(2)}
\haus^{n,e}\big(\Gamma^{j,k}\cap B^e_{2\tau}(0)\big) \leq 2 ||V^j||(B_{\bar{c}\rho}(y)) \quad \text{for } k \geq k_0.
\end{equation}
Finally, by the monotonicity formula (see Theorem 3.4.(2) in \cite{allard-bdry}),
\begin{equation}\label{final est 6.3(3)}
||V^j||(B_{\bar{c}\rho}(y) \leq C_{\M}||V^j||(\M)\rho^n.
\end{equation}
By gathering the estimates \eqref{final est 6.3(1)}, \eqref{final est 6.3(2)}, and \eqref{final est 6.3(3)} (and having chosen $\tau$ small enough as instructed, depending only on $\M$), we deduce that if 
\begin{itemize}
\item $\rho$ is chosen small enough that
\begin{equation*}
2(4 + c^2)C_{\M}||V^j||(\M)\rho^n < \frac{1}{2^{m+2}j}
\end{equation*}
holds, 
\item $k_0$ is chosen large enough that \eqref{final est 6.3(2)} holds
\item and $r \in (\tau,2\tau)$ is fixed so that it satisfies \eqref{transversal slice} and \eqref{slice from coarea}, 
\end{itemize}
then we can construct $\{\Omega_t\}$ as above, concluding the proof. \qed

\subsection{Proof of Corollary \ref{c:regularity}}\label{s:differente_3} {\rm \bf Step 1. Minimality in the interior.} Again, we focus on the constrained problem, since the unconstrained problem is exactly the same.  
Strictly speaking, the conclusion of the corollary is new even in the interior, because in \cite{dl-t} the homotopic Plateau's problem was stated in the framework of Caccioppoli sets, i.e. not allowing multiplicities for our currents. We thus first show how to remove this technical assumption in the interior; in turn, the following argument also gives the needed technical adjustment to the arguments in \cite{dl-t} in order to show the interior regularity, cf. Remark \ref{r:differente_2}. 

\medskip

Fix $j \in \NN, y \in {\rm Int}\, (An)$, and let $B_{\rho}(y) \subset An$ be the ball given by Lemma \ref{replacements first step lemma} where we assume in addition $\rho < \dist(y,\partial \M)$. We will prove, by contradiction, that the integral current $Z^j$ (obtained as the weak limit of currents $\llbracket \Gamma^{j,k} \rrbracket$) is area minimizing in $B_{\rho/2}(y)$. Assume, therefore, it is not, and there exists an integral current $S$, with $\partial S = \gamma$, $S = Z^j$ on $\M \setminus B_{\rho/2}(y)$ and
\begin{equation}\label{competitor current}
\mathbf{M}(S) < \mass(Z^j) - \eta
\end{equation}
Since $\sup_k (\mass(\Gamma^{j,k}) + \mass(\gamma)) < \infty$, and therefore the weak and flat convergence are equivalent, we have the existence of currents integral $A_{j,k}$ and $B_{j,k}$ with
\begin{equation}
\Gamma^{j,k} - Z^j = \partial A_{j,k} + B_{j,k} \quad\mbox{and}\quad \mass(A_{j,k}) + \mass(B_{j,k}) \to 0
\end{equation}
In fact, considering that $\partial (\Gamma^{j,k} - Z^j)=0$, we can assume w.l.o.g. that $B_{j,k}=0$. By slicing theory, we can choose $\rho/2 < \tau < \rho$ and a subsequence (not relabeled) such that
\begin{equation}\label{slicing}
\partial (A_{j,k} \res B_{\tau}(y)) = (\partial A_{j,k}) \res B_{\tau}(y) + R_{j,k}, \quad \mass(R_{j,k}) \to 0
\end{equation}
where $\spt R_{j,k} \subset \partial B_{\tau}(y)$, and $R_{j,k}$ is integer multiplicity (cf. Fig. \ref{f:paste}).
Now, define the integer $n$-rectifiable current 
\begin{equation*}
S^{j,k} := S \res B_{\tau}(y) - R_{j,k} +  \Gamma^{j,k} \res (\M \setminus B_{\tau}(y)).
\end{equation*}
It is easy to check from the above that $\partial S^{j,k} = \gamma$. Moreover, from the weak convergence $\Gamma^{j,k} \rightharpoonup Z^j$ we get $\mass(Z^j\res B_{\tau}) \leq \liminf_{k \to \infty} \mass (\Gamma^{j,k}\res B_{\tau})$, and together with \eqref{competitor current}, \eqref{slicing}, this implies
\begin{equation}\label{mass gap}
\limsup_{j \to \infty} \big( \mass (S^{j,k}) - \mass(\Gamma^{j,k}) \big) \leq - \eta.
\end{equation}
We now proceed to approximate $S^{j,k}$ with smooth surfaces, which would by construction exhibit a similar gap in area (mass) with respect to $\Gamma^{j,k}$. The idea is to then apply Lemma \ref{replacements first step lemma}, thereby showing that these smooth surfaces belong to the class $\haus(\Gamma^j,An)$, and thus contradicting the minimality of the original sequence $\Gamma^{j,k}$.\\
\begin{figure}[htbp]
\begin{center}
\scalebox{0.5}{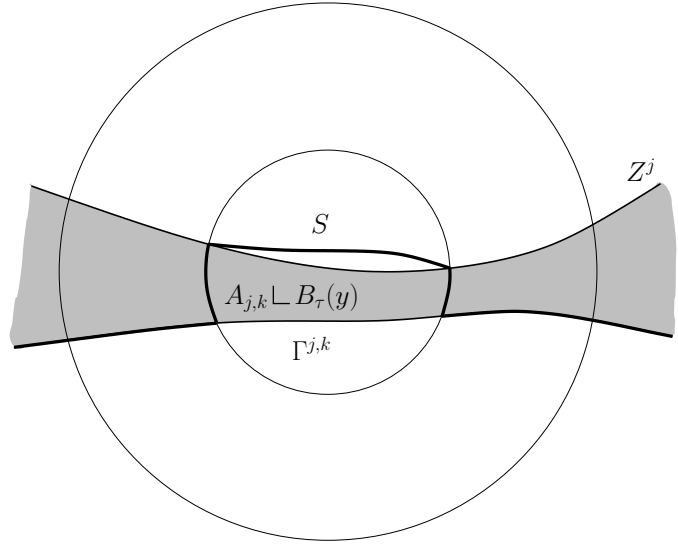}
\end{center}
\caption{The cut-and-paste procedure to produce a suitable competitor.}\label{f:paste}
\end{figure}

Let us first fix $(a,b) \subset \subset (\tau,\rho)$ with the property that $\Gamma^{j,k} \cap An(y,a,b)$ is a smooth surface. Since $\partial [S^{j,k} \res B_b(y)] \subset \partial B_b(y)$, we can find an $n$-rectifiable current $\Xi$ with $\spt (\Xi) \subset \partial B_b(y)$ and $\partial \Xi = \partial [S^{j,k} \res B_b(y)]$. Taking $R=S^{j,k}\res B_b(y) - \Xi$ we apply 4.5.17 of \cite{federer} to find a decreasing sequence of $\haus^{n+1}$-measurable sets $\{U_i\}_{i=-\infty}^{\infty}$ (of finite perimeter in $B_b$) and use them to construct rectifiable currents
\begin{align}\label{current decomposition}\begin{split}
S_i^{j,k} = \partial \llbracket U_i \rrbracket \res B_b(y) \quad \text{with} \quad \spt \partial S_i^{j,k} \subset \partial B_b(y), \quad \text{and}\\
S^{j,k} \res B_b(y) = \sum_{i \in \ZZ}S_i^{j,k}, \qquad \mass(S^{j,k} \res B_b(y)) = \sum_{i \in \ZZ} \mass (S_i^{j,k}).
\end{split}
\end{align}
In fact, $R = \partial T$ where $T = \sum_{i=1}^{\infty}\llbracket U_i \rrbracket - \sum_{i = -\infty}^{0}\llbracket B_b(y) \setminus U_i \rrbracket$. Let us therefore define the integer valued function $f:B_b(y) \to \ZZ$ by
\begin{equation*}
 f := \sum_{i=1}^{\infty} \chi_{U_i} - \sum_{i=-\infty}^{0} \chi_{B_b\setminus U_i},
 \end{equation*}
where $\chi_{A}$ denotes the characteristic function of a set $A$. Because the sequence $\{ U_i \}_{i=-\infty}^{\infty}$ is decreasing, we see immediately that $U_i=\{x: \, f(x) \geq i\}$. In fact, $f$ is of bounded variation inside $B_b(y)$, which follows from \eqref{current decomposition} and the fact that (see Remark 27.7 in \cite{simon-gmt})
\begin{equation}\label{perimeter-mass}
\mass(\partial\llbracket U_i \rrbracket \res B_b(y)) = \int_{B_b(y)} | D \chi_{U_i}|.
\end{equation}

By recalling the standard way of approximating functions of bounded variation by smooth functions, we take a compactly supported convolution kernel $\varphi$ and consider the functions $f_{\epsilon} = f * \varphi_{\epsilon}$, for $\epsilon < \rho - b$ (hence $\spt f_{\epsilon} \subset B_{\rho}(y)$). Of course, $\int_{B_{\rho}} |Df_{\epsilon}| \to \int_{B_{\rho}} |Df| \text{ for } \epsilon \to 0$. If we define $U_{t,\epsilon}:=\{x: \, f_{\epsilon}(x) \geq t\}$, then by coarea formula
\begin{equation*}
\int |Df_{\epsilon}| = \int_{-\infty}^{\infty} \, dt \int |D \chi_{U_{t,\epsilon}}|.
\end{equation*}
By a simple argument, which essentially follows from Chebyshev's inequality applied to the function $f_{\epsilon}(x) - f(x)$ (see Lemma 1.25 in \cite{giusti}), we get $\chi_{U_{t,\epsilon}} \to \chi_{U_i}$ in $L^1$ for every $t \in (i-1,i), \,\, i \in \ZZ$. Taking a sequence $\epsilon_l \to 0$, and using the lower semicontinuity of the perimeter w.r.t $L^1$ convergence, we deduce
\begin{align}
\int_{B_b}|Df| &= \lim_{j \to \infty} \int_{B_b}|Df_{\epsilon_l}| \geq \int_{-\infty}^{\infty}dt \, \liminf_{l \to \infty} \int_{B_b} |D \chi_{U_{t,\epsilon_l}}| \nonumber\\
& \geq \sum_{i = -\infty}^{\infty}\int_{i-1}^i dt \, \int_{B_b} |D \chi_{U_i}| = \int_{B_b}|Df|.
\end{align}
Hence, for all $i \in \ZZ$ and almost all $t \in (i-1,i)$, $\liminf_{l \to \infty} \int_{B_b} |D \chi_{U_{t,\epsilon_l}}| \to \int_{B_b} |D \chi_{U_i}|$. Moreover, since almost all level sets are smooth by Sard's lemma, for all $i \in \ZZ$ we may choose a $t_i \in (i-1,i)$ such that:
\begin{itemize}
\item $\partial U_{t_i,\epsilon_l}$ is smooth;\\
\item $\liminf_{l \to \infty} \int_{B_b} |D \chi_{U_{t_i,\epsilon_l}}| \to \int_{B_b} |D \chi_{U_i}|$.
\end{itemize}
By choosing a diagonal subsequence (without relabeling), we can ensure that the $\liminf_{l \to \infty}$ is replaced by a $\lim_{l \to \infty}$. We now define a current 
\begin{equation*}\Delta^{j,k,l} = \sum_{i=-\infty}^{\infty} \partial \llbracket U_{t_i,\epsilon_l}\rrbracket \res B_b(y),
\end{equation*}
and note that it is induced by a smooth surface (for each $l \in \NN$), since it is composed of smooth level sets of a smooth function. Furthermore, the properties above together with \eqref{current decomposition} and \eqref{perimeter-mass} imply that  $\mass(\Delta^{j,k,l}) \to \mass(S^{j,k}\res B_b(y))$ as $l \to \infty$. 

We would now like to patch $\Delta^{j,k,l}$  with $\Gamma^{j,k}$ outside $B_b(y)$. For this, recall that $S^{j,k} \cap An(y,a,b) = \Gamma^{j,k} \cap An(y,a,b)$ is also a smooth surface. Therefore, fixing a regular tubular neighborhood $T$ of $S^{j,k}$ inside $An(y,a,b)$ and the corresponding normal coordinates $(\xi,\sigma)$ on it, we conclude that for $l$ sufficiently large (consequently $\epsilon_l$ sufficiently small), $T \cap \Delta^{j,k,l}$ is the set $\{\sigma = g_{\epsilon_l}(\xi)\}$ for some function $g_{\epsilon_l}$. Moreover, $g_{\epsilon_l} \to 0$ smoothly, as $l \to \infty$. Now, using a patching argument entirely analogous to the one of the freezing construction in Lemma \ref{freezing lemma} (one dimensional version) allows us to modify $\Delta^{j,k,l}$ to coincide with $S^{j,k}$ (and therefore $\Gamma^{j,k}$) in some smaller annulus $An(y,b',b) \subset An(y,a,b)$, without increasing the area too much.
Thus, observing the definition of $S^{j,k}$ and \eqref{mass gap}, we are able to construct currents $\Delta^{j,k}$ with the following properties:
\begin{itemize}
\item $\Delta^{j,k}$ is smooth outside of a finite set;
\item $\Delta^{j,k} \res (\M \setminus B_{\rho}(y)) = \Gamma^{j,k} \res (\M \setminus B_{\rho}(y))$ ;
\item $\limsup_k \big(\mass(\Delta^{j,k}) - \mass(\Gamma^{j,k})\big) \leq -\eta <0$.
\end{itemize}
For $k$ large enough, Lemma \ref{replacements first step lemma} tells us that $\Delta^{j,k} \in \haus(\Gamma^j , An)$, which would in turn imply that $\Gamma^{j,k}$ is not a minimizing sequence, thus closing the contradiction argument.\\

{\bf Step 2. Minimality at the boundary.} We are still left with proving the statement in case $y \in \gamma \subset \partial \M$. As before, we start with a competitor current $S$ and the assumption \eqref{competitor current}. As a matter of fact, we will reduce this to the previous case by constructing the current $S^{j,k}$, "pushing" it slightly towards the interior of $\M$, and then "attaching" to it a smooth layer which connects it to $\gamma$. If the mass of the resulting current is very close to the mass of $S^{j,k}$, we retain \eqref{competitor current} with a smaller constant, and proceed with smoothing as before. First, analogously to the above, we obtain the currents $S^{j,k}$ and \eqref{mass gap}. Choose $(a,b)\subset\subset(\tau,\rho)$ such that $\Gamma^{j,k}\cap An(y,a,b)$ (and hence also $S^{j,k}$) is a smooth surface with boundary $\gamma \cap An(y,a,b)$. Parametrize a tubular neighborhood $U_{\delta}(\partial \M) = \{x \in \M : \, |\text{dist}(x,\partial M)|<\delta\}$ of $\partial \M$ with the usual smooth diffeomorphism
\begin{equation*}
\Phi: \partial \M \times [0,\delta) \to U_{\delta}(\partial \M), \quad (t,s) \mapsto \Phi(t,s) = \expo_{t}(s \nu(t)),
\end{equation*}
where $\nu(t)$ is the inward pointing normal of $\partial \M$ at $t$. Let us denote by $N:= \gamma \times [0,\delta)$ the smooth hypersurface which meets $\partial \M$ orthogonally in $\gamma$. Next, we pick $a<a'<b'<b$ and slightly deform $S^{j,k}$ to make it coincide with $N$ in $An(y,a',b')\cap U_{\xi}(\partial \M)$ for some $\xi$ small enough. To do this, note for example that near $\gamma$, $S^{j,k}\cap An(y,a',b')$ is a graph of a function $g$ over $N$, due to the convexity assumption on $\M$. By considering $g(1-\psi)$, where $\psi$ is a suitable cutoff function supported in $An(y,a,b)\cap U_{2\xi}(\partial \M)$ and equal to $1$ in $An(y,a',b')\cap U_{\xi}(\partial \M)$, we obtain the desired surface. Furthermore, its area will be arbitrarily close to the area of $S^{j,k}$, provided $\xi$ is chosen small enough. Thus, w.l.o.g. we can assume
\begin{equation}\label{orthogonal surface N}
S^{j,k}=N \text{ in } An(y,a',b')\cap U_{\xi}(\partial \M), \, \text{ for some } \xi \text{ small enough.}
\end{equation}

\noindent We fix: 
\begin{itemize}
\item a smooth function $\varphi : [0,\infty) \to [0,\epsilon]$ \, such that $\varphi(0) = \epsilon, \quad \varphi(x) =0$ for $x \geq \sqrt{\epsilon}$, and $|\varphi'(x)| \leq C \sqrt{\epsilon}$ \,(where $\epsilon$ will be fixed later);
\item a smooth function $\eta: \partial \M \to [0,1]$ \, such that $\eta(t)=1$ for $t \in \partial \M \cap B_{a'}(y)$ and $\eta(t)=0$ for $t \in \partial \M \setminus B_{b'}(y)$.
\end{itemize}
Consider now the map
\begin{equation}
\Psi(x):=\left\{ \begin{array}{ll}
(t,s) \mapsto (t, s + \varphi(s)\eta(t)) & \mbox{for $x=(t,s) \in U_{\sqrt{\epsilon}}(\partial \M)$};\\
\text{Id} & \mbox{for $x \in \M \setminus U_{\sqrt{\epsilon}}(\partial \M)$}.
\end{array} \right.
\end{equation}
If $\epsilon<\delta^2$ is small enough that $|\varphi'(x)|<1$, we ensure that $s \mapsto s + \varphi(s)\eta(t)$ is monotone increasing, and $\Psi: \M \to \M$ is a well defined, smooth, proper map, with a Lipschitz constant $1 + O(\sqrt{\epsilon})$. This means that we can push forward the current $S^{j,k}$ to obtain $\Psi_{\#}(S^{j,k})$ with a (possibly) small gain in mass, and with $\partial (\Psi_{\#}(S^{j,k})) =\Psi_{\#}(\partial S^{j,k}) = \Psi_{\#}(\gamma)$ being a smooth submanifold of $N$. It is now obvious that, by attaching to it a smooth surface $\gamma |_{\spt(\eta)} \times [0,\epsilon \eta(t))$ with mass $O(\epsilon)$ (and the proper orientation assigned), we are able to construct a current $\tilde{S}^{j,k}$ with $\partial \tilde{S}^{j,k} = \gamma$, $\tilde{S}^{j,k}\setminus B_{\rho}(y) = \Gamma^{j,k}\setminus B_{\rho}(y)$ and with $\mass(\tilde{S}^{j,k})$ arbitrarily close to $\mass(S^{j,k})$. Moreover, it follows from the construction and \eqref{orthogonal surface N} that $\tilde{S}^{j,k}$ is smooth in $U_{\epsilon}(\partial \M)\cap B_b(y)$ (in fact, it coincides with $N$ in $U_{\epsilon}(\partial \M)\cap B_{b'}(y)$).

\noindent We can now repeat the smoothing procedure from the previous case, centered around the point $y' = \Psi (y) \in \tilde{S}^{j,k}$, with one modification; we may not be able to actually choose (metric) balls around $y'$ with some radii $\tilde{a},\tilde{b}$, contained in Int($\M$) such that $\tilde{S}^{j,k}\cap An(y',\tilde{a},\tilde{b})$ is smooth, as before. Nevertheless, it follows from the above that we may choose some open neighborhoods $V_a(y') \subset \subset V_b(y') \subset \subset B_b(y)$ diffeomorphic to balls, such that this is true. All the arguments can be easily modified for this case, and we reach a contradiction once again.\\

{\bf Step 3. $\bm{Z^j = V^j}$.} We first show that $\mass(\Gamma^{j,k})$ converges to  $\mass(Z^j)$. Indeed, if this were not the case, we would have
\begin{equation*}
\mass(Z^j\cap B_{\rho/2})(y) < \limsup_{k \to \infty} \mass(\Gamma^{j,k} \cap B_{\rho/2})(y)
\end{equation*}
for some $y \in An$ and some $\rho$ to which we can apply the conclusion of Lemma \ref{replacements first step lemma}. We can then use $Z^j$  instead of $S$ in the beginning of this proof to once again contradict the minimality of the sequence $\{\Gamma^{j,k}\}_{k \in \NN}$. The convergence of the mass is then a simple consequence of the following well known fact

\begin{lemma}\label{currents-varifolds}
Let $ V^j $ be a sequence of rectifiable currents in $\M$ such that
\begin{enumerate}[(i)]
\item $V^j \rightharpoonup V$ in the flat norm;
\item $\mass(V^j) \to \mass (V)$.
\end{enumerate}
Let the rectifiable varifolds $W^j$ associated to $V^j$ converge to a (rectifiable) varifold $W$. Then $W$ is the varifold associated to the rectifiable current $V$.
\end{lemma}

{\bf Step 4. Regularity and stability.} In the constrained case the regularity in the interior follows from the standard theory for area-minimizing currents, see for instance \cite{simon-gmt}. The regularity at the boundary follows instead from \cite{allard-bdry} because $\partial \M$ is uniformly convex (actually \cite{allard-bdry} deals with the case where $\M$ is a subset of the Euclidean space, but the modifications to handle the case of a general Riemannian manifold are just routine ones). In the unconstrained case, the regularity is proved in Gr\"uter's work \cite{Gruter1} (here again, the arguments, given in the euclidean setting, can be easily adapted to deal with the general Riemannian one).

The minimality and stability of the surfaces $\bar\Gamma^j$ (together with the condition that they meet orthogonally $\partial \M$ in the unconstrained case) are obvious consequences of the minimality property. \qed

\subsection{Proof of Proposition \ref{existence of replacements}} We can now use the compactness theorems in Section \ref{s:stability} to show that the varifold $\tilde{V}$ has all the regularity properties required by Definition \ref{d:replacement}. In the unconstrained case we use Theorem \ref{t:SS-GY}, whereas in the constrained case we use Theorem \ref{t:SS-bdry}. Note that we can apply the latter theorem thanks to Lemma \ref{wedge property}. As for the remaining claims, the arguments are the same as in \cite{dl-t}. \qed

\section{Proof of Theorem \ref{t:Bdry_regularity} and Theorem \ref{t:main}}\label{s:bdry_reg}

Clearly Theorem \ref{t:main} is a direct consequence of Theorem \ref{t:Bdry_regularity} and Proposition \ref{a.m. in annuli}. Thus from now on we focus on Theorem \ref{t:Bdry_regularity}: we fix a varifold $V$ as in there and we want to prove that it is regular. In particular, we already know that $V$ is regular in the interior. Moreover,
\begin{itemize}
\item[(a)] In the constrained case we know that $\supp (V)\cap \partial \M \subset \gamma$. We thus need to show the regularity of $V$ at {\em any} point $p\in \gamma$, and more precisely that for every $p\in \gamma$ there is a neighborhood $U$ such that $V$ is a regular minimal surface $\Gamma$ in $U$ counted with multiplicity $1$, such that $\partial \Gamma = \gamma$ (in $U$).
\item[(b)] In the unconstrained case we need to show that, with the exception of a closed set of dimension at most $n-7$, for any $p\in \partial \M$ there is a neighborhood $U$ such that $V$ is a regular minimal surface $\Gamma$ in $U$ (counted with integer multiplicity, not necessarily $1$) which meets $\partial \M$ orthogonally. 
\end{itemize}

\subsection{Tangent varifolds and integrality}\label{s:vartan} We already know, in the constrained case, that $V$ is an integer rectifiable varifold and that $\|V\| (\partial \M) = 0$. In the unconstrained case we know the integrality of $V$ in ${\rm Int}\, (\M)$. We now wish to show that $\|V\| (\partial \M) =0$ even in this case. 

Fix a point $p\in \partial \M$ where $\Theta (V, p) > 0$ and consider the standard blow-up procedure of Lemma \ref{l:bu}. The upper bound on the density provided by Proposition \ref{p:GJ-monot} ensures that any sequence of rescaled varifolds $V_{p, r_k}$ have locally uniformly bounded mass and thus converges, up to subsequences, to some tangent varifold $W$: recall that ${\rm Tan}\, (V, p)$ denotes the set of such tangent varifolds. We fix one of them, say $W$, together with a converging sequence of rescalings $V_{p, r_k}$. $W$ is supported in $T_p \M$. Indeed the hyperplane $T_p \partial \M$ divides $T_p \M$ in two connected components $\pi^+$ and $\pi^-$ and $W$ is supported in the closure of one of them, say $\pi^+$.

\medskip

By Proposition \ref{p:GJ-monot} and standard arguments we conclude that
\begin{itemize}
\item[i)] $\delta W (\chi) =0$ for every smooth compactly supported vector field on $T_p \M$ tangent to $T_p \M$;
\item[(ii)] $\rho^{-n} \|W\| (B_\rho (0)) = \sigma^{-n} \|W\| (B_\sigma (0)) = \omega_m \Theta (V, p)$ for every $\sigma, \rho > 0$. 
\end{itemize}
Thanks to Proposition \ref{existence of replacements}, there exists a varifold $\tilde{V}_k$ which is a replacement for $V$ in the annulus $An (p, r_k, 2r_k)$. Rescaling such a replacement suitably we get a second varifold $\bar V_k$ which is a replacement for $V_{p, r_k}$ in $\iota_{p,r_k} (An (p, r_k, 2r_k))$. In particular, by the compactness Theorem \ref{t:SS-GY} (in the appropriately modified version discussed in Section \ref{s:changing_manifolds}) we obtain the convergence of $\bar V_k$ to a replacement $\bar W$ for $W$. Now, the latter replacement has the property that it is regular in $B_2 (0)\setminus \overline{B}_1 (0) \subset T_p \M$ and meets $T_p \partial \M$ orthogonally. Moreover, by the property of the replacement we must have
\[
\|\bar W\| (B_{1/2} (0)) =\|W\| (B_{1/2} (0)) \qquad \mbox{and} \qquad \|\bar W\| (B_{5/2} (0)) = \|W\| (B_{5/2} (0))\, .
\]
Using (ii) above and Lemma \ref{l:GJ-again} we conclude that
\begin{equation}\label{e:monot_const}
\sigma^{-n} \|\bar W\| (B_\sigma (0)) = \rho^{-n} \|\bar W\| (B_\rho (0)) = \omega_m \Theta (V, p) \qquad \forall \sigma, \rho > 0\, .
\end{equation}
In particular we must have that $\supp (\|\bar W\|) \cap \partial B_r (0) \neq \emptyset$ for every $r>0$: otherwise we would find an annulus $B_{r+\varepsilon} (0) \setminus B_{r-\varepsilon} (0)$ which does not intersect $\supp (\|\bar W\|)$, implying in turn that $\|\bar W\| (B_{r+\varepsilon} (0)) = \|\bar W\| B_{r-\varepsilon} (0)$, which would contradict \eqref{e:monot_const}. 

Fix now a point $q\in B_{3/2} (0) \cap \bar W$. Since $\bar W$ is regular in $B_2 (0) \setminus \overline{B}_1 (0)$, clearly $\Theta (\bar W, q) \geq \frac{1}{2}$. Using the monotonicity formula in $B_{1/2} (q)$ we then conclude that 
\[
\|\bar W\| (B_{1/2} (q)) \geq c (n) > 0\, ,
\]
where $c(n)$ is a geometric constant. In particular we have obtained a uniform lower bound for the density $\Theta (V, p)$, namely we have
\begin{equation}\label{e:uniform_LB}
\Theta (V, p) \geq c (n) > 0 \qquad \forall V\in \supp (\|V\|)\, .
\end{equation}
In turn, by standard arguments (cf. \cite[Chapter 8]{simon-gmt}: in place of the usual monotonicity formula we can use Proposition \ref{p:GJ-monot}), we conclude that, if $p\in \supp (\|V\|\cap \partial \M)$, any tangent varifold $W$ has the same uniform lower bound on the density. The classical Allard's rectifiability theorem implies that $V$ is then rectifiable ``in the interior'', namely in $\pi^+$ (cf. again \cite[Chapter 8]{simon-gmt}). On the other hand we also have that $\|V\|\res T_p \partial \M
= \Theta \mathcal{H}^n \res T_p \partial \M$ and we can use the same argument as in the proof of Corollary \ref{c:Gru-Jost-rect} in order to show that indeed $W$ is rectifiable {\em everywhere}. Finally, using the Gr\"uter-Jost monotonicity formula and arguing as in \cite[Chapter 8]{simon-gmt}, the property (ii) together with the rectifiability imply that $W$ is indeed a cone.

\medskip

Going back to the replacement $\bar W$, for the same reason we can argue that $\bar W$ is a cone and thus conclude
that $W$ itself is regular in the punctured plane $T_p \M \setminus \{0\}$. Moreover, by the considerations in \cite{Gruter1} and \cite{gru-jost}, the reflection of $W$ along $T_p \partial \M$ gives a stable minimal hypercone in $T_p \M \setminus \{0\}$, regular up to a set of codimension at least $7$. Finally, see for instance \cite{SS}, since the origin has zero $2$-capacity, such a cone turns out to be stable on the whole $T_p \M$. In particular, by the classical result of Simons, the cone is in fact a hyperplane if $n\leq 6$. 

Before going on, we observe that the argument above applies literally in the same way to the constrained case as well. We conclude that $W$ is a cone $C$ in $T_p \M \setminus \{0\}$ with the property that $\partial C = T_p \gamma$. In particular we conclude that $C$ is a multiplicity $1$ half-hyperplane and indeed, by the Wedge property of Lemma \ref{wedge property}, $C$ meets $T_p \partial \M$ transversally. 

We summarize our conclusions in the following

\begin{lemma}\label{tangent cones}
Let $V$ be as in Theorem \ref{t:Bdry_regularity}, $p$ a point in $\partial \M$ and $W\in {\rm Tan}\, (V, p)$ a tangent varifold.
\begin{itemize}
\item[(i)] In the constrained case $W=0$ unless $p\in \gamma$ and if $p\in \gamma$ then $W$ is a half hyperplane of $T_p\M$, counted with multiplicity $1$, which meets $T_p \partial \M$ transversally at $T_p \gamma$.
\item[(ii)] In the unconstrained case $W$ is a minimal hypersurface $\Xi$ meeting $T_p \partial \M$ orthogonally, which is half of a stable minimal cone in $T_p \M$ (counted with multiplicity), regular up to a set of dimension at most $n-7$. When $n\leq 6$, $\Xi$ is half of a hyperplane meeting $T_p \partial \M$ orthogonally. 
\end{itemize}
\end{lemma}

Next, in the unconstrained case the lemma above implies that $\|W\| (T_p \partial \M) =0$. In particular, since $\partial \M$ is a closed subset of $\M$, we easily conclude that
\[
\lim_{k\to\infty} r_k^{-n} \|V\| (\partial \M \cap \overline{B}_{r_k} (p))
= \lim_{r\downarrow 0} \|V_{p, r_k} \| (\iota_{p,r_k} (\partial \M \cap \overline{B}_{r_k} (p)) \leq
\|W\| (T_p \partial \M \cap \overline{B}_1) =0\, . 
\]
Therefore,
\[
\lim_{r\downarrow 0} r^{-n} \|V\| (\partial \M \cap B_r (p)) = 0\qquad \mbox{for every $p\in \partial M$,}
\]
which in turn implies easily $\|V\| (\partial \M) = 0$. 

In particular we have concluded that $V$ is integral even in the unconstrained case. 

\subsection{Regularity in the constrained case} In the constrained case, Lemma \ref{tangent cones} implies that we fall under the 
assumptions of Allard's boundary regularity theorem for stationary varifolds: $V$ is therefore regular at every point $p\in \gamma$, which completes the proof.

\subsection{Unconstrained case: regularity in the punctured ball} Our first goal is to show that, if $p\in \partial \M$, then there is a radius $r$ such that $V$ is regular {\em up to the boundary} in the punctured ball $B_\rho (p)\setminus \{p\}$ (except for a singular set of dimension at most $n-7$). 

\medskip

We first fix $p\in \partial\M$ and we then choose $r>0$ so that $V$ is a.m. in any annulus centered at $p$ and with outer radius smaller than $r$. Fix next a $\rho>0$ with $4\rho <r$ and recall that we do know that $V$ is regular in the interior
of $An (p, \rho, 4\rho)$, namely in $An (p, \rho, 4\rho)\setminus \partial \M$. Enumerate next the connected components 
$\Gamma_1, \ldots, \Gamma_i\, \ldots$ of $(An (p, \rho, 4\rho) \setminus \partial \M) \cap \supp (\|V\|)$ (which might be infinitely many). Fix any point $q\in \Gamma_i$ and consider a small ball $B_{\tilde{\sigma}} (q)\subset An (p, \rho, 4\rho)\setminus \partial \M$ so that $B_\sigma (q) \cap \Gamma_i$ is (diffeomorphic to) an $n$-dimensional ball for every $\sigma < \tilde\sigma$. Let $s$ be the distance between $p$ and $q$ and observe that, by the classical maximum principle, there is a positive $\sigma$ such that $\partial B_t (p)$ must intersect $\Gamma_i \cap B_\sigma (q)$ for every $t\in ]s, s+\sigma[$. Moreover, for a.e. $t$ such intersection must be transversal by Sard's Lemma. Let $t$ be any such radius, let $\tilde{q}\in \partial B_t \cap \Gamma_i \cap B_\sigma (q)$ and let $\gamma$ be a curve connecting $\tilde{q}$ and $q$ in $\Gamma_i \cap B_\sigma (q)$. Without loss of generality, by possibly changing the point $\tilde{q}$, we can assume that $\gamma$ is contained in $B_t (p)\setminus \overline{B}_{s-\sigma} (p)$ (except for the endpoint $\tilde{q}$). 

Consider now a replacement $V'$ for $V$ in the annulus $An (p, \rho, t)$. It turns out that $\tilde{q}\in \supp (\|V\|)$ necessarily. On the other hand $V'$ is a.m. in annuli and thus it is regular in the interior, namely in $\M \setminus  \partial \M$: more precisely it can have only singularities of codimension at most $7$ and at a singular point any tangent cone must be singular. This is certainly not the case for $\tilde{q}$ because ``outside of $B_t (q)$'' $V'$ is regular in a neighborhood of $\tilde{q}$ and meets $\partial B_t (q)$ transversally: in particular we know that any tangent cone of $V'$ at $\tilde{q}$ must contain half of an hyperplane. Since such tangent cone is stable and regular, except for a singular set of dimension $n-7$, we conclude that any tangent cone to $V'$ at $\tilde{q}$ is indeed the (same!) hyperplane. 

Hence $\tilde{q}$ is a regular point for $V'$ as well. Let now $\tilde{\Gamma}$ be the connected component of $\supp (\|V'\|)\cap (An (p, \rho, 4\rho)\setminus \partial \M)$ which contains $\tilde{q}$. By unique continuation, $\tilde{\Gamma}$ must in fact contain $\Gamma_i \cap B_\sigma (q)$. Again by unique continuation we conclude that the connected component $\tilde{\Gamma}$ and $\Gamma_i$ must coincide. On the other hand, because of the properties of the replacement, $\tilde{\Gamma}$ is regular up to the boundary in $An (p, \rho, t)$. 

\medskip

Since $q$ can be chosen arbitrarily close to $\sup \{d (p,p'): p' \in \Gamma_i$, we conclude that in fact $\Gamma_i$ is regular up to the boundary on the whole annulus $An (p, \rho, 4\rho)$. Now, by the monotonicity formula, we conclude immediately that, if $\Gamma_i$ contains a point in $An (p, 2\rho, 3\rho)$ (no matter whether such point is in the interior or in the boundary), then its $n$-dimensional volume is bounded from below by $c(n) \rho^n$: in particular there are only finitely many $\Gamma_i$'s which intersect $An (p, 2\rho, 3\rho)$.
For simplicity we will assume that they are the first $N_0$ in the chosen enumeration.

Recall that the singular sets $S_i := {\rm Sing}\, (\overline{\Gamma}_i\cap An (p, 2\rho, 3\rho))$ have dimension at most $n-7$. Consider a boundary point $q\in \partial \M \cap \overline{\Gamma}_i \cap An (p, 2\rho, 3\rho)$ which is regular for $\Gamma_i$ and at the same time does not belong to any other $S_j$. If $q$ were in the closure of some other $\Gamma_j$, then it would be a regular point for $\Gamma_j$ as well. Recall that $\Gamma_i$ and $\Gamma_j$ cannot cross in the interior and that they meet $\partial \M$ orthogonally. In particular we would necessarily have that $\Gamma_i$ and $\Gamma_j$ have the same tangent at $q$. However this would violate the maximum principle. 

\medskip

Consider now a point $p\in \supp (\|V\|)\cap An (p, 2\rho, 3\rho)\cap \partial \M$. Since $\|V\| (\partial \M)=0$ and since the $\Gamma_i$'s intersecting $An (p, 2\rho, 3\rho)$ are finitely many, we must necessarily have that $p\in \overline{\Gamma}_i$ for some $i$.

Summarizing we have concluded so far that
\[
\supp (\|V\|)\cap An (p, 2\rho, 3\rho) \subset \bigcup_{i=1}^{N_0} \overline{\Gamma}_i\, 
\]
and 
\[
\mbox{any point } q\in (An (p, 2\rho, 3\rho) \cap \overline{\Gamma}_i)\setminus \bigcup_{i=1}^{N_0} S_i 
\]
is a regular point for $V$.

Clearly, these two properties together imply that $V$ is regular in $An (p, 2\rho, 3\rho)$ (except for the usual closed set of dimension at most $n-7$). Since the argument is valid for any $\rho < \frac{r}{4}$, we easily conclude the regularity of $V$ in a punctured ball. 

\subsection{Unconstrained case: removing singular points for $n\leq 6$}\label{s:remove} From the previous step and by a simple covering argument,
we conclude that the set of singular points at the boundary is at most finite when $n\leq 6$. We now wish to remove said points. Again the argument is a suitable variant of the argument which deals with the same issue in the interior (cf. \cite{dl-t}). Consider the smooth surface $\Gamma$ (counted with multiplicity) which gives the varifold $V$ in $B_r (p)\setminus \{p\}$. If we choose $r$ sufficiently small, by Lemma \ref{tangent cones}, for every $\rho<r$ we know that the rescalings $\iota_{p, \rho} (\Gamma)$ are $\varepsilon$ close, in the varifold sense and in the annulus $\iota_{p,\rho} (An (p, \rho/8, 4\rho))$, to a varifold of the form $\Theta (V, p) \pi (\rho)$ where $\pi (\rho) \subset T_p \M$ is a half-hyperplane meeting $T_p \partial \M$ orthogonally. We can also assume that the tilt between $\pi (\rho)$ and $\pi (2\rho)$ is smaller than $\varepsilon$, provided $r$ is chosen even smaller. 

By the compactness Theorem \ref{t:SS-GY} (again, in the more general version where the ambient manifolds can change, cf. Section \ref{s:changing_manifolds}), if $r$ is sufficiently small and $\rho<r$, then $V_{p, \rho} \res \iota_{p, \rho} (An (p, \rho/4, 2\rho))$ consists of finitely many Lipschitz graphs $\Gamma_1 (\rho), \ldots , \Gamma_k (\rho)$ over $\pi (\rho)$, with controlled Lipschitz constant (say, at most $1$), each counted with multiplicity $m_i$. The same then holds for $V \res An (p, \rho/4, 2\rho)$. Moreover since the tilt between $\pi (\rho)$ and $\pi (\rho/2)$ is small, we easily conclude that the numbers of connected components in $An (p, \rho/8, \rho)$ is the same, that they can be ordered so that $\Gamma_i (\rho)$ and $\Gamma_i (\rho/2)$ overlap smoothly and that the corresponding multiplicities are the same.

We can repeat the above argument over dyadic radii $\rho 2^{-j}$ and we conclude that $V \res (B_\rho (p)\setminus \{p\}$ consists of finitely many connected components $\Gamma_i$ counted with multiplicity $m_i$, which are topologically punctured $n$-dimensional balls, smooth up to $\partial \M$. Taking one such connected component and removing the multiplicity, we get a multiplicity $1$ varifold in $B_\rho (p)$ which is stationary for the free boundary problem and has flat tangent cones at $p$, with multiplicity $1$. This falls therefore under the assumptions of the Allard's type theorem proved by Gr\"uter and Jost in the paper \cite{gru-jost-1}, from which we conclude that $p$ is a regular point. Hence each $\Gamma_i$ continues smoothly across $p$. The classical maximum principle now implies that the $\Gamma_i$ cannot actually touch at the point $p$, implying in fact that the number of connected components of $\Gamma$ in any ball $B_\rho$ is $1$.

\section{Competitors: proofs of Corollary \ref{c:main1} and \ref{c:main2}}\label{s:final}

We start with Corollary \ref{c:main1}.

\begin{proof}[Proof of Corollary \ref{c:main1}]
Without loss of generality we can assume that $\M$ is connected.

First of all we show that there is a generalized family $\{\Sigma_t\}_{t\in [0,1]}$ where $\Sigma_0$ and $\Sigma_1$ are trivial (namely as closed sets which consist of a collection of finitely many points). Indeed it suffices to take the level sets of a Morse function $f$ whose range is $[0,1]$, with the additional requirement that the restriction of $f$ to $\partial \M$ is also a Morse function. Since Morse functions are generic on smooth manifolds, the existence of such an $f$ is guaranteed. We then construct a homotopically closed family $X$ by taking the smallest such family which contains $\Sigma_t$. 

Take now any $\{\Sigma'_t\}_t \in X$. Away from the singularities $S_t$ the family $\{\Sigma'_t\}$ can be given locally and for $t$ in an interval $[a,b]$ as the image of a smooth map $\Phi: U \times [a,b]$. Thus the family $\{\Sigma'_\tau\}_{\tau\in [0,1]}$ induces canonically a current $\Omega'_t$ such that $\partial \Omega'_t = \Sigma'_t$. If $\{\Gamma_{t,s}\}_{(t,s)\in [0,1]^2}$ is a homotopy between $\{\Sigma_t\}_{t\in [0,1]}$ and $\{\Sigma'_t\}_{t\in [0,1]}$, it is easy to check that the corresponding currents $\Omega_{t,s}$ such that
$\partial \Omega_{t,s} = \Gamma_{t,s}$ also vary continuously. Observe however that:
\begin{itemize}
\item $\Omega_{t,0} = \llbracket \{f< t\} \rrbracket$ and thus $\Omega_{1,0} = \llbracket \M\rrbracket$, whereas $\Omega_{0,0} = 0$;
\item Since $\Gamma_{1,s}$, resp. $\Gamma_{0,s}$ are all trivial currents, each $\Omega_{1,s}$, resp. $\Omega_{0,s}$, is either $0$ or $\M$
(because we are assuming that $\M$ is connected);
\item The continuity of $\Omega_{1,t}$ and $\Omega_{0,t}$ ensures then that $\Omega'_1 = \Omega_{1,1} = \llbracket \M \rrbracket$ and
$\Omega'_0 = \Omega'_{0,1} = 0$. 
\end{itemize}
We thus conclude that there must be one $\Omega'_t$ such that $\mass (\Omega'_t) = \frac{1}{2} {\rm Vol} (\M)$. Now the isoperimetric inequality implies that $\haus^n (\Sigma'_t) \geq c_0 (\M) >0$, where the constant $c_0$ depends only upon the ambient manifold. 

The above argument shows that $m_0 (X) > bM_0 (X) = 0$ and thus we can apply Theorem \ref{t:main} to find a free boundary minimal hypersurface with total area equal to $m_0 (X)$. This completes the proof. 
\end{proof}

Similarly, Corollary \ref{c:main2} will be an immediate consequence of Theorem \ref{t:main} applied to constrained families, once we are able to show the existence of two strictly stable minimal surfaces gives a homotopically closed set $X$ of constrained families parametrized by $\P = [0,1]$ which satisfies the condition \eqref{mountain pass family}. The proof will be divided into two lemmas. In the first one we show the existence of a particular smooth family of hypersurfaces $\{\Sigma\}_{t \in [0,1]}$, starting from $\Sigma_0$ and ending in $\Sigma_1$. In the second lemma we show that any integer rectifiable current with sufficiently small flat distance to $\Sigma_0$ or $\Sigma_1$ must have mass which is strictly greater, with a uniform lower bound depending on the distance.
More precisely our two lemmas are

\begin{lemma}\label{sweepout}
Assume $\Sigma_0$ and $\Sigma_1$ are as in Corollary \ref{c:main2}. Then there exists a smooth family of hypersurfaces $\{\Sigma_t\}$ parametrized by $[0,1]$ which is constrained by $\gamma$.
\end{lemma}

\begin{lemma}\label{increase}
Let $\Sigma_0, \Sigma_1$ be as above. There exists an $\epsilon_0>0$ and $f:(0,\epsilon_0] \to \RR^+$ such that
\begin{align*}
 (S) \quad \begin{split}&\text{If }  \> \Gamma\> \text{ is an integer rectifiable current with } \> \fl(\llbracket \Sigma_i \rrbracket - \Gamma)=\epsilon, \>\>i \in \{0,1\}\>, \text{ and} \\
 &\partial \Gamma=\partial \llbracket \Sigma_i \rrbracket = \gamma, \>\text{ then } \> \mass(\Gamma) \geq \mass(\llbracket\Sigma_i\rrbracket) + f(\epsilon) \end{split}
\end{align*}
\end{lemma}

The two lemmas above easily imply our corollary.

\begin{proof}[{\bf \emph{Proof of Corollary \ref{c:main2}}}.] Obviously, by taking the homotopy class of the family in Lemma \ref{sweepout} we construct a homotopically closed set $X$. The second lemma then clearly implies that any smooth family $\{\Gamma_t\}$ with $\Gamma_0 = \Sigma_0$ and $\Gamma_1=\Sigma_1$ must satisfy \eqref{mountain pass family}, since $\fl (\Gamma_t, \Gamma_s)$ is a continuous function of $t$ and $s$ and $\fl (\Gamma_0, \Gamma_1) > 0$. In particular there is a smooth minimal surface $\Gamma$ with volume equal to $m_0 (X) > \max \{\haus^n (\Sigma_0), \haus^n (\Sigma_1)\}$ which bounds $\gamma$. 

Now, by Assumption \ref{a:stronger} the surface $\Gamma$ cannot be given by $\Sigma_0$ (or $\Sigma_1$) plus a closed minimal hypersurface, since the latter cannot exist. Recall moreover that the volume of $\Gamma$ must be strictly larger than $\Sigma_0$ (resp. $\Sigma_1$) and the multiplicity of $\Gamma$ must be everywhere $1$ thanks to part (b) of Theorem \ref{t:main}, we conclude that $\Gamma$ is distinct from $\Sigma_0$ (resp. $\Sigma_1$). In particular, 
if $\gamma$ is connected, then all the $\Sigma_i$'s must be connected and thus $\Sigma_2$ would give a third distinct minimal surface. 

In general, such argument would still be correct if one between $\Sigma_0$ and $\Sigma_1$ were connected. Otherwise, $\Sigma_2$ might be the union of some connected components of $\Sigma_0$ and of some connected components of $\Sigma_1$, arranged in such a way that $\partial \Sigma_2 = \gamma$ and that $\haus^n (\Sigma_2)
> \max \{\haus^n (\Sigma_0), \haus^n (\Sigma_1)\}$. In order to avoid such situation, we consider the family
$\mathcal{F}$ of minimal surfaces which bound $\gamma$ and which can be described as union of some connected components of $\Sigma_0$ and of some connected components of $\Sigma_1$. Since $\mathcal{F}$ consists of finitely many elements, we can pick one of maximal volume, which we denote by $\Gamma_0$. The remaining connected components of $\Sigma_0$ and of $\Sigma_1$ (namely those which are not connected components of $\Gamma_0$) form a second stable minimal surface $\Gamma_1$ which also bounds $\gamma$. If we now run the previous argument with $\Gamma_0$ and $\Gamma_1$ replacing $\Sigma_0$ and $\Sigma_1$, we achieve yet another minimal surface $\Gamma_2$: since $\Gamma_2$ must have volume strictly larger than that of $\Gamma_0$, the maximality of the latler in the class  $\mathcal{F}$ guarantees that $\Gamma_2$ has at least one connected component which is neither contained in $\Sigma_0$, nor in $\Sigma_1$. \qed

\subsection{Proof of Lemma \ref{sweepout}} {\bf{Step 1}}: Let us first extend $\M$ slightly across $\partial \M$, in order to make the following arguments more elegant (as per Remark \ref{extr vs. intr} we can even do this so that  $\M \subset \tilde{\M}$, for some closed manifold $\tilde{\M}$, if necessary). Consider the normal tubular neighborhood of $\gamma$ in $\M$, which is realized by an embedding $\iota : U \to \M$, where $U \subset N\gamma$ is a neighborhood of the zero section of the normal bundle $N\gamma$, such that $\iota|_{\gamma} = 1_{\gamma}$ and $\iota (U)$ is open in $\M$.
Take a (smooth) vector field $e_{1}(x)$ along $\gamma$, which is the normal to $\partial \M$ pointing inwards.

For each point $x$ of $\gamma$, consider the sets $\iota^{-1}(\Sigma_i\cap \iota(U))\cap U_x, \>\> i\in\{1,2\}$, where $U_x \cong \RR^2$ is the fiber of the normal bundle at $x$. Since $\Sigma_0$ and $\Sigma_1$ are smooth and minimal, we can use the same arguments as in the proof of Lemma \ref{wedge property} to conclude that, if $U$ is small enough, these are smooth, non-intersecting curves (starting at the origin) which are contained inside a $2$-dimensional wedge of opening angle at most $\theta < \frac{\pi}{2}$, with $e_1(x)$ lying on its axis. Hence, choosing $U$ even smaller if necessary, we can make sure that they are graphs over $e_1(x)$. That is, for each $U_x$ there exist (smooth) functions $\phi_{x}^0, \phi_{x}^1$ such that:
\begin{equation*}
\phi_{x}^0, \phi_{x}^1:W_{x} \rightarrow \RR, \qquad \phi_{x}^i (W_{x}) = \iota^{-1}(\Sigma_i\cap \iota(U))\cap U_x  \text{ for } i=0,1
\end{equation*}
where $W_{x}:= e_1(x) \cap U_{x}$.\vspace{2mm}

\noindent Recall that, by our assumption, $\Sigma_0\cup\Sigma_1$ bounds an open set $A$. Consider a point $y \in \bar{A}\cap\iota(U_{x})$, for some $x \in \gamma$. Note that the orthogonal projection of $\iota^{-1}(y)$ on $e_1(x)$, which we denote by $\bar{y}$, lies on the line segment $W_{x}$. We define:
\begin{equation}\label{eq1}
u_{x}(y):=t, \qquad \text{where}\quad \iota^{-1}(y)=t\phi_{x}^1(\bar{y}) + (1-t)\phi_{x}^0(\bar{y}), \>\> t \in[0,1].
\end{equation}
Now, by the properties of the tubular neighborhood, to each $y\in A\cap\iota(U)$ is associated a unique fiber $U_x$, hence there exists an $\eta >0$, such that when we are at most $\eta$-away from $\partial \M$, i.e. on some open set $E_0 = A \cap\iota(U) \cap (\M\backslash \M_{\eta})$ with $\M_{\eta} := \{x \in \M: \dist(x,\partial\M)\geq \eta\}$, these fiber-wise constructions yield a well defined function $f_0:E_0 \to \mathbb{R}$ such that,
\begin{equation*}
f_0(y):=u_{x}(y), \quad \text{where} \quad y \in \iota(U_{x}).
\end{equation*}

\noindent Furthermore, this function is smooth (by smoothness of $\Sigma_0, \Sigma_1$ and $\iota$), and it has no critical points, provided we choose $\eta$ small enough, since obviously the derivative in the direction orthogonal to $e_1$ (and $\gamma$) will be different from $0$. 
\vspace{2mm} \\
\indent Now, we construct a covering of $\bar{A}\setminus E_0$ with balls, satisfying the following two properties:
\begin{itemize}
\item [a)]Each ball has a radius less or equal than $\frac{\eta}{2}$;
\item [b)]Each ball can only contain points from one of the surfaces $\Sigma_0$ and $\Sigma_1$, and if it does, its center must lie on the surface.
\end{itemize}
Through compactness, we obtain a finite subcover, consisting of balls centered at the points $x_1,x_2,\ldots,x_N$. We will denote these balls by $E_{1},\ldots,E_{N}$.\\
Around each of these points $x_k$ lying on one of the $\Sigma_i$-s, we can characterize the submanifold through a local trivialization, i.e. there exists a neighborhood $W \subset \M$ of the point (which we can w.l.o.g. assume to be bigger than the ball $E_k$), an open set $W' \subset \mathbb{R}^{n+1} \cong \mathbb{R}^{n} \times \mathbb{R}^1$, and a diffeomorphism
\begin{equation*}
\Psi_{x_k}:W \to W', \quad \Psi_{x_k}(\Sigma_i \cap U) = W' \cap (\mathbb{R}^{n} \times \{0\}).
\end{equation*}
We assume in these cases that the points lying inside the set $A$ are mapped into the positive half-space $\mathbb{R}^{n+1}_+:= \{(y_1,\ldots,y_{n+1}), \, y_{n+1}>0\}$. We now define functions $f_i$ on the balls $E_i, \, i \in \{1,\ldots, N\}$ in the following way:
\begin{equation}\label{eq2}
f_i(y):= \left\{  \begin{array}{ll}
         \frac{1}{2} & \mbox{if $x_i$ lies in the interior of }A;\\
         (y_{n+1}\circ \Psi_{x_i})(y) & \mbox{if $x_i$ lies on $\Sigma_0$};
        \\ 1-(y_{n+1}\circ \Psi_{x_i})(y) & \mbox{if $x_i$ lies on $\Sigma_1$}. \end{array} \right.
\end{equation}
Here, $y_{n+1}:\mathbb{R}^{n+1} \to \mathbb{R}$ is just a function which evaluates the corresponding coordinate. Functions $f_i$ defined in this way are obviously smooth.

Finally, take a partition of unity $\{\varphi_j\}_{0 \leq j \leq N}$ of $\bar{A}$, subordinate to the covering $E_0,\ldots, E_N$. This allows us to define a function
$h:\bar{A} \to \mathbb{R}$ via:
\begin{equation}\label{eq3}
h(x):=\sum_{i=0}^{N} \varphi_i(x) f_i(x),
\end{equation}
which is smooth up to the boundary of $A$, excluding $\gamma$ of course.\\

{\bf Step 2}: {\it The function $h$ defined in \eqref{eq3} has no critical points near $\Sigma_0$ and $\Sigma_1$, as well as in a small neighborhood of $\gamma$. Moreover $h(\Sigma_0\setminus \gamma)=0$ and $h(\Sigma_1 \setminus \gamma)=1$.}\\

It is obvious from \eqref{eq1}, \eqref{eq2} and \eqref{eq3} that $h(\Sigma_0 \setminus \gamma)=0$ and $h(\Sigma_1 \setminus \gamma)=1$. Note also that, by the definitions of $E_i$, when we are at most $\frac{\eta}{2}$ away from $\partial \M$, only $\varphi_0$ is supported in this region, hence here it must hold $h(x) = f_0(x)$, and we already know that $f_0$ has no critical points in it.
In points $q \in \Sigma_0$, we have
\begin{equation*}
\left. \frac{\partial h}{\partial y_{n+1}} \right|_q = \sum_{i} \frac{\partial \varphi_i}{\partial y_{n+1}} \cdot f_i + \sum_i \varphi_i \cdot \frac{\partial f_i}{\partial y_{n+1}}
\end{equation*}
where $y_{n+1}$ again denotes the "height" with respect to some fixed local chart $E_j$. We have $f_i(q)=0 \>\> \forall i$ according to \eqref{eq1} and \eqref{eq2}, so the first sum vanishes. We also see that $\frac{\partial f_j}{\partial y_{n+1}}=1$ and $\frac{\partial f_i}{\partial y_{n+1}}>0$ for $i \neq j$, $i>0$ due to the compatibility of charts. so since $\varphi_i$-s are nonnegative and $\sum_i \varphi_i = 1$, it follows that the second sum is positive. Hence $q$ cannot be a critical point of $h$. With similar arguments, we deduce this also for points lying on $\Sigma_1$. 

It can be seen from the construction, however, that the function $h$ will be mostly constant inside the open set $A$ away from $\Sigma_0$ and $\Sigma_1$. So in this region we will use the fact that Morse functions form a dense, open subset in the $C^2$ topology, and define one such function $g$, say on the open set $B:=A \cap \text{Int}(\M_{\eta / 4})$ (recall the definition above), such that
\begin{equation}\label{morse} \|h-g\|_{C^2(B)} < \epsilon \end{equation}
for some small $\epsilon >0$, which will be fixed later. Next, we define a cut-off function $\psi:\M\to \RR$, such that $\psi = 0$ on $(\M \setminus \M_{\eta/4})\cup W$ and $\psi=1$ on $\M_{\eta /2}\setminus W'$, where $W\subset \subset W'$ are sufficiently small neighborhoods of $\Sigma_0 \cup \Sigma_1$. We finally define:
\begin{equation}\label{eq4}
f:\bar{A} \to \RR, \qquad f(x):=\psi(x) g(x) + (1-\psi(x))h(x)
\end{equation}\\

{\bf Step 3}: {\it For $\epsilon$ small enough, the function $f$ is Morse inside $A$, and its level sets provide a smooth family parametrized by $[0,1]$, where $f^{-1}(1) =\Sigma_1$ and $f^{-1}(0)=\Sigma_0$.}\\

It follows from the construction that $f$ does not have any degenerate critical point in the regions where $\psi=0$ or $\psi=1$. In the intermediate region, due to \eqref{morse} we have:
\begin{equation*}
Df = Dh + D\psi(g-h) + \psi(Dg-Dh)
\end{equation*}
Due to the previous steps, we know $h = f_0$ when at most $\frac{\eta}{2}$ away from $\partial \M$, and thus we have no critical points for $h$ close to $\partial \M$. Thus $|Dh| > \delta$ for some $\delta > 0$ on $\mathcal{M}\setminus \M_{\frac{\eta}{2}}$. Hence we have
\begin{equation*}
|Df| \geq |Dh| - (|D\psi| + |\psi|) \|h-g\|_{C^2} \geq \delta - C\epsilon,
\end{equation*}
for some constant $C$ depending on $\psi$ (which in turn depends only on $\eta$) on $\mathcal{M}\setminus \mathcal{M}_{\frac{\eta}{2}} \cup W'$. Now, we can fix $\epsilon$ small enough so that $|Df| >0$ on $\mathcal{M}\setminus \mathcal{M}_{\frac{\eta}{2}} \cup W'$. On the other hand, since $f=g$ on $\mathcal{M}_{\frac{\eta}{2}}$, we conclude that $f$ is a Morse function. It is clear from the construction that the level sets of $f$ will be smooth hypersurfaces near $\gamma$ and will in fact have $\gamma$ as boundary. \qed

\subsection{Proof of Lemma \ref{increase}}
The proof uses heavily Brian White's similar result in \cite{white-strong} where he proves that $\llbracket\Sigma_i\rrbracket$ is the unique minimizer among all currents in the same homology class whose support is contained in a sufficiently small neighborhood of $\Sigma_i$ (note that, actually, it follows from standard arguments that, if the neighborhood is sufficiently small, any current with the same boundary as $\Sigma_i$ must be in its holomogy class).
In this lemma we just need to replace the assumption of being close in the $L^\infty$ sense to the one of being close in the flat norm. Thus our lemma is
indeed very close to \cite[Lemma 4.1]{AFM}. 

W.l.o.g. we assume $i=0$. By Theorem 2 of White \cite{white-strong}, there exists an open set $U$ containing $\Sigma_0$ such that
\begin{equation}\label{strong minimax thm}
\mass(\Gamma) > \mass( \llbracket \Sigma_0\rrbracket) \quad \forall \Gamma \text{ with } \partial \Gamma = \partial \llbracket \Sigma_0 \rrbracket, \text{ and }\spt(\Gamma) \subset U.
\end{equation}

We define
\begin{equation}\label{minimizing in flat distance}
m_0(\epsilon): = \inf \{ \mass(\Gamma) : \> \partial \Gamma = \partial \llbracket \Sigma_0 \rrbracket \text{ and } \fl(\Gamma - \llbracket \Sigma_0 \rrbracket) = \epsilon \}.
\end{equation}
Our aim is to show that $m_0(\epsilon) > \mass(\llbracket \Sigma_0 \rrbracket) \>\> \forall \epsilon \in (0,\epsilon_0]$, which clearly implies the statement (S), by setting $f(\epsilon)=m_0(\epsilon) - \mass(\llbracket \Sigma_0 \rrbracket)$. Note that
the infimum in \eqref{minimizing in flat distance} is actually a minimum. We would like to show that, if $\varepsilon$ sufficiently small, a minimizer $\Gamma_\varepsilon$ must be contained in the tubular neighborhood $U$ of $\Sigma_0$: this would then conclude the proof
because by \eqref{strong minimax thm} the mass of $\Gamma_\varepsilon$ would be strictly larger than that of $\llbracket \Sigma_0\rrbracket$. 
In fact, what we will really show is that there is certainly a $Z$ which has at most the same mass as $\Gamma_\varepsilon$, has boundary $\gamma$ and it is contained in $U$, which still suffices to reach the desired conclusion.

\medskip

{\bf{Step 1}}: Let us denote first extend $\Sigma_0$ slightly outside $\partial \M$ to a $\Sigma'_0$ (remember that we can embed $\M$ in a smooth closed manifold $\tilde{\mathcal{M}}$) and denote by $U_{\delta}$ the $\delta$-tubular neighborhood of $\Sigma'_0$ intersected with $\M$. We will choose $\delta$ small enough so that $U_{2\delta}\subset U$. Note that $\partial U_{\tau}$ is smooth for all $\tau \in (\delta,2\delta)$ and diffeomorphic to two copies of $\Sigma_0$, with diffeomorphisms whose smoothness can be bounded independently of $\tau$. Hence, by the isoperimetric inequality, we can choose some constant $C>0$ (independent of $\tau$) such that for every $(n-1)$-dimensional integer rectifiable current $\alpha$ homologous to $0$ in $\partial U_{\tau}$, there exists an $n$-dimensional integer rectifiable current $S$ in $\partial U_{\tau}$ with
\begin{equation}\label{isoperimetric}
\partial S = \alpha \quad \text{and} \quad \mass(S) \leq C \mass(\alpha)^{\frac{n}{n-1}}.
\end{equation}
Take $\Gamma_{\epsilon}$ to be the minimizer in \eqref{minimizing in flat distance}. For every $\tau \in (\delta,2\delta)$ we define:
\begin{itemize}
\item $A(\tau):= \mass\big(\Gamma_{\epsilon} \res (U_{\tau})^c\big);$\\
\item $L(\tau):=\mass\big(\partial\big(\Gamma_{\epsilon} \res (U_{\tau})^c \big)\big)=\mass\big(\partial(\Gamma_{\epsilon}\res U_{\tau})- \gamma\big)$.
\end{itemize}
A standard inequality using coarea formula yields
\begin{equation}\label{coarea-step1}
L(\tau) \leq - A'(\tau) \qquad \mbox{for a.e. $\tau$.}
\end{equation}
Let us now fix $\tau \in (\delta,2\delta)$. One of the following alternatives must hold:
\begin{itemize}
\item[(A1)] $L(\tau)=0$. This means that $\partial \big( \Gamma_{\epsilon} \res U_{\tau}\big)=\gamma$, and hence $ \Gamma_{\epsilon} \res U_{\tau}$ is homologous to $\Sigma_0$ in $U$. Consequently, by \eqref{strong minimax thm},
\begin{equation*}
m_0(\epsilon) = \mass(\Gamma_{\epsilon}) \geq \mass(\Gamma_{\epsilon}\res U_{\tau}) > \mass (\llbracket \Sigma_0 \rrbracket),
\end{equation*}
hence we are finished.
\item[(A2)]$L(\tau) > 0$. Since $\fl(\Gamma_{\epsilon} - \llbracket \Sigma_0 \rrbracket)$ is sufficiently small, then $\fl(\Gamma_{\epsilon} - \llbracket \Sigma_0 \rrbracket)=\mass(T)$ with $\partial T = \Gamma_{\epsilon} - \llbracket \Sigma_0 \rrbracket$. Note that the slice of the $(n+1)$-current $T$, which is supported in $\partial U_{\tau}$, bounds the slice of the $n$-current $\Gamma_{\epsilon} - \llbracket \Sigma_0\rrbracket$, which in fact coincides with the slice of $\Gamma_\epsilon$ because $\Sigma_0\cap \partial U_\tau =0$. Let us denote the slice of $T$ by $S$. This means that $S$ lies in $\partial U_{\tau}(\Sigma_0)$ with $\partial S = \gamma - \partial ( \Gamma_{\epsilon} \res U_{\tau})$, and by \eqref{isoperimetric},
\begin{equation*}
\mass(S) \leq C L(\tau)^{\frac{n}{n-1}}
\end{equation*}
Let us set $Z=\Gamma_{\epsilon} \res U_{\tau} + S$. At this point, we make a further distinction between two cases:
\begin{itemize}
\item[(A2.1)] $\mass(Z) \leq \mass(\Gamma_{\epsilon})$. By construction, $Z$ is homologous to $\Sigma_0$ in $U$; thus by \eqref{strong minimax thm},
\begin{equation*}
m_0(\epsilon) = \mass(\Gamma_{\epsilon}) \geq \mass(Z) > \mass(\llbracket \Sigma_0 \rrbracket),
\end{equation*} 
and the claim follows.
\item[(A2.2)] $\mass(Z) \geq \mass(\Gamma_{\epsilon})$. By the above, this implies
\begin{equation*}
\mass \big(\Gamma_{\epsilon} \res (U_{\tau})^c\big) \leq \mass(S) \leq C L(\tau)^{\frac{n}{n-1}}.
\end{equation*}
\end{itemize}
\end{itemize}
In summary, it follows from the considerations above that for the rest of the proof we may assume w.l.o.g. the following properties for a.e. $\tau \in (\delta,2\delta)$:
\begin{itemize}
\item $L(\tau)>0$  
\item $A(\tau) \leq C L(\tau)^{\frac{n}{n-1}}$.\\
\end{itemize}

{\bf{Step 2}}: We claim that the minimizers $\Gamma_{\epsilon}$ satisfy
\begin{equation}\label{step2claim}
\mass(\Gamma_{\epsilon}) \to \mass(\llbracket \Sigma_0 \rrbracket) \quad \text{as } \epsilon \to 0.
\end{equation}
By the lower semicontinuity of mass with respect to flat convergence, we immediately get
\begin{equation*}
\liminf_{\epsilon \to 0} \mass(\Gamma_{\epsilon}) \geq \mass (\llbracket \Sigma_0 \rrbracket).
\end{equation*}
The other inequality needed to prove the claim follows by constructing suitable competitors. Consider the currents $\tilde{\Gamma}_r:= \llbracket \Sigma_0 \rrbracket + \partial \llbracket B_r(p) \rrbracket$, where $B_r(p) \subset U$, $B_r(p) \cap \Sigma_0 = \emptyset$. Clearly, $\fl(\tilde{\Gamma}_r - \llbracket \Sigma_0 \rrbracket) \to 0$ as $r \to 0$, hence (for $\epsilon$ small enough) there exists some $r(\epsilon)$ such that $\fl(\tilde{\Gamma}_{r(\epsilon)} - \llbracket \Sigma_0 \rrbracket) = \epsilon$. Moreover, $\mass(\tilde{\Gamma}_{r(\epsilon)}) \to \mass(\llbracket \Sigma_0 \rrbracket)$ as $\epsilon \to 0$. This shows that
\begin{equation*}
\limsup_{\epsilon \to 0} \mass(\Gamma_{\epsilon}) \leq \mass (\llbracket \Sigma_0 \rrbracket),
\end{equation*}
and the claim follows.\\

{\bf{Step 3}}: We next prove that
\begin{equation}\label{step3claim}
\lim_{\epsilon \to 0} \mass\big(\Gamma_{\epsilon} \res (U_{3\delta/2})^c \big)=0.
\end{equation}
As before, we can assume $\Gamma_{\epsilon} - \llbracket \Sigma_0 \rrbracket = \partial T_{\epsilon}$, with $\mass(T_{\epsilon}) \to 0$ as $\epsilon \to 0$. If \eqref{step3claim} were wrong, there would exist a sequence $\epsilon_k \downarrow 0$ and an $\alpha > 0$ such that
\begin{equation}\label{step3contradiction}
\mass\big(\Gamma_{\epsilon_k} \res (U_{3\delta/2})^c \big) \geq \alpha.
\end{equation}
If we let $\langle T_{\epsilon_k}, \tau \rangle = \partial (T_{\epsilon_k} \res U_{\tau}) - (\partial T_{\epsilon_k})\res U_{\tau}$ denote the slices of $T_{\epsilon_k}$ (w.r.t the distance from $\Sigma_0$), then by coarea formula
\begin{equation*}
\int_{\delta}^{\frac{3}{2}\delta} \mass ( \langle T_{\epsilon_k}, \tau \rangle) \, d\tau \leq \mass(T_{\epsilon_k}) \to 0 
\end{equation*}
as $k \to 0$. Since $L^1$ convergence implies a.e. pointwise convergence, we are able to extract a subsequence (not relabeled) and a $\tau \in (\delta, \frac{3}{2}\delta)$ such that $\mass ( \langle T_{\epsilon_k}, \tau \rangle) \to 0$. On the other hand, we can apply \eqref{strong minimax thm} to the current $ \langle T_{\epsilon_k}, \tau \rangle + \Gamma_{\epsilon_k} \res U_{\tau}$ as we did in Step 1, which gives us
\begin{equation*}
\mass \big( \langle T_{\epsilon_k}, \tau \rangle + \Gamma_{\epsilon_k} \res U_{\tau}\big)> \mass(\llbracket \Sigma_0 \rrbracket).
\end{equation*}
Using these two facts together with \eqref{step2claim}, one easily concludes $\displaystyle\lim_{k \to \infty}\mass\big(\Gamma_{\epsilon_k} \res (U_{\tau})^c\big) \to 0$, which is a contradiction to \eqref{step3contradiction}.\\

{\bf{Step 4}}: Note that the previous step tells us that $A\left(\frac{3}{2}\delta\right) \to 0$ as $\epsilon \to 0$. Recall that we assume $L(\tau)>0 \>\> \forall \tau \in (\delta,2\delta)$ since Step 1, which immediately implies that also $A(\tau)>0$. However, from \eqref{coarea-step1} and the other assumption of Step 1 we deduce
\begin{equation*}
A(\tau) \leq C L(\tau)^{\frac{n}{n-1}} \leq C (- A'(\tau))^{\frac{n}{n-1}},
\end{equation*}
giving (by a slight abuse of notation regarding the constants involved)
\begin{equation*}
-\frac{A'(\tau)}{A(\tau)^{\frac{n-1}{n}}} \geq \frac{1}{C} \quad \forall \tau \in (\delta,2\delta).
\end{equation*}
Integrating the above inequality between $\frac{3}{2}\delta$ and $2\delta$, we get
\begin{equation*}
A\left(\tfrac{3}{2}\delta\right)^{\frac{1}{n}} \geq A\left(\tfrac{3}{2}\delta\right)^{\frac{1}{n}} - A(2\delta)^{\frac{1}{n}} \geq \frac{\delta}{2nC}
\end{equation*}
which gives a contradiction for $\epsilon$ small enough.
\end{proof}

\

\end{document}